\documentclass{article}
\usepackage[utf8]{inputenc}
\usepackage{a4wide}
\usepackage{authblk}
\usepackage[normalem]{ulem}

\usepackage[utf8]{inputenc}
\usepackage[T1]{fontenc}
\usepackage[british]{babel} 
\usepackage{csquotes}

\usepackage{amssymb}
\usepackage{mathtools}
\mathtoolsset{showonlyrefs}						

\usepackage[scr=boondoxo]{mathalpha}
\usepackage{bm}
\usepackage{bbm}
\usepackage[only,llbracket,rrbracket]{stmaryrd} 

\usepackage{amsthm}
\usepackage{thmtools}							

\usepackage[a4paper,left=2.5cm,top=3cm,right=3cm,bottom=4cm]{geometry}
\usepackage{lmodern}							
\usepackage{setspace}			
\usepackage{soul}								

\allowdisplaybreaks

\usepackage[protrusion=true, 					
expansion=true, 								
final, 											
activate={true,nocompatibility}, 				
babel=true, 									
kerning=true, 									
spacing=true									
]{microtype}									

\usepackage{enumitem}

\usepackage{pdfpages}							

\usepackage[dvipsnames]{xcolor}					
\usepackage{graphicx}
\usepackage{tikz}								
\usetikzlibrary{decorations.markings}			
\usepackage{tikz-3dplot}
\usepackage{pgfplots}
\usepgfplotslibrary{fillbetween, colormaps} 
\pgfplotsset{compat=newest}
\usepackage{standalone}

\usepackage[
]{caption} 										
\numberwithin{figure}{section}

\usepackage{subcaption} 						

\usepackage[
style=ext-alphabetic,
backend=biber,
maxnames=3, 
maxcitenames=3, 
isbn=false,
date=year,
url=false,
isbn=false,
doi=false,
eprint=false,
natbib=true,
giveninits=true 
]{biblatex}

\usepackage{hyperref}
\definecolor{mntf}{RGB}{0,101,97}
\hypersetup{
	allcolors=mntf,
	colorlinks=true
}

\declaretheorem[numberwithin=section, name=Theorem, refname={Theorem,Theorems}]{theorem}
\declaretheorem[sibling=theorem, name=Lemma, refname={Lemma,Lemmas}]{lemma}
\declaretheorem[sibling=theorem, name=Corollary, refname={Corollary,Corollaries}]{cor}
\declaretheorem[sibling=theorem, name=Definition, refname={Definition,Definitions}]{definition}

\theoremstyle{remark}
\declaretheorem[sibling=theorem, name=Remark, refname={Remark,Remarks}]{rem}

\numberwithin{equation}{section}

\makeatletter

\@namedef{theHtheorem}{thm.\thetheorem}
\@namedef{theHlemma}{lem.\thetheorem}
\@namedef{theHcor}{cor.\thetheorem}
\@namedef{theHdefinition}{def.\thetheorem}
\@namedef{theHproposition}{prop.\thetheorem}
\@namedef{theHrem}{rem.\thetheorem}
\makeatother

\addto\extrasbritish{%
	
	}			

\newcommand{\eps}{\varepsilon}
\renewcommand{\phi}{\varphi}
\newcommand{\p}{\partial}
\newcommand{\Ome}{\Omega}
\newcommand{\dd}{\mathop{}\!\mathrm{d}}								
\newcommand{\per}{{\mathrm{per}}}									
\newcommand{\dist}{\mathrm{dist}}									
\newcommand{\DD}{\mathrm{D}}										
\newcommand{\Cci}{C^\infty_\mathrm{c}}										

\newcommand{\set}[2]{\left\{#1 \,\colon\, #2\right\}}				
\newcommand{\interior}[1]{\operatorname{int}\left(#1\right)}
\newcommand{\restr}{\,\raisebox{-.127ex}{\reflectbox{\rotatebox[origin=br]{-90}{$\lnot$}}}\,} 

\DeclarePairedDelimiter{\norm}{\lVert}{\rVert}

\DeclarePairedDelimiter\floor{\lfloor}{\rfloor}						
\newcommand{\scp}[3]{\left(#1,#2\right)_{#3} }						
\newcommand{\dotproduct}[3]{\left\langle #1,#2\right\rangle_{#3}}	

\newcommand{\norms}[2]{\lVert #1 \rVert_{\smash{#2}\raisebox{-0.2ex}{\mathstrut}}} 
\newcommand{\Lpnorm}[3]{\norms{#1}{L^{#2}(#3)}}					
\newcommand{\Lptnorm}[3]{\lvert\kern-0.15ex\lvert\kern-0.15ex\lvert #1
	\rvert\kern-0.15ex\rvert\kern-0.15ex\rvert_{\smash{L_t^{#2}(#3)}\raisebox{-0.2ex}{\mathstrut}}}		
\newcommand{\Cknorm}[3]{\norms{#1}{C^{#2}(#3)}}					

  \def\calF{{\mathcal F}}
 \def\calH{{\mathcal H}} \def\calI{{\mathcal I}}
  \def\calL{{\mathcal L}}
 \def\calN{{\mathcal N}} 
 \def\calQ{{\mathcal Q}} 
 \def\calT{{\mathcal T}} \def\calU{{\mathcal U}}

 \def\IN{{\mathbbm N}} 
  \def\IR{{\mathbbm R}}

 \def\IZ{{\mathbbm Z}}

\newcommand{\Omee}{\Ome_\eps}
\newcommand{\Omeh}{\widehat{\Ome}_{\eps,h}}
\newcommand{\OmeM}{\Ome^{\mathrm{M}}_\eps}
\newcommand{\OmeMS}{\Ome^{\mathrm{M}}_{*,\eps}}
\newcommand{\OmeMSh}{\widehat{\Ome}^{\mathrm{M}}_{*,\eps,h}}

\newcommand{\QTe}{\calQ_{T,\eps}}
\newcommand{\QTeM}{\calQ_{T,\eps,*}^{\mathrm{M}}}

\newcommand{\QTM}{\calQ_{T,*}^{\mathrm{M}}}

\newcommand{\tD}{\tilde{D}}

\newcommand{\tq}{\tilde{q}}
\newcommand{\tb}{\tilde{b}}
\newcommand{\tu}{\tilde{u}}
\newcommand{\tx}{\tilde{x}}
\newcommand{\delt}{\delta_{l^\prime,\eps}}

\newcommand{\sumpm}{\mathop{\smash{\sum_\pm}}\displaylimits}		

\renewcommand{\div}{\mathop{\mathrm{div}}\nolimits}
\title{Effective transmission through an interface\\with evolving microstructure}

\author[1]{Lucas M.~Fix}
\author[1]{Gianna G\"{o}tzmann}
\author[1,2]{Malte A.~Peter}
\author[1,2]{Jan-F.~Pietschmann}
\affil[1]{{\small University of Augsburg, Institute of Mathematics, Universit\"{a}tsstra\ss e 12a, 86159 Augsburg, Germany.}}
\affil[2]{{\small Centre for Advanced Analytics and Predictive Sciences (CAAPS), Universit\"{a}tsstra\ss e 12a, 86159 Augsburg, Germany.}}

\date{\today}

\usepackage[refpage, intoc]{nomencl}

\usepackage{ifthen}

\makenomenclature
\RequirePackage{calc}
\renewcommand{\nomgroup}[1]{%
     \item[\bfseries
     \ifthenelse{\equal{#1}{F}}{Function spaces}{%
     \ifthenelse{\equal{#1}{G}}{Microscopic and macroscopic geometry}{%
     }}]%
 	\vspace{0.5\baselineskip}
   }

\makeatletter
\newcommand{\definesymbol}[5][]{%
	\expandafter\def\csname sym@#2@cat\endcsname{#3}%
	\expandafter\def\csname sym@#2@sym\endcsname{#4}%
	\expandafter\def\csname sym@#2@desc\endcsname{#5}%
	\if\relax#1\relax
	\else
	\expandafter\def\csname sym@#2@sort\endcsname{#1}%
	\fi
}
\makeatother

\newcommand{\usesym}[1]{%
	\begingroup
	\edef\x{\endgroup
		\noexpand\nomenclature[\ifcsname sym@#1@sort\endcsname
		\csname sym@#1@sort\endcsname
		\else
		\csname sym@#1@cat\endcsname
		\fi]{%
			\csname sym@#1@sym\endcsname%
		}{%
			\csname sym@#1@desc\endcsname%
		}%
	}\x
}

\definesymbol[GZ01]{Z}{G}{\detokenize{$Z$}}{$n$-dimensional standard cell}
\definesymbol[GZ02]{Z*}{G}{\detokenize{$Z_*$}}{reference channel in $Z$}
\definesymbol[GZ03]{Zstx}{G}{\detokenize{$Z_*(t,x^\prime)$}}{spacetime-evolving reference channel for $(t,x')\in (0,T)\times \Sigma$}

\definesymbol[GN01]{N}{G}{\detokenize{$N$}}{lateral boundary of $Z_*$}
\definesymbol[GN02]{Nstx}{G}{\detokenize{$N(t,x^\prime)$}}{lateral boundary of $Z_*(t,x^\prime)$ for $(t,x')\in (0,T)\times \Sigma$}
\definesymbol[GN03]{Neps}{G}{\detokenize{$N_\eps$}}{lateral boundaries of all microscopic channels in $\detokenize{\OmeMS}$}
\definesymbol[GN04]{Neh}{G}{\detokenize{$\widehat{N}_{\eps,h}$}}{lateral boundaries of all microscopic channels in $\detokenize{\OmeMSh}$}
\definesymbol[GN05]{NTe}{G}{\detokenize{$\mathcal{N}_{T,\eps}$}}{time-evolving lateral boundaries of all microscopic channels in $\detokenize{\QTeM}$}
\definesymbol[GN06]{NT}{G}{\detokenize{$\mathcal{N}_{T}$}}{spacetime-evolving lateral boundaries $N(t,x')$, $(t,x')\in (0,T)\times \Sigma$}

\definesymbol[GS01]{Spm}{G}{\detokenize{$S^\pm$}}{top and bottom of the reference cell $Z$}
\definesymbol[GS02]{Sspm}{G}{\detokenize{$S_{*}^\pm$}}{top and bottom of the reference channel $Z_*$}
\definesymbol[GS03]{Sepspm}{G}{\detokenize{$S_{\eps}^\pm$}}{top and bottom of the thin layer \detokenize{$\OmeM$}}
\definesymbol[GS04]{Ssepspm}{G}{\detokenize{$S_{*,\eps}^\pm$}}{interface between the channel domain $\detokenize{\OmeMS}$ and the bulk regions $\detokenize{\Omee^\pm}$}
\definesymbol[GS05]{Ssehpm}{G}{\detokenize{$\widehat{S}_{\eps,*,h}^\pm$}}{interface between the channel domain $\detokenize{\OmeMSh}$ and the bulk regions $\detokenize{\Omeh^\pm}$}

\definesymbol[GO01]{Omepm}{G}{\detokenize{$\Ome^\pm$}}{macroscopic bulk domains}
\definesymbol[GO02]{Omee}{G}{\detokenize{$\Omee$}}{microscopic reference domain}
\definesymbol[GO03]{Omeepm}{G}{\detokenize{$\Omee^\pm$}}{microscopic bulk domains}
\definesymbol[GO04]{OmeM}{G}{\detokenize{$\OmeM$}}{microscopic thin layer}
\definesymbol[GO05]{OmeMS}{G}{\detokenize{$\OmeMS$}}{microscopic channel domain}
\definesymbol[GO06]{Omeeh}{G}{\detokenize{$\Ome_{\eps,h}$}}{subset of points in $\Omee$ having at least distance $h$ from $\detokenize{\p \Sigma \times (-H,H)}$}
\definesymbol[GO07]{Omeehpm}{G}{\detokenize{$\Ome_{\eps,h}^\pm$}}{microscopic bulk regions in $\Ome_{\eps,h}$}
\definesymbol[GO08]{hatOmeeh}{G}{\detokenize{$\Omeh$}}{Lipschitz regular subset of $\Ome_{\eps,h}$}
\definesymbol[GO09]{Omehpm}{G}{\detokenize{$\Omeh^\pm$}}{microscopic bulk regions in $\detokenize{\Omeh}$}
\definesymbol[GO10]{OmeMSh}{G}{\detokenize{$\OmeMSh$}}{microscopic channel domain in $\detokenize{\Omeh}$}
\definesymbol[GO11]{QTe}{G}{\detokenize{$\QTe$}}{time-evolving microscopic domain}
\definesymbol[GO12]{QTM}{G}{\detokenize{$\QTM$}}{spacetime-evolving reference cells $Z_*(t,x')$, $(t,x')\in (0,T)\times \Sigma$}
\definesymbol[GO13]{QTeM}{G}{\detokenize{$\QTeM$}}{time-evolving channel domain}
\definesymbol[GOS01]{Sigma}{G}{\detokenize{$\Sigma$}}{$(n-1)$-dimensional hypersurface separating $\Ome^+$ and $\Ome^-$}
\definesymbol[GOS02]{Sigmah}{G}{\detokenize{$\Sigma_h$}}{subset of points in $\Sigma$ having at least distance $h$ from $\detokenize{\p \Sigma}$}
\definesymbol[GOS03]{Sigmaeh}{G}{\detokenize{$\widehat{\Sigma}_{\eps,h}$}}{\detokenize{Lipschitz regular subset of $\Sigma_h$}}{}

\definesymbol[FA01]{C0Zsper}{F}{\detokenize{$C^0_\#(\overline{Z_*})$}}{\detokenize{space of continuous $(Y,0)$--periodic functions on $\IR^{n-1}\times [-1,1]$, restricted to $\overline{Z_*}$}}
\definesymbol[FA02]{H0Gamma}{F}{\detokenize{$H^1_{\Gamma,0}(U)$}}{\detokenize{space of functions in $H^1(U)$ with vanishing trace on $\Gamma \subset \p U$ for an open and bounded Lipschitz domain $U\subset \IR^n$}}
\definesymbol[FA03]{H1div}{F}{\detokenize{$H^1_{\div}(U)$}}{\detokenize{space of functions in $L^2(U;\IR^n)$ with distributional divergence belonging to $L^2(U)$ for an open and bounded Lipschitz domain $U\subset \IR^n$}}
\definesymbol[FA04]{H1pm0}{F}{\detokenize{$H^1_{\pm,0}(Z_*)$}}{space of functions in $H^1(Z_*)$ with vanishing trace on $S_*^+ \cup S_*^-$}
\definesymbol[FL01]{L}{F}{\detokenize{$\calL$}}{space $L^2(\Ome^+) \times L^2(\Sigma \times Z_*) \times L^2(\Ome^-)$}
\definesymbol[FL02]{Leps}{F}{\detokenize{$\calL_\eps$}}{space $L^2(\Omee)$ equipped with $\eps$-scaled scalar product}
\definesymbol[FL03]{Lepsh}{F}{\detokenize{$\calL_{\eps,h}$}}{space $\detokenize{L^2(\Omeh)}$ equipped with $\eps$-scaled scalar product}
\definesymbol[FH01]{H}{F}{\detokenize{$\calH$}}{solution space of the macroscopic problem}
\definesymbol[FH02]{Heps}{F}{\detokenize{$\calH_\eps$}}{space $H^1(\Omee)$ equipped with $\eps$-scaled scalar product}
\definesymbol[FH03]{HepsM}{F}{\detokenize{$\calH_\eps^{\mathrm{M}}$}}{\detokenize{space of restrictions to $\OmeMS$ of functions belonging to $\calH_\eps$}}
\definesymbol[FH04]{Heps0pm}{F}{\detokenize{$\calH_{\eps,0}^\pm$}}{space of functions in $H^1(\Omee^\pm)$ with vanishing trace on $S_{*,\eps}^\pm$, equipped with $\eps$-scaled scalar product}
\definesymbol[FH05]{Heps0M}{F}{\detokenize{$\calH_{\eps,0}^{\mathrm{M}}$}}{space of functions in $\detokenize{H^1(\OmeMS)}$ with vanishing trace on $\detokenize{S_{*,\eps}^+\cup S_{*,\eps}^-}$, equipped with $\eps$-scaled scalar product}
\definesymbol[FH06]{Hepsh}{F}{\detokenize{$\calH_{\eps,h,0}$}}{subspace of $\detokenize{H^1(\Omeh)}$ equipped with $\eps$-scaled scalar product}

\begin{document}

\maketitle

\begin{abstract}
	We study the asymptotic behaviour of a system of nonlinear reaction--diffusion--advection equations in a domain consisting of two bulk regions connected via microscopic channels distributed within a thin membrane. Both the width of the channels and the thickness of the membrane are of order $\varepsilon \ll 1$, and the geometry evolves in time in an a priori known way.
	We consider nonlinear flux boundary conditions at the lateral boundaries of the channels and critical scaling of the diffusion inside the layer. Extending the method of homogenisation in domains with evolving microstructure to thin layers, we employ two-scale convergence and unfolding techniques in thin layers to derive an effective model in the limit $\varepsilon \to 0$, in which the membrane is reduced to a lower-dimensional interface. We obtain jump conditions for the solution and the total fluxes, which involve the solutions of local, space--time-dependent cell problems in the reference channel.
\end{abstract}

\tableofcontents

\newpage
\section{Introduction} \label{Sec:Intro}
Thin porous layers with microscopic heterogeneous structures which regulate exchange between two bulk regions play a fundamental role in a multitude of applications ranging from biological transport to industrial filtration processes. Examples include ion channels distributed within cell membranes \cite{Hil01}, polymer separators in lithium-ion batteries \cite{AroZha04}, single-ion-conducting membranes in polymer--electrolyte fuel cells \cite{Web04} and thin-film composite membranes used in reverse osmosis \cite{Bak12, Gei10}.
In these systems, macroscopic transport properties, like permeability, selectivity and effective reaction rates, are determined by the microscopic geometry of the layer, which exhibits spatial heterogeneity at scales several orders of magnitude smaller than the bulk regions.

Additional complexity arises when the microscopic geometry undergoes an evolution in time. For example, this is the case in conformational changes of ion channels regulating their conductivity \cite{Hil01, LopRom19}, dendrite growth progressively occluding pores in battery separators \cite{JanElyGar15} or pore blocking or fouling phenomena altering the effective pore size distribution in membrane filtration \cite{Menetal09}. Consequently, transport processes are coupled to geometric changes, leading to time-dependent effective properties which must be accounted for in predictive models.

Direct numerical simulations of such systems face a fundamental challenge as they require to resolve both the macroscopic (observation)  scale and the microscopic scale of the heterogeneities. This motivates the combined use of techniques from periodic homogenisation and dimension reduction to derive effective models encoding the effect of the microstructure in upscaled non-trivial transmission conditions given at a sharp interface between the bulk regions, thus allowing for a macroscopic approximation of the fully resolved system.

In this work, we apply this strategy to a prototypical system of nonlinear reaction--diffusion--advection equations for an unknown vector of concentrations $u_\eps=u_\eps(t,x)$ in a microscopic domain composed of two bulk regions $\Omee^\pm$ (in which $u_\eps$ is denoted by $u_\eps^\pm$) connected through channels in a thin porous layer $\OmeMS(t)$ of thickness $\eps>0$ (in which $u_\eps$ is denoted by $u_\eps^{\mathrm{M}}$), where the time parameter $t$ describes the current configuration of the time-evolving (not necessarily periodic) heterogeneous microstructure. The microscopic system under consideration reads
\begin{align}
	\left\{\!\!\begin{array}{r@{\ }c@{\ }l@{\ }c@{\ \, }l}
		\p_t u_{j\eps}^\pm - \nabla \cdot\left(D_{j}^\pm \nabla u_{j\eps}^\pm - u_{j\eps}^\pm q_{j}^\pm\right)  & = & f_{j}(u_\eps^\pm) && \text{in } (0,T)\times \Omee^\pm, \\
		\frac 1\eps \p_t u_{j\eps}^{\mathrm{M}} - \nabla \cdot \left( \eps D_{j\eps}^{\mathrm{M}} \nabla u_{j\eps}^{\mathrm{M}} - u_{j\eps}^{\mathrm{M}} q_{j\eps}^{\mathrm{M}} \right) & = & \frac 1\eps g_{j}(u_\eps^{\mathrm{M}}) && \text{in } \set{(t,x)}{t \in (0,T),\, x \in \OmeMS(t)},\\
		(D_j^\pm \nabla u_{j\eps}^\pm - u_{j\eps}^\pm q_j^\pm )\cdot \nu_\eps^\pm & = & 0 && \text{on } (0,T)\times (\p\Omee^\pm\setminus S_{*,\eps}^\pm), \\
		-(\eps D_{j\eps}^{\mathrm{M}} \nabla u_{j\eps}^{\mathrm{M}} - u_{j\eps}^{\mathrm{M}} q_{j\eps}^{\mathrm{M}} + \frac 1\eps u_{j\eps}^{\mathrm{M}} b_\eps^{\mathrm{M}}) \cdot \nu_\eps^{\mathrm{M}}  & = & h_{j}(u_\eps^{\mathrm{M}}) && \text{on } \set{(t,x)}{t \in (0,T),\, x \in N_\eps(t)}, 
	\end{array} \right.
\end{align}
where $b_\eps^{\mathrm{M}}$ is the velocity of the evolving lateral boundary $N_\eps(t)$ of the channels. Moreover, we equip the system with appropriate transmission conditions on the top and bottom $S_{*,\eps}^\pm$ of the channels, guaranteeing the continuity of the solution and the total fluxes. 
The diffusion in the thin layer is of critical order $\eps$, i.e.~the time scale of diffusion in the channels with respect to the microscopic scale matches that of diffusion in the bulk on the macroscopic scale, see \cite{PetBoe05}, and the (surface) reaction kinetics are assumed to be nonlinear with Lipschitz growth. The goal of this work is to establish a rigorous combined homogenisation--dimension reduction analysis as $\eps\to 0$ of the problem above including the identification of the system of equations satisfied by the limit of $u_\eps$.

In a prototypical application, $u_{\eps}$ would model the concentration of a chemical substance (i.e.~in terms of its volume fraction) as it is transported through the channels by means of diffusion and drift due to the velocity field $q_{j\eps}$. In this setting, the terms $f_j$ and $g_j$ would account for reactions in the respective part of the domain, while $h_j$ describes the transfer of mass through the channel walls. One concrete example for the latter phenomenon is the motion of charged ions through ion channels \cite{SchNadEis01} or artificial nanopores in polymers \cite{Pieetal13}. There, the change of the channel geometry could model an opening and closing (ion channels) or molecules binding and unbinding to the channel walls (nanopores). Clearly, a prescribed evolution only serves as a first step towards the more realistic situation of a coupling to the actual concentrations as well as the mass flux through the channel walls, which presents an interesting venue for further investigations.

\paragraph{Mathematical Approach and Contributions}
In what follows, $\eps>0$ is a small parameter describing the typical length of the microscopic channels in relation to the observation length scale as well as the periodicity of the fixed reference layer $\OmeMS$ which is built by periodic repetition of the $\eps$-scaled unit channel cell $Z_*$. We assume the evolution of the channels to be given a priori in terms of a map $\psi_\eps \colon [0,T]\times \overline{\OmeMS} \to \IR^n$, that is $\OmeMS(t)= \psi_\eps(t,\OmeMS)$, and similarly for the moving boundary $N_\eps(t)$ of the channels. The maps $\psi_\eps$ are assumed to be $C^1$-diffeomorphisms satisfying particular $\eps$-uniform estimates with so-called limit transformation $\psi_0$, see \ref{it:T1:est_psi}--\ref{it:T4:2s_conv} below.

Since the evolution of the microstructure is known a priori, we follow the approach outlined in \cite{Pet07} and transform the problem to one on the fixed reference geometry $\OmeMS$, but with coefficients dependent on the transformation $\psi_\eps$ itself. Subsequently, we provide a rigorous derivation of the macroscopic problem making use of the theory of two-scale convergence and unfolding techniques in thin domains, see \cite{NeuJae07, MarMar00, GahNeu21, BhaGahNeu22}. Their application is based on $\eps$-uniform a priori estimates for the solutions in Sobolev--Bochner spaces with $\eps$-dependent weights adapted to the problem-structure. Moreover, $\eps$-uniform estimates only hold for the transformed time derivatives $\p_t(J_\eps \tu_\eps)$, where $J_\eps = \det \DD \psi_\eps$ denotes the Jacobian of the transformation, and not for $\p_t u_\eps$. As these estimates scale badly in $\eps$ when restricted to the layer $\OmeMS$, a standard Aubin--Lions-type argument cannot be applied to obtain strong convergence of the microsolutions. Instead, we use unfolding techniques in thin domains which require the application of a Kolmogorov--Simon-type argument. The latter one is based on an abstract duality argument for obtaining the existence and $\eps$-uniform estimates of the time derivative of the unfolded sequence, and the derivation of estimates on shifts of the microsolutions.
Finally, an abstract convergence result for the generalised time derivatives is needed in order to pass to the limit $\eps \to 0$, where we obtain two reaction--diffusion--advection-type problems posed on the bulk domains $\Ome^\pm$ separated by the sharp interface $\Sigma$, and coupled by effective transmission conditions for the solutions and the total fluxes on $\Sigma$. These transmission conditions are given in terms of local cell problems posed in the channel cell $Z_*$. 
Here, the macrovariable $x^\prime \in \Sigma$ enters as a parameter as we consider the critical case of small diffusivity of order $\eps$ in the layer, leading to rapid oscillations of the microsolutions with respect to the spatial variable, and hence a dependency on the macro- and the microvariable of the local cell solutions in the limit $\eps \to 0$. After applying the reverse transformation (which is given in terms of the two-scale limit $\psi_0$ of the microscopic transformations $\psi_\eps$) we eventually obtain local cell problems, each posed on an evolving microcell $Z_*(t,x^\prime)$, and thus capturing the limiting evolution of the microstructure.

To the best of our knowledge, this is the first time that the techniques of rigorous dimension reduction and homogenisation in a thin layer, whose microstructure evolves with time, are combined. Our main results are as follows.
\begin{itemize}
	\item Derivation of $\eps$-uniform a priori estimates for the microsolutions and their shifts on the fixed reference domain (including the thin layer $\OmeMS$), taking into account the transformation-dependent coefficients (see \autoref{lem:apriori} and \autoref{lem:shifts}). 
	\item General compactness and convergence result for the generalised time derivatives considered as functionals on the whole microscopic domain $\Omee$ (see \autoref{thm:convergence}).
	\item Rigorous dimension reduction and homogenisation limit, both on the fixed reference and the evolving domain. This includes the derivation and analysis of a macroscopic problem of reaction--diffusion--advection-type, subject to effective transmission conditions across the interface $\Sigma$, which are given in terms of local cell problems on the evolving cell $Z_*(t,x^\prime)$ (see \autoref{thm:macroscopic} and \autoref{cor:ev_macro_prob}).
\end{itemize}
The results and framework established here provide a basis for the more complex scenario in which the channel evolution is fully coupled to the reaction--diffusion--advection process.

\paragraph{Literature}
An approach to analyse the homogenisation of problems posed on domains with time-dependent microstructure by transforming them onto a fixed reference domain was proposed in \cite{Pet07, Mei08} and applied to coupled reaction--diffusion processes in \cite{Pet09} and chemical degradation mechanisms in \cite{Pet07b, PetBoe09}. By a similar technique, surface exchange processes and chemical reactions at interfaces were treated in \cite{Dob15}. The aforementioned strategy assumes the equivalence of the direct homogenisation procedure on the time-dependent domain and the one on the fixed reference geometry, which which was proven for porous bulk domains by \cite{Wie22}. 
This approach allows the application of two-scale convergence and unfolding methods (see \cite{All92,Ngu89, CioDamGri02} and the classical textbooks \cite{Hor97, CioSai99, CioDamGri18}) also in the case of a priori time-dependent geometries. Based on this method, in \cite{GahNeuPop21} rigorous homogenisation results for reaction--diffusion--advection equations in bulk domains with a priori known evolution of the microstructure were established. 
The corresponding analysis for the quasistationary Stokes flow through perforated media was performed in \cite{WiePet24}, and in \cite{WiePet25} for the instationary case, including the identification of memory effects arising from the evolution.
A linear, fully coupled thermoelasticity problem in a two-phase medium was studied in \cite{EdeMun17}, demonstrating that similar techniques extend to systems undergoing phase transformations.

If the evolution of the microstructure is an unknown of the problem itself and coupled to the transport processes, the analysis becomes even more challenging. For mineral dissolution and precipitation models, we mention \cite{GahPop23}, where the local concentration determines the (radially symmetric) evolution of the grains. The homogensiation analysis of the closely related problem local colloid evolution induced by reaction and diffusion was performed in \cite{WiePet23}. Very recently, a comprehensive treatment of Stokes flow, advection--reaction--diffusion and adsorption--desorption processes in freely evolving porous media, i.e.~fully coupled to the evolution of microstructure of the medium, has been presented in \cite{GahPetPopWie25}, which additionally provides a detailed overview of the state of the art in the field of homogenisation of evolving microstructures. We also mention \cite{EdeMun26}, in which a one-phase thermoelasticity system with phase transformations and small growing or shrinking inclusions was studied, describing the evolution by a height function and applying the Hanzawa transform.

\medskip
The derivation of effective interface and transmission conditions for transport between bulk regions across thin layers combines homogenisation with dimension reduction.
For fixed heterogeneous layers, rigorous transmission conditions for reaction--diffusion processes with critical diffusivity of order $\eps$ were derived in \cite{NeuJae07}. Effective transmission coefficients determined by local cell problems were obtained, thus incorporating information on the microstructure in the macroscopic model. This framework was extended to diffusivities of order $\eps^\gamma$, $\gamma \in [-1,1)$, in \cite{GahNeuKna17}, leading to qualitatively different interface conditions, and to nonlinear transmission conditions at the bulk--layer interface in \cite{GahNeuKna18}. Corrector and rigorous error estimates were established in \cite{GahJaeNeu21}, providing quantitative bounds on the approximation error.
For thin layers composed of periodic channel-like structures, in \cite{GahNeu21} the case of critically scaled diffusivity was analysed, including the derivation of effective jump conditions given in terms of local cell problems in the reference channel. These results were extended in \cite{BhaGahNeu22} to more general reaction kinetics at the lateral boundaries and subcritical scaling. 
By means of gradient flow theory and EDP-convergence, diffusion--advection in a structure composed of thin layers whose thickness is tending to zero at different rates was analysed in \cite{FreLie21}, showing that the limit system also has a gradient structure.

\paragraph{Structure of the paper}
This paper is organised as follows. In \autoref{sec:setting}, we introduce the microscopic model, the geometric setting and the assumptions on the evolution of the microstructure as well as on the data. Then, in \autoref{sec:existence} the problem on the time-dependent geometry is transformed to one on a fixed, periodic reference geometry. Next, existence of weak solutions in function spaces adapted to the scaling of the equations in the layer and uniform a priori estimates for the microsolutions and their shifts are established. \autoref{sec:macro_prob} deals with the derivation of the homogenised macroscopic problem, which is based on weak and strong two-scale compactness results for the microsolutions (\autoref{subsec:convergence}), the limit passage $\eps \to 0$ and identification of the macroscopic problem on the reference geometry and local cell problems (\autoref{subsec:macro_prob_ref}) and finally the reverse transformation back to the evolving domain (\autoref{subsec:macro_prob_evolving}). Necessary (known) results about two-scale convergence for thin channels are collected in \autoref{app:two-scale}, while a summary of the notation used for function spaces and the description of the geometry is provided in the \hyperref[symbols]{List of symbols}.

\paragraph{Notation}
Throughout this paper, we make use of the following notation.
We decompose coordinates $x \in \IR^n$ into $x=(x', x_n) \in \IR^{n-1} \times \IR$. By $\floor{x}$, we denote the floor function, applied in a component-wise manner to $x\in \IR^n$. Given an invertible matrix $F\in \IR^{n\times n}$, we define $F^{-\top}\coloneqq (F^{-1})^\top$.
For two $\eps$-dependent quantities $A_\eps$ and $B_\eps$, we write $A_\eps \lesssim B_\eps$ if there exists a universal constant $C>0$ independent of $\eps$ such that $A_\eps \leq C B_\eps$. Moreover, we write $A_\eps \simeq B_\eps$ if $A_\eps \lesssim B_\eps$ and $B_\eps \lesssim A_\eps$.
To shorten notation, we also set $\sumpm a_\pm = a_+ + a_-$ for $a_\pm \in \IR$.

We further use the following notation for function spaces.
For an open and bounded set $U\subset \IR^n$ with Lipschitz boundary $\p U$ and $\Gamma \subset \p U$ we denote by $H^1_{\Gamma,0}(U) \usesym{H0Gamma}$ the Sobolev space of functions in $H^1(U)$ with vanishing trace on $\Gamma$.
For a Gelfand triple $V\xhookrightarrow{} H \xhookrightarrow{} V^\prime$, the Sobolev--Bochner space $W^{1,2,2}(0,T;V,V^\prime)$ consists of all functions in $L^2(0,T;V)$ with generalised time derivative in $L^2(0,T;V^\prime)$. A complete list of the function spaces used in this work is found in the \hyperref[symbols]{List of symbols}.

\section{The microscopic model} \label{sec:setting}
We first give a precise description of the microscopic geometry evolving in time in terms of a (microscopic) reference layer. Then, we formulate the microscopic model and state the main assumptions regarding the time evolution of the microstructure and the data.

\paragraph{Description of the reference geometry}
In what follows, let $\eps=(\eps_k)_{k\in \IN}$ be a sequence tending to $0$ such that $\frac 1{\eps_k} \in \IN$, and let $n \in \IN_{\geq 2}$ and $H>0$ be fixed.
We consider the macroscopic domain $\Ome \coloneqq Y \times (-H,H) \subset \IR^n$, where $Y \coloneqq (0,1)^{n-1}$ is the $(n-1)$-dimensional unit cube, and we subdivide $\Ome$ as follows. The bulk regions $\Omee^+\usesym{Omeepm} \coloneqq Y \times (\eps,H)$ and $\Omee^- \coloneqq Y \times (-H,-\eps)$ are separated by the thin layer $\OmeM\usesym{OmeM} \coloneqq Y \times (-\eps,\eps)$, and we denote the interfaces between these subsets by $S_\eps^\pm\usesym{Sepspm} \coloneqq Y \times \{\pm \eps\}$.
In the limit $\eps \to 0$, the thin layer $\OmeM$ reduces to the $(n-1)$-dimensional hypersurface $\Sigma\usesym{Sigma} \coloneqq Y \times \{0\}$ and we write $\Ome^+\usesym{Omepm} = Y \times (0, H)$ and $\Ome^-= Y \times (-H, 0)$.

To describe the microscopic structure of the thin layer $\OmeM$, let $Z \usesym{Z} \coloneqq Y \times (-1,1)$ be the $n$-dimensional standard unit cell and denote its upper and lower boundary by $S^\pm\usesym{Spm} \coloneqq \set{z \in  \p Z}{z_n = \pm 1}$. 
A channel across $Z$ is represented by an open Lipschitz domain $Z_*\usesym{Z*} \subset Z$, see \autoref{fig:microcell}, for an illustration, such that both
\begin{align}
    S^\pm_*\usesym{Sspm} \coloneqq \set{z \in  \p Z_*}{z_n = \pm 1}
\end{align} 
are Lipschitz domains in $\IR^{n-1}$ of positive measure. The lateral boundary of the channel is denoted by
\begin{align}
    N\usesym{N} \coloneqq \p Z_* \setminus (S^+_* \cup S^-_*)
\end{align}
and we assume that $N$ has a positive distance to $\p Z \setminus (S^+ \cup S^-)$ so that the lateral channel boundary does not touch the lateral boundary of the unit cell. The periodic microstructure of the thin layer $\OmeM$ is then described by repetition of $\eps Z_*$ along $\Sigma$. More precisely, we define the channel domain $\OmeMS \usesym{OmeMS}$ by
\begin{align}
    \OmeMS \coloneqq \bigcup_{k' \in \calI_\eps} \eps (Z_* + (k',0)),
\end{align}
where $\calI_\eps \coloneqq \set{k' \in \IZ^{n-1}}{\eps Z + \eps (k',0) \subset \OmeM}$, the combined interfaces $S_{*,\eps}^\pm \usesym{Ssepspm}$ between the channels and the bulk regions accordingly by
\begin{align}
    S_{*,\eps}^\pm \coloneqq \bigcup_{k' \in \calI_\eps} \eps (S_*^\pm + (k',0)),
\end{align}
and the set of lateral boundaries of all channels $N_\eps \usesym{Neps}$ by
\begin{align}
    N_\eps \coloneqq \bigcup_{k' \in \calI_\eps} \eps (N + (k',0)).
\end{align}
Finally, we denote the microscopic domain $\Omee \usesym{Omee}$ as the union
\begin{align}
	\Omee \coloneqq \interior{\Omee^+ \cup \OmeMS \cup \Omee^- \cup S_{*,\eps}^+ \cup S_{*,\eps}^-},
\end{align}
see \autoref{fig:micro_geo} for an illustration, and assume it to be Lipschitz regular. For a function $v_\eps$ defined on $\Omee$, we use the superscripts $\pm$ and $M$ to denote its restriction to the subdomains $\Omee^\pm$ and $\OmeMS$, respectively, and we also write $v_\eps=(v_\eps^+,v_\eps^{\mathrm{M}},v_\eps^-)$.
\begin{figure}[htbp]
	\begin{subfigure}[t]{.4\textwidth}
		\centering
		\includegraphics[scale=0.7]{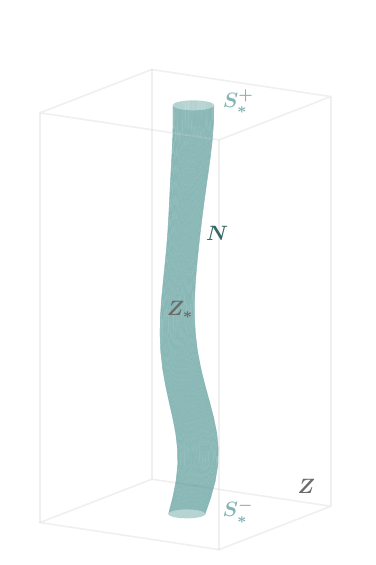}
		\caption{The standard cell $Z$ contains the reference channel $Z_*$ (dark green non-cylindric tube) with top and bottom $S_*^\pm$ and lateral boundary $N$.}
		\label{fig:microcell}
	\end{subfigure}
	\hfill
	\begin{subfigure}[t]{.55\textwidth}
		\centering
		\includegraphics[scale=0.7]{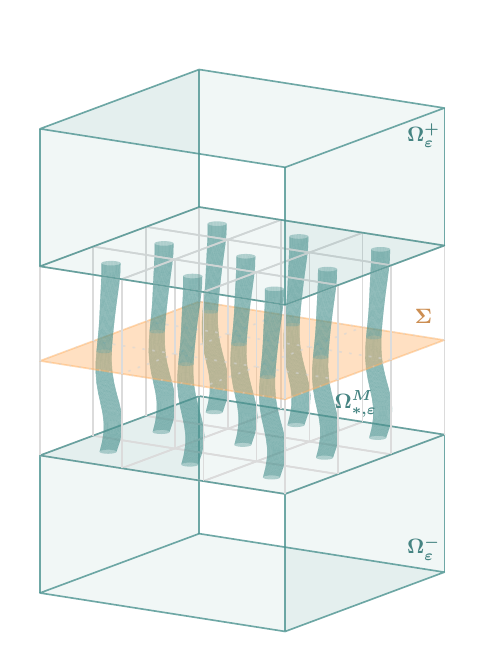}
		\caption{The microscopic reference domain $\Omee$ consists of two bulk regions $\Omee^\pm$, being separated by the thin perforated layer $\OmeM$ with microscopic structure $\OmeMS$. The microstructure of the layer is described by periodic repetition of scaled microcells along the interface $\Sigma$.}
		\label{fig:micro_geo}
	\end{subfigure}
	\caption{Illustration of the microscopic reference domain $\Omee$ in $n=3$ dimensions for $\eps=\frac 13$.}
\end{figure}

\paragraph{Description of the evolving geometry}
For $t\in [0,T]$, the evolving channel domain $\OmeMS(t)$ and the evolving lateral boundary of the channels $N_\eps(t)$ are defined using a map
\begin{align}
    \psi_\eps \colon [0,T] \times \overline{\OmeM} \to \overline{\OmeM}
\end{align}
with $\psi_\eps(t,\cdot)\rvert_{S_{*,\eps}^\pm} = \mathrm{id}_{S_{*,\eps}^\pm}$, and we write $\OmeMS(t)\coloneqq \psi_\eps(t,\OmeMS)$ and $N_\eps(t)\coloneqq \psi_\eps(t,N_\eps)$. For further assumptions on $\psi_\eps$, we refer to \ref{it:T1:est_psi}--\ref{it:T4:2s_conv} below. 
We introduce the non-cylindrical space--time domains
\begin{align}
    \QTeM \usesym{QTeM} \coloneqq \set{(t,x)}{t \in (0,T),\, x \in \OmeMS(t)} \qquad \text{and} \qquad 
    \calN_{T,\eps} \usesym{NTe} \coloneqq \set{(t,x)}{t \in (0,T),\, x \in N_\eps(t)},
\end{align}
and write
\begin{align}
	\QTe \usesym{QTe} \coloneqq \set{(t,x)}{t \in (0,T),\, x \in \interior{\Omee^+ \cup S_{*,\eps}^+ \cup \OmeMS(t) \cup S_{*,\eps}^- \cup \Omee^-}}.
\end{align}

\paragraph{The microscopic model}
On $\QTe$, we consider a system of reaction--diffusion--advection equations for the unknown function $u_\eps=(u_{1\eps}, \ldots, u_{m\eps}) \colon \calQ_{T,\eps} \to \IR^m$ given by
{\mathtoolsset{showonlyrefs=false} 
\begin{subequations}\label{eq:RDA_full}
\begin{align}\label{eq:RDA_evms}
	\left\{\!\!\begin{array}{r@{\ }c@{\ }l@{\ }c@{\ \, }l}
		\p_t u_{j\eps}^\pm - \nabla \cdot\left(D_{j}^\pm \nabla u_{j\eps}^\pm - u_{j\eps}^\pm q_{j}^\pm\right)  & = & f_{j}(u_\eps^\pm) && \text{in } (0,T)\times \Omee^\pm, \\
		\frac 1\eps \p_t u_{j\eps}^{\mathrm{M}} - \nabla \cdot \left( \eps D_{j\eps}^{\mathrm{M}} \nabla u_{j\eps}^{\mathrm{M}} - u_{j\eps}^{\mathrm{M}} q_{j\eps}^{\mathrm{M}} \right) & = & \frac 1\eps g_{j}(u_\eps^{\mathrm{M}}) && \text{in } \QTeM,
	\end{array} \right.
\end{align}
and subject to the flux boundary conditions
\begin{align}\label{eq:RDA_BC_evms}
    \left\{\!\!\begin{array}{r@{\ }c@{\ }l@{\ }c@{\ \, }l}
    -(D_j^\pm \nabla u_{j\eps}^\pm - u_{j\eps}^\pm q_j^\pm )\cdot \nu_\eps^\pm & = & 0 && \text{on } (0,T)\times (\p\Omee^\pm\setminus S_{*,\eps}^\pm), \\
    -(\eps D_{j\eps}^{\mathrm{M}} \nabla u_{j\eps}^{\mathrm{M}} - u_{j\eps}^{\mathrm{M}} q_{j\eps}^{\mathrm{M}} + \frac 1\eps u_{j\eps}^{\mathrm{M}} b_\eps^{\mathrm{M}}) \cdot \nu_\eps^{\mathrm{M}}  & = & h_{j}(u_\eps^{\mathrm{M}}) && \text{on } \calN_{T,\eps}, 
    \end{array} \right.
\end{align}
on the exterior boundary and the lateral channel walls, respectively, where $b_\eps^{\mathrm{M}}(t,\cdot) = \p_t\psi_\eps(t,\psi_\eps^{-1}(t,\cdot))$ is the velocity of the evolving boundary $N_\eps(t)$, and $\nu_\eps^\pm$ denotes the outer unit normal of $\Omee^\pm$ and $\nu_\eps^{\mathrm{M}}=\nu_\eps^{\mathrm{M}}(t)$ the outer unit normal of $\OmeMS(t)$ with respect to $N_\eps(t)$. From a physical point of view, the flux term induced by $b_\eps^{\mathrm{M}}$ describes that when the solute separates from the carrier phase to become part of the channel walls excess solute is pushed away by the movement of the lateral boundaries of the channels.
As initial conditions, we require 
\begin{align}\label{eq:RDA_IC_evms}
	u_\eps(0) = U_{\eps,0} \qquad \text{in } \Omee(0)
\end{align}
for given functions $U_{\eps,0} \colon \Omee(0) \to \IR^m$.
We further impose the natural transmission conditions
\begin{align}\label{eq:RDA_TC_evms}
	\left\{\!\!\begin{array}{r@{\ }c@{\ }l@{\ }c@{\ \, }l}
		u_\eps^\pm & = & u_\eps^{\mathrm{M}} && \text{on } (0,T)\times S_{*,\eps}^\pm,  \\
		\left(-D_{j}^\pm \nabla u_{j\eps}^\pm + u_{j\eps}^\pm q_{j}^\pm \right) \cdot \nu_\eps^\pm & = & \left(-\eps D_{j\eps}^{\mathrm{M}} \nabla u_{j\eps}^{\mathrm{M}} + u_{j\eps}^{\mathrm{M}} q_{j\eps}^{\mathrm{M}}\right) \cdot \nu_\eps^\pm && \text{on } (0,T)\times S_{*,\eps}^\pm,
	\end{array}\right.
\end{align}
\end{subequations}}%
i.e.~the continuity of the solutions and of the total normal fluxes across $S_{*,\eps}^\pm$. 
For the detailed assumptions on the diffusion coefficients $D_\eps$ and the advective velocity $q_\eps$, the reaction rates $f$ and $g$, the adsorption/desorption rate $h$, as well as the initial conditions $U_{\eps,0}$, we refer to \ref{it:propD}--\ref{it:propfg} below.

\begin{rem}[Mass conservation]
	We call $m_\eps(t) = \sumpm \int_{\Omee^\pm} u_\eps^\pm(t,x) \dd x + \frac 1\eps \int_{\OmeMS(t)} u_\eps^{\mathrm{M}}(t,x) \dd x$ the total mass of $u_\eps$, where the prefactor $\frac 1\eps$ takes into account the shrinking of the domain $\OmeMS(t)$ to the interface $\Sigma$ as $\eps \to 0$. 
    Assuming further that the reaction terms $f$ and $g$ and the adsorption/desorption term $h$ vanish for non-positive concentrations $u_\eps$, it follows by standard arguments that $u_\eps$, and thus also $m_\eps$, is non-negative. Moreover, an application of Reynold's transport theorem shows that, in the case of completely vanishing nonlinearities $f,\, g,\, h$, sufficiently regular solutions $u_\eps$ of \eqref{eq:RDA_full} are mass-conserving.
\end{rem}

We introduce the following concept of weak solutions to {\eqref{eq:RDA_full}}. To avoid the usage of generalised time derivatives for functions in Bochner spaces with values in time-dependent Banach spaces, we formally integrate by parts using Reynold's transport theorem and choose test functions $\phi \in C^1([0,T]\times \overline{\Ome})^m$ in what follows.
\begin{definition}[Weak solution of \eqref{eq:RDA_full}]
    A function $u_\eps \in L^2(0,T;H^1(\Omee(t))^m)$ with $u_\eps^\pm = u_\eps^{\mathrm{M}}$ on $(0,T)\times S_{*,\eps}^\pm$ is called \ul{weak solution of {\eqref{eq:RDA_full}}} if for all $\phi \in C^1([0,T]\times \overline{\Ome})^m$ with $\phi(T,\cdot)=0$ there holds
\begin{align}
    \begin{split}\label{eq:RDA_weak_sol}
    &-\sumpm \int_0^T \int_{\Omee^\pm} u_{j\eps}^\pm \p_t \phi_j \dd x \dd t - \frac 1\eps \int_0^T \int_{\OmeMS(t)} u_{j\eps}^{\mathrm{M}} \p_t \phi_j \dd x \dd t\\
    &\quad+ \sumpm \int_0^T \int_{\Omee^\pm} D_j^\pm \nabla u_{j\eps}^\pm \cdot \nabla \phi_j \dd x \dd t + \eps \int_0^T \int_{\OmeMS(t)} D_{j\eps}^{\mathrm{M}} \nabla u_{j\eps}^{\mathrm{M}} \cdot \nabla \phi_j \dd x \dd t\\
    &\quad - \sumpm \int_0^T \int_{\Omee^\pm} u_{j\eps}^\pm q_j^\pm  \cdot \nabla \phi_j \dd x \dd t - \int_0^T \int_{\OmeMS(t)} u_{j\eps}^{\mathrm{M}} q_{j\eps}^{\mathrm{M}} \cdot \nabla \phi_j \dd x \dd t \\
    &= \sumpm \int_0^T \int_{\Omee^\pm} f_j(u_{\eps}^\pm) \phi_j \dd x \dd t + \frac 1 \eps \int_0^T \int_{\OmeMS(t)} g_{j}(u_\eps^{\mathrm{M}}) \phi_j \dd x \dd t
    - \int_0^T \int_{N_\eps(t)} h_{j}(u_\eps^{\mathrm{M}}) \phi_j \dd \calH^{n-1}(x) \dd t \\
    &\quad + \sumpm \int_{\Omee^\pm} U_{j\eps,0}^\pm\, \phi_j(0) \dd x + \frac 1\eps \int_{\OmeMS(0)} U_{j\eps,0}^{\mathrm{M}}\, \phi_j(0) \dd x.
    \end{split}
\end{align}
\end{definition}

\paragraph{Assumptions on the transformation}
We assume that the transformation $\psi_\eps\colon [0,T] \times \overline{\OmeMS} \to \IR^n$ is such that for all $t\in [0,T]$, the map $\psi_\eps(t,\cdot)$ is a bijection onto its image $\overline{\OmeMS}(t)\coloneqq \psi_\eps(t,\overline{\OmeMS})$ and a $C^1$-diffeomorphism from $\OmeMS$ to $\OmeMS(t) \coloneqq \psi_\eps(t,\OmeMS)$. We further make the following assumptions, where we denote $\psi_\eps^{-1}(t,\cdot) \coloneqq \psi_\eps(t,\cdot)^{-1}$, $F_\eps \coloneqq \DD_x\psi_\eps$, and $J_\eps \coloneqq \det(F_\eps)$. 
\begin{enumerate}[label=(T\arabic*)]
	\item\label{it:T1:est_psi} For all $t\in [0,T]$, we have $\psi_\eps(t,\cdot) \rvert_{S_{*,\eps}^\pm} = \mathrm{id}_{S_{*,\eps}^\pm}$, $\bar\psi_\eps \coloneqq (\mathrm{id_{\overline{\Omee^+}}},\psi_\eps,\mathrm{id_{\overline{\Omee^-}}})\in C^1([0,T]\times \overline{\Omee})^n$, and there holds 
	\begin{align}\label{eq:est_psi}
		\frac 1\eps \Cknorm{\psi_\eps-\mathrm{id}_{\overline{\OmeMS}}}{0}{[0,T]\times \overline{\OmeMS}} 
		+ \frac 1 \eps \Cknorm{\p_t \psi_\eps}{0}{[0,T]\times \overline{\OmeMS}}
		+ \Cknorm{F_\eps}{0}{[0,T]\times \overline{\OmeMS}} 
		&\lesssim 1.
	\end{align}
    \item\label{it:T2:est_J} We have $J_\eps = \det(\DD_x \bar \psi_\eps) \in C^1([0,T]\times \overline{\Omee})$ and $J_\eps\simeq 1$ uniformly on $[0,T]\times \overline{\OmeMS}$ as well as 
    \begin{align}\label{eq:est_J}
    	\Lpnorm{\p_t J_\eps}{2}{0,T;\calH_\eps^\prime} 
    	+ \eps \Lpnorm{\nabla J_\eps}{\infty}{(0,T)\times \OmeMS}
    	\lesssim 1.
    \end{align}
    Here, the space $\calH_\eps$ is topologically isomorphic to $H^1(\Omee)$, but equipped with an appropriately $\eps$-scaled scalar product, see \autoref{sec:existence} below.
    \item \label{it:T3:shifts} For all $l^\prime \in \IZ^{n-1}$ and $\eps>0$ with $|l^\prime \eps| \ll h$, we have 
  	\begin{align}
  		\frac 1\eps \Lpnorm{\delt \p_t\psi_\eps}{\infty}{(0,T)\times \OmeMSh}
  		+ \Cknorm{\delt F_\eps}{0}{[0,T]\times\overline{\OmeMSh}}
  		+ \eps \Lpnorm{\delt \nabla J_\eps}{\infty}{(0,T)\times \OmeMSh} 
  		& \lesssim |l^\prime \eps|.
  	\end{align}
  	Here, $\OmeMSh$ is a suitably chosen subset of $\OmeMS$ such that, for functions $v_\eps\colon (0,T)\times \OmeMSh \to \IR$, shifts $\delt v_\eps$ of the form $\delta_{l^\prime,\eps} v_\eps(t,x) \coloneqq v_\eps(t,x+(l^\prime,0)\eps)- v_\eps(t,x)$, where $l^\prime \in \IZ^{n-1}$ and $\eps>0$, are well-defined (see \autoref{subsec:shifts}).
    \item\label{it:T4:2s_conv} There exists a function $\psi_0 \in  C^0(\overline{\Sigma}; C^1([0,T] \times \overline{Z_*})^n)$ such that, for all $(t,x^\prime) \in (0,T) \times \Sigma$, the map $\psi_0(t,x^\prime,\cdot_z)$ is a bijection onto its image $\overline{Z_*}(t,x^\prime)\coloneqq \psi_0(t,x^\prime,\overline{Z_*})$ and a $C^1$--diffeomorphism from $Z_*$ to $Z_*(t,x^\prime) \coloneqq \psi_0(t,x^\prime,Z_*)$. Moreover, in the strong two-scale sense (see \autoref{def:2s_conv} below) in any $L^p$-space, $p\in [1,\infty)$, there holds 
  	\begin{align}
  		\frac 1\eps (\psi_\eps - \mathrm{id}_{\OmeMS}) \to \psi_0 - \mathrm{id}_{Z_*}, \qquad
  		F_\eps \to \DD_z \psi_0, \qquad  		
  		F_\eps^{-1} \to \DD_z \psi_0^{-1}, \qquad
  		\frac 1\eps \p_t \psi_\eps \to \p_t \psi_0.
  	\end{align}	
\end{enumerate}
We emphasise that the evolving domain $\OmeMS(t)$, $t\in [0,T]$, does not need to be periodic, but only locally periodic, where the estimate on $\psi_\eps -\mathrm{id}_{\overline{\OmeMS}}$ in \eqref{eq:est_psi} quantitatively describes the deviation between the time-dependent layer $\OmeMS(t)$ and the (periodic) reference configuration $\OmeMS$. For a detailed discussion of the previous assumptions in the context of porous media with evolving microstructure, we refer to \cite{GahNeuPop21, Wie22, Wie24}. Using the extension property of $\psi_\eps$ stated in \ref{it:T1:est_psi}, it is convenient not to distinguish between $\psi_\eps$ as a map on $\overline{\OmeMS}$ and $\overline{\Omee}$, and similarly for $J_\eps$ and $J_\eps^{\mathrm{M}}$. 
The assumptions on the shifts given in \ref{it:T3:shifts} gain importance when establishing strong two-scale convergence, see \autoref{lem:shifts}. We emphasise that the estimates \eqref{eq:est_psi} already yield two-scale convergences as stated in \ref{it:T4:2s_conv} but only in the weak sense and only for a subsequence (see \autoref{lem:2s_cpct}); hence, the requirement of strong two-scale convergence is an additional assumption.
The following remark summarises additional two-scale convergence results obtained from the previously stated assumptions.
\begin{rem}[Two-scale convergences]\label{rem:2s_conv_psi}
	For the mathematical analysis to follow, the two-scale convergence of $\frac 1\eps (\psi_\eps - \mathrm{id}_{\OmeMS})$ is not needed (but see \autoref{rem:gen_2s_conv}). However, from a modelling point of view it guarantees that the limit deformation stays close to the identity, which can also be seen using that $\psi_\eps \to (\mathrm{id}_\Sigma,0)$ in the strong two-scale sense by \ref{it:T4:2s_conv}.
	Moreover, rewriting $F_\eps$ in terms of the cofactor matrix, we also have $F_\eps^{-1} \to F_0^{-1}$, and similarly $J_\eps \to J_0 \coloneqq \det (\DD_z \psi_0)$, in the strong two-scale sense by \autoref{lem:2s_prod}. In particular, by the characterisation of two-scale convergence in \autoref{lem:char_2s_conv} we have that, for all $t\in (0,T)$, $\psi_0(t,\cdot)\rvert_{\Sigma \times S_*^\pm} = \mathrm{id}_{\Sigma \times S_*^\pm}$ and, hence, $J_0^{\mathrm{M}} \equiv 1$ on $(0,T)\times \Sigma \times S_*^\pm$.
\end{rem}

The following remark gives a possible construction of an admissible transformation $\psi_\eps$ by describing the evolution of the channels locally in each microcell which is of particular interest in applications.
\begin{rem}[Local construction of $\psi_\eps$]
    We consider a function $\psi_{*} \colon [0,T] \times \overline{\Sigma} \times \overline{Z} \to \overline{Z}$ with the following properties:
    \begin{enumerate}[label=(\roman*)]
        \item\label{it:T*_est} 
        $\psi_* \in C^1([0,T];C^2(\overline{\Sigma}\times \overline{Z}))$ with $\Cknorm{\psi_*}{1}{[0,T]; C^2(\overline{\Sigma} \times \overline{Z})}\lesssim 1$ and $\det \DD_z \psi_* \gtrsim 1$.
        \item\label{it:T*_ext} There exists $\delta \in (0,\dist(N,\p Z))$ sufficiently small such that
        \begin{align}
        	\psi_*(t,x^\prime,z)=z \qquad \text{for all } (t,x^\prime,z)\in [0,T]\times \overline{\Sigma} \times B_\delta(\p Z).
        \end{align} 
    \end{enumerate}
    Then, after extending $\psi_*(t,x^\prime,\cdot)$ periodically to $\IR^{n-1} \times [-1,1]$ with periodicity cell $Y$, for $(t,x) \in [0,T]\times \overline{\OmeMS}$ we define 
    \begin{align}
        \psi_{\eps}(t,x)
        &\coloneqq \eps \left(\floor*{\tfrac{x^\prime}\eps},0\right) + \eps \psi_{*}\left(t,\eps\floor*{\tfrac{x^\prime}\eps},\tfrac x\eps - \left(\floor*{\tfrac{x^\prime}\eps},0\right)\right)
        = \eps \left(\floor*{\tfrac{x^\prime}\eps},0\right) + \eps \psi_{*}\left(t,\eps\floor*{\tfrac{x^\prime}\eps},\tfrac x\eps\right).
    \end{align}
    The incorporation of the variable $x^\prime \in \overline{\Sigma}$ allows the description of a non-periodic microstructure of the time-dependent layer $\OmeMS(t)$. 
    Then, a direct calculation shows that $\psi_\eps \in C^1([0,T]; C^2(\overline{\OmeMS}))$ as well as the validity of the estimates in \ref{it:T1:est_psi} and \ref{it:T2:est_J}.
    Similarly, the estimates for the shifts in \ref{it:T3:shifts} are obtained using the mean-value theorem.
    Defining the map $\psi_0 \colon [0,T] \times \overline{\Sigma} \times \overline{Z_*} \to \IR^n$ by
    \begin{align}
        \psi_0(t,x^\prime,z) = (x^\prime,0) + \psi_*(t,x^\prime,z),
    \end{align}
    the two-scale convergences in \ref{it:T4:2s_conv} are a direct consequence of the oscillation lemma on thin domains, see \cite[Lemma~4.3]{NeuJae07}, and the uniform convergence $\eps \floor*{\tfrac{x^\prime}{\eps}} \to x^\prime$.  
\end{rem}

\paragraph{Structural assumptions on the data} 
In order to obtain the existence of (weak) microsolutions and suitable a priori estimates, we make the following assumptions on the data $D_\eps$, $q_\eps$, and $U_{\eps,0}$ as well as on the terms $f$, $g$ and $h$:
\begin{enumerate}[label=(D\arabic*)]
    \item\label{it:propD} The diffusion coefficients $D_{j\eps}=(D_j^+,D_{j\eps}^{\mathrm{M}},D_j^-) \in L^\infty(\calQ_{T,\eps})^{n\times n}$ satisfy  $\Lpnorm{D_{j\eps}}{\infty}{\calQ_{T,\eps}}\lesssim 1$ and are strongly uniformly elliptic uniformly in $\eps$, i.e.~for any $\xi\in \IR^n$ we have $\xi^T D_{j\eps} \xi \gtrsim |\xi|^2$ uniformly in $(t,x) \in \QTe$.
    \item\label{it:propq} The advection velocities $q_{j\eps} \in L^\infty(\calQ_{T,\eps})^n$ satisfy $\Lpnorm{q_{j\eps}}{\infty}{\calQ_{T,\eps}}\lesssim 1$.
    \item\label{it:propU0} The initial conditions $U_{\eps,0}=(U_0^+,U_{\eps,0}^{\mathrm{M}},U_0^-)\colon \Omee(0) \to \IR^m$ satisfy
    \begin{align}
    	\sumpm\Lpnorm{U_0^\pm}{2}{\Omee^\pm(0)} + \frac 1\eps \Lpnorm{U_{\eps,0}^{\mathrm{M}}}{2}{\OmeMS(0)}\lesssim 1.
    \end{align}
    \item\label{it:propfg} The source/sink terms $f_{j},\, g_{j}, \,h_j \in L^2((0,T)\times \IR^m)$ are globally Lipschitz continuous in $z\in \IR^m$ uniformly with respect to $t\in (0,T)$. In particular, we have
    \begin{align}
    	|f_{j}(t,z)| + |g_{j}(t,z)| + |h_{j}(t,z)| \lesssim 1 + |z|\label{eq:growth_fgh}
    \end{align}
    uniformly in $t\in (0,T)$. 
\end{enumerate}

\begin{rem}[Assumptions on the reaction rates]
	To focus on the main aspects of this paper and for the sake of readability, we assume that the source/sink terms $f$, $g$ and $h$ are independent of the spatial variable. However, with only slight modifications of the proofs, the results of this work still can be obtained in this more general case if it is assumed that $f$, $g$ and $h$ are uniformly Lipschitz continuous with respect to the concentration variable. Then, after applying the transformation $\psi_\eps$, one obtains an additional $\eps$-dependency of the nonlinearities $g$ and $h$, and it is even possible to consider $\eps$-dependent nonlinearities right from the beginning if their two-scale convergence still can be guaranteed (see e.g.~\cite[Lemma~3.6]{GahNeu21}). It is further possible to relax the assumption on the Lipschitz continuity of the reaction rates by only requiring the Lipschitz continuity of the maps $z\mapsto [f_j(z) z]_+$, $z\mapsto [g_j(z) z]_+$, and $z\mapsto [h_j(z) z]_-$, where $[\cdot]_\pm$ denotes the positive/negative part, and absorbing the corresponding negative/positive part in the operator of the boundary value problem in a coercivity-preserving way.
\end{rem}

\paragraph{Estimates on the shifts of the data} 
The proof of strong two-scale convergence of the microsolutions is based on the following control of the shifts $D_\eps$, $q_\eps$, and $U_{\eps,0}$. As this property is needed only for the transformed problem on the reference domain $\Omee$, we formulate it for the functions
\begin{align}\label{eq:trafo_data}
	\bar D_{j\eps}^{\mathrm{M}} = D_{j\eps}^{\mathrm{M}}\circ_x \psi_\eps, \qquad 
	\bar q_{j\eps}^{\mathrm{M}} = q_{j\eps}^{\mathrm{M}} \circ_x \psi_\eps, \qquad 
	\tilde{U}_{\eps,0} = U_{\eps,0} \circ \psi_\eps(0,\cdot).
\end{align}

\begin{enumerate}[label=(S\arabic*)]
	\item\label{it:shifts:D} The diffusion coefficients $D_\eps$ satisfy $\Lpnorm{\delt \bar D_{j\eps}^{\mathrm{M}}}{\infty}{(0,T)\times \OmeMSh} \lesssim |l^\prime \eps|$.
	\item\label{it:shifts:q} The advection velocities $q_{\eps}$ satisfy $\Lpnorm{\delt \bar q_{j\eps}^{\mathrm{M}}}{\infty}{(0,T)\times \OmeMSh} \lesssim |l^\prime \eps|$.
	\item\label{it:shifts:U0} The initial conditions $U_{\eps,0}$ satisfy $\frac 1\eps \Lpnorm{\delt \tilde{U}_{j\eps,0}^{\mathrm{M}}}{2}{\OmeMSh} \to 0$ as $\eps l^\prime \to 0$.
\end{enumerate}

\paragraph{Asymptotic behaviour of the data} 
The limit behaviour of the functions $D_\eps$, $q_\eps$, and $U_{\eps,0}$ is assumed to be as follows: 
\begin{enumerate}[label=(A\arabic*)]
	\item\label{it:limit:D} There exists a strongly uniformly elliptic function $\bar D_0^{\mathrm{M}} \in L^\infty((0,T)\times \Sigma \times Z_*)^{n\times n}$ such that $\bar D_{j\eps}^{\mathrm{M}} \to \bar D_{j0}^{\mathrm{M}}$ in the strong two-scale sense.
	\item\label{it:limit:q} There exists a function $\bar q_0^{\mathrm{M}} \in L^\infty((0,T)\times \Sigma \times Z_*)^n$ such that $\bar q_{j\eps}^{\mathrm{M}} \to \bar q_{j0}^{\mathrm{M}}$ in the strong two-scale sense.
	\item\label{it:limit:U0} There exists a function $\tilde{U}_{0,0}^{\mathrm{M}} \in L^2(\Sigma \times Z_*)^n$ such that $J_\eps^{\mathrm{M}}(0)\tilde{U}_{j\eps,0}^{\mathrm{M}} \to J_0^{\mathrm{M}}(0)\tilde U_{j0,0}^{\mathrm{M}}$ in the strong two-scale sense. 
\end{enumerate}
Here, the two-scale convergence is assumed to hold with respect to test functions in $L^2$. In \ref{it:limit:U0}, we do not assume the (strong) two-scale convergence of the sequence $(\tilde{U}_{j\eps,0}^{\mathrm{M}})_{\eps}$ as the natural initial conditions for the microscopic problem are given in terms of the product $J_\eps^{\mathrm{M}}(0) \tilde{U}_{j\eps,0}^{\mathrm{M}}$, also see \autoref{def:weak_sol} and \autoref{cor:conv_u}~\ref{it:conv_ptu} below.

\section{Existence of microscopic solutions and a priori estimates}\label{sec:existence}
We introduce a further concept of weak solutions to \eqref{eq:RDA_full} by transforming the weak formulation \eqref{eq:RDA_weak_sol} to the static reference domain $\Omee$. This turns out to suitable for the application of standard tools from the theory of periodic homogenisation in thin layers, which enables us to pass to the limit $\eps \to 0$ in \autoref{sec:macro_prob}. The main ingredients for the limit passage are the a priori estimates and estimates for the shifts of the microsolutions derived in this section.

Given a function $u_\eps$ on $\QTeM$, we denote by $\tu_\eps$ its evaluation in the coordinates on $\OmeMS$, that is for $(t,\tx) \in (0,T) \times \OmeMS$ we have $\tu_\eps(t,\tx) = u_\eps(t,\psi_\eps(t,\tx))$. By the chain rule, we have
\begin{align}
    \nabla u_\eps(t,x)=F_\eps^{-\top} \nabla_{\tilde x} \tu_\eps\rvert_{(t,\psi_\eps^{-1}(t,x))},
\end{align}
whereas the surface measure transforms according to
\begin{align}
    \calH^{n-1}\restr{N_\eps(t)} = J_\eps \norm{F_\eps^{-\top} \tilde \nu_\eps^{\mathrm{M}}} \calH^{n-1}\restr{N_\eps},
\end{align}
see \cite[p.~117]{MarHug94} and the explicit calculation in three dimensions given in \cite[p.~39]{Cia88}, where $\tilde \nu_\eps^{\mathrm{M}}$ is the outer normal of $\OmeMS$ with respect to $N_\eps$. To keep notation simpler, we always write $\nu_\eps^{\mathrm{M}} = \tilde\nu_\eps^{\mathrm{M}}$ in what follows.

Following the approach of \cite{GahNeu21}, we further introduce Hilbert spaces $\calL_\eps$ and $\calH_\eps$ with inner products adapted to the scaling of the equation in the channel domain $\OmeMS$. We begin with
\begin{align}
	\calL_\eps \usesym{Leps} \coloneqq L^2(\Omee^+) \times L^2(\OmeMS)\times L^2(\Omee^-),
\end{align}
equipped with the inner product
\begin{align}
    \scp{v_\eps}{w_\eps}{\calL_\eps} \coloneqq \sumpm \scp{v_\eps^\pm}{w_\eps^\pm}{L^2(\Omee^\pm)} + \frac 1\eps \scp{v_\eps^{\mathrm{M}}}{w_\eps^{\mathrm{M}}}{L^2(\OmeMS)},
\end{align}
and define
\begin{align}
	\calH_\eps \usesym{Heps} \coloneqq \set{v_\eps=(v_\eps^+,v_\eps^{\mathrm{M}},v_\eps^-) \in H^1(\Omee^+) \times H^1(\OmeMS)\times H^1(\Omee^-)}{v_\eps^\pm = v_\eps^{\mathrm{M}} \text{ on } S_{*,\eps}^\pm}
\end{align}
with inner product
\begin{align}
    \scp{v_\eps}{w_\eps}{\calH_\eps} 
    \coloneqq \scp{v_\eps}{w_\eps}{\calL_\eps} + \sumpm \scp{\nabla v_\eps^\pm}{\nabla w_\eps^\pm}{L^2(\Omee^\pm)} + \eps \scp{\nabla v_\eps^{\mathrm{M}}}{\nabla w_\eps^{\mathrm{M}}}{L^2(\OmeMS)}.
\end{align}
To shorten notation, we write $\calH_\eps^{\mathrm{M}}\usesym{HepsM}$ for the space of restrictions to $\OmeMS$ of functions belonging to $\calH_\eps$. 
Clearly, we have $\calL_\eps = L^2(\Omee)$ and $\calH_\eps = H^1(\Omee)$ in a topological sense by identifying the latter spaces with the product spaces on the different subdomains, and consequently $\calH_{\eps} \xhookrightarrow{} L^2(\Omee) \xhookrightarrow{} \calH_{\eps}^\prime$ is a Gelfand triple. Assuming that a function $v_\eps \in L^2(0,T;\calH_\eps)$ with $\p_t v_\eps \in L^2(0,T;\calH_\eps^\prime)$ is more regular in the sense that $\p_t v_\eps^\pm \in L^2(0,T;(H^1(\Omee^\pm))^\prime)$ and $\p_t v_\eps^{\mathrm{M}} \in L^2(0,T;(H^1(\OmeMS))^\prime)$, from the previous Gelfand triple we obtain
\begin{align}
    \dotproduct{\p_t v_\eps}{\phi_\eps}{\calH_\eps^\prime,\,\calH_\eps}
    &= \sumpm \dotproduct{\p_t v_\eps^\pm}{\phi_\eps}{(H^1(\Omee^\pm))^\prime,\, H^1(\Omee^\pm)} 
    + \frac 1\eps \dotproduct{\p_t v_\eps^{\mathrm{M}}}{\phi_\eps}{(H^1(\OmeMS))^\prime,\, H^1(\OmeMS)},\label{eq:ptu_higher_reg}
\end{align}
which shows how the generalised time derivative of $v_\eps$ is related to the scaling of the weak equation in the thin layer $\OmeMS$.
To simplify the writing and using the notation from \eqref{eq:trafo_data}, we introduce the transformed data 
\begin{align}
	\tD_{j\eps}^{\mathrm{M}} = F_\eps^{-1} \bar D_{j\eps}^{\mathrm{M}} F_\eps^{-\top}, \qquad 
	\tq_{j\eps}^{\mathrm{M}} = F_\eps^{-1} \bar q_{j\eps}^{\mathrm{M}}, \qquad 
	\tb_\eps^{\mathrm{M}} = -\p_t \psi_\eps^{-1}\circ_x \psi_\eps = F_\eps^{-1} b_\eps^{\mathrm{M}}\circ_x \psi_\eps = F_\eps^{-1} \p_t \psi_\eps
\end{align}
and recall that $\tilde{U}_{\eps,0} = U_{\eps,0} \circ \psi_\eps(0,\cdot)$. 
We choose test functions $\Phi_\eps = \phi_\eps \circ_x \psi_\eps^{-1}$ for some $\phi_\eps \in C^1([0,T]\times \overline{\Omee})^m$ in \eqref{eq:RDA_weak_sol}, apply the inverse transformation $\psi_\eps^{-1}$ and integrate by parts in time for sufficiently regular $\tu_\eps$ by taking into account \eqref{eq:ptu_higher_reg}. Then, the weak formulation \eqref{eq:RDA_weak_sol} yields 
\begin{align}
	\begin{split}\label{eq:RDA_fd_weak}
		&\dotproduct{\p_t(J_\eps \tu_{j\eps})}{\phi_{j\eps}}{\calH_\eps^\prime,\,\calH_\eps} 
		+ \sumpm \int_{\Omee^\pm} D_{j}^\pm \nabla \tu_{j\eps}^\pm \cdot \nabla \phi_{j\eps} \dd x
		+\eps\int_{\OmeMS} J_\eps^{\mathrm{M}} \tD_{j\eps}^{\mathrm{M}} \nabla \tu_{j\eps}^{\mathrm{M}} \cdot \nabla \phi_{j\eps} \dd x\\
		&\quad - \sumpm \int_{\Omee^\pm} \tu_{j\eps}^\pm q_{j}^\pm \cdot \nabla \phi_{j\eps} \dd x
		- \int_{\OmeMS} J_\eps^{\mathrm{M}} \tu_{j\eps}^{\mathrm{M}} \tq_{j\eps}^{\mathrm{M}} \cdot \nabla \phi_{j\eps} \dd x
		+ \frac 1\eps \int_{\OmeMS} J_\eps^{\mathrm{M}} \tu_{j\eps}^{\mathrm{M}} \tb_\eps^{\mathrm{M}} \cdot \nabla \phi_{j\eps} \dd x\\
		&= \sumpm \int_{\Omee^\pm} f_{j}(\tu_{\eps}^\pm)\phi_{j\eps} \dd x
		+ \frac 1\eps \int_{\OmeMS} J_\eps^{\mathrm{M}} g_{j}(\tu_\eps^{\mathrm{M}}) \phi_{j\eps} \dd x
		- \int_{N_\eps} J_\eps^{\mathrm{M}} \norm{F_\eps^{-\top} \nu_\eps^{\mathrm{M}}} h_{j}(\tu_\eps^{\mathrm{M}}) \phi_{j\eps} \dd \calH^{n-1}(x)
	\end{split}
\end{align}
for all $\phi_\eps \in C^1(\overline{\Ome})^m$ and almost every $t\in (0,T)$. Here, the appearance of the term involving $\tilde b_\eps^{\mathrm{M}}$ is a consequence of the chain rule applied to $\phi$. To use a unified notation, we consider weak solutions $\tu_\eps$ as functions on $(0,T)\times \Omee$ taking into account the extension property of $\psi_\eps$ stated in \ref{it:T1:est_psi}.

\begin{definition}[Weak solution]\label{def:weak_sol}
    A function $\tu_\eps \in L^2(0,T;\calH_\eps^m)$ with $\p_t(J_\eps \tu_\eps) \in L^2(0,T;(\calH_\eps^m)^\prime)$ is called \underline{\em weak solution of the reference problem} if for all $\phi_\eps\in \calH_\eps^m$ and almost every $t\in (0,T)$ there holds \eqref{eq:RDA_fd_weak}
    and the initial condition $(J_\eps \tu_\eps)(0) = J_\eps(0)\tilde{U}_{\eps,0}$ is satisfied a.e.~in $\Omee$. 
\end{definition}
By standard embedding theorems for Sobolev--Bochner spaces, we have $J_\eps \tu_\eps \in C^0([0,T];\calL_\eps^m)$; hence, the initial condition is well-posed. Due to the regularity of $\p_t(J_\eps \tu_\eps)$ and $J_\eps$, we further have 
\begin{align}\label{eq:ptu}
	\p_t \tu_\eps = J_\eps^{-1} \p_t(J_\eps \tu_\eps) - J_\eps^{-1} \tu_\eps \p_t J_\eps \in L^2(0,T;(\calH_\eps^m)^\prime),
\end{align}
and, therefore, $\tu_\eps \in C^0([0,T];\calL_\eps^m)$. Consequently, the initial conditions $(J_\eps \tu_\eps)(0) = J_\eps(0)\tilde{U}_{\eps,0}$ and $\tu_\eps(0) = \tilde{U}_{\eps,0}$ are equivalent.

\begin{rem}[Strong formulation of \eqref{eq:RDA_fd_weak}]
{\mathtoolsset{showonlyrefs=false} 
For a sufficiently regular weak solution $\tu_\eps$ the strong problem associated with \eqref{eq:RDA_fd_weak} is given by 
\begin{subequations}
\label{eq:RDA_fd_full}
\begin{align}\label{eq:RDA_fd}
	\left\{\!\!\begin{array}{r@{\ }c@{\ }l@{\ }c@{\ \, }l}
		\p_t u_{j\eps}^\pm - \nabla \cdot\left(D_{j}^\pm \nabla u_{j\eps}^\pm - u_{j\eps}^\pm q_{j}^\pm\right)  & = & f_{j}(u_\eps^\pm) && \text{in } (0,T)\times \Omee^\pm,\\
		\frac 1 \eps\p_t(J_\eps^{\mathrm{M}} \tu_{j\eps}^{\mathrm{M}}) - \nabla \cdot \left(\eps J_\eps^{\mathrm{M}} \tD_{j\eps}^{\mathrm{M}} \nabla \tu_{j\eps}^{\mathrm{M}} - J_\eps^{\mathrm{M}} \tu_{j\eps}^{\mathrm{M}} \tq_{j\eps}^{\mathrm{M}} + \frac 1\eps J_\eps^{\mathrm{M}} \tu_{j\eps}^{\mathrm{M}} \tb_\eps^{\mathrm{M}}\right) & = & \frac 1\eps J_\eps^{\mathrm{M}} g_{j}(\tu_\eps^{\mathrm{M}}) && \text{in } (0,T)\times \OmeMS,
	\end{array}\right.
\end{align}
equipped with the boundary conditions
\begin{align}\label{eq:RDA_BC_fd}
	\left\{\!\!\begin{array}{r@{\ }c@{\ }l@{\ }c@{\ \, }l}
        -(D_j^\pm \nabla u_{j\eps}^\pm - u_{j\eps}^\pm q_j^\pm )\cdot \nu^\pm & = & 0 && \text{on } (0,T)\times (\p\Omee^\pm\setminus S_{*,\eps}^\pm), \\
        -(\eps \tD_{j\eps}^{\mathrm{M}} \nabla \tu_{j\eps}^{\mathrm{M}} - \tu_{j\eps}^{\mathrm{M}} \tq_{j\eps}^{\mathrm{M}} + \frac 1\eps \tu_{j\eps}^{\mathrm{M}} \tb_\eps^{\mathrm{M}}) \cdot \nu_\eps^{\mathrm{M}} & = & \norm{F_\eps^{-\top} \nu_\eps^{\mathrm{M}}} h_{j}(\tu_\eps^{\mathrm{M}}) && \text{on } (0,T)\times N_\eps,
	\end{array} \right.
\end{align}
the initial condition $\tu_\eps(0)=\tilde{U}_{\eps,0}$ and the transmission conditions 
\begin{align}\label{eq:RDA_TC_fd}
	\left\{\!\!\begin{array}{r@{\ }c@{\ }l@{\ }c@{\ \, }l}
		\tu_\eps^\pm & = & \tu_\eps^{\mathrm{M}} && \text{on } (0,T) \times S_{*,\eps}^\pm,  \\
        \left(-D_{j}^\pm \nabla \tu_{j\eps}^\pm + \tu_{j\eps}^\pm q_{j}^\pm \right) \cdot \nu_\eps^\pm & = & \left(-\eps \tD_{j\eps}^{\mathrm{M}} \nabla \tu_{j\eps}^{\mathrm{M}} + \tu_{j\eps}^{\mathrm{M}} \tq_{j\eps}^{\mathrm{M}}\right) \cdot \nu_\eps^\pm && \text{on } (0,T) \times S_{*,\eps}^\pm.
	\end{array}\right.
\end{align}
\end{subequations}}%
We emphasise that the boundary conditions and the transmission conditions are incorporated in the weak formulation \eqref{eq:RDA_fd_weak} and the definition of the test space $\calH_\eps$ implicitly.
\end{rem}

\subsection{A priori estimates}
We first recall the following trace inequality for functions defined on $\OmeMS$.
\begin{lemma}[Trace estimate]\label{lem:traces}
	For every $\theta \in (0,1]$ and $v_\eps \in H^1(\OmeMS)$, there holds
	\begin{align}
		\int_{N_\eps} |v_\eps|^2 \dd \calH^{n-1}(x)
		\lesssim \theta \eps \int_{\OmeMS} |\nabla v_\eps|^2 \dd x 
		+ \frac{1}{\theta\eps} \int_{\OmeMS} |v_\eps|^2 \dd x.
	\end{align}
\end{lemma}
\begin{proof}
	Decomposing $\OmeMS$ into (shifted and scaled) microcells $\eps(Z_* + (k^\prime,0))$, where $k^\prime\in \calI_\eps$, transforming the domain of integration to $Z_*$, and using the weighted trace inequality for Lipschitz domains as given e.g.~in \cite[Theorem~1.5.1.10]{Gri85} on $\p Z_*$, the result is directly obtained.
\end{proof}

For the weak solutions $\tu_\eps$ of \eqref{eq:RDA_fd_weak}, we have the following a priori estimates.
\begin{lemma}[A priori estimates]\label{lem:apriori}
	The weak solutions $\tu_\eps$ of \eqref{eq:RDA_fd_weak} satisfy the a priori estimates
	\begin{subequations}
    \begin{align}
		\Lpnorm{\tu_\eps}{\infty}{0,T;\calL_\eps^m} + \Lpnorm{\tu_\eps}{2}{0,T;\calH_\eps^m} &\lesssim 1, \label{eq:apriori1}\\
		\Lpnorm{\p_t(J_\eps \tu_\eps)}{2}{0,T;(\calH_\eps^m)^\prime} + \Lpnorm{J_\eps \tu_\eps}{2}{0,T;\calH_\eps^m} & \lesssim 1,\label{eq:apriori2}
	\end{align}
	\end{subequations}
    where the constants only depend on $T$ and the constants from \ref{it:T1:est_psi}, \ref{it:propD}--\ref{it:propfg} and \autoref{lem:traces}.
\end{lemma}
\begin{proof}
    Similar estimates were proved by analogous arguments in \cite[Lemma~1]{GahNeuPop21} for periodically perforated bulk domains, and in \cite[Lemma~2.3]{GahNeu21} for non-evolving channels in the thin layer $\OmeM$, that is for $J_\eps \equiv 1$. As $\calH_\eps \xhookrightarrow{} \calL_\eps \xhookrightarrow{}\calH_\eps^\prime$ is a Gelfand triple, due to the regularity in time of $J_\eps$, we have 
	\begin{align}
		\frac 12 \frac{\dd}{\dd t} \norms{\sqrt{J_\eps}\tu_{j\eps}}{\calL_\eps}^2
		&= \dotproduct{\p_t(\sqrt{J_\eps} \tu_{j\eps})}{\sqrt{J_\eps} \tu_{j\eps}}{\calH_\eps^\prime,\,\calH_\eps}\\
		&= \dotproduct{\p_t(J_\eps \tu_{j\eps})}{\tu_{j\eps}}{\calH_\eps^\prime,\,\calH_\eps} - \frac 12 \dotproduct{(\p_t J_\eps) \tu_{j\eps}}{\tu_{j\eps}}{\calH_\eps^\prime,\,\calH_\eps}\\
		&=\dotproduct{\p_t(J_\eps \tu_{j\eps})}{\tu_{j\eps}}{\calH_\eps^\prime,\,\calH_\eps} - \frac 1{2\eps} \scp{\p_t J_\eps^{\mathrm{M}}}{(\tu_{j\eps}^{\mathrm{M}})^2}{L^2(\OmeMS)}. \label{eq:dual_pairing_u}
	\end{align}
    Rewriting $F_\eps^{-1}$ in terms of the cofactor matrix by Cramer's rule, we also have $\Cknorm{F_\eps^{-1}}{0}{[0,T]\times \overline{\OmeMS}} \lesssim1$. By \ref{it:propD} and \ref{it:T1:est_psi} the diffusion coefficients $\tD_{j\eps}$ are strongly uniformly elliptic uniformly in $\eps$.
    Testing the weak formulation \eqref{eq:RDA_fd_weak} with the function $\tu_{j\eps}$ itself, using $J_\eps \gtrsim 1$, we thus obtain
    \begin{align}
        &\frac{\dd}{\dd t} \norms{\sqrt{J_\eps}\tu_{j\eps}}{\calL_\eps}^2 
        + \sumpm \Lpnorm{\nabla \tu_{j\eps}^\pm}{2}{\Omee^\pm}^2 
        + \eps \Lpnorm{\nabla \tu_{j\eps}^{\mathrm{M}}}{2}{\OmeMS}^2\\
        &\quad \lesssim \sumpm \int_{\Omee^\pm} \tu_{j\eps}^\pm q_j^\pm \cdot \nabla \tu_{j\eps}^\pm \dd x
        + \int_{\OmeMS} J_\eps^{\mathrm{M}} \tu_{j\eps}^{\mathrm{M}} \tq_{j\eps}^{\mathrm{M}} \cdot \nabla \tu_{j\eps}^{\mathrm{M}} \dd x 
        - \frac 1\eps \int_{\OmeMS} J_\eps^{\mathrm{M}} \tu_{j\eps}^{\mathrm{M}} \tb_\eps^{\mathrm{M}} \cdot \nabla \tu_{j\eps}^{\mathrm{M}} \dd x \\
        &\qquad - \frac 1 {2\eps} \int_{\OmeMS} \p_t J_\eps^{\mathrm{M}} (\tu_{j\eps}^{\mathrm{M}})^2 \dd x
		+ \sumpm \int_{\Omee^\pm} f_{j}(\tu_{\eps}^\pm)\tu_{j\eps}^\pm \dd x 
		+ \frac 1\eps \int_{\OmeMS} J_\eps^{\mathrm{M}} g_{j}(\tu_\eps^{\mathrm{M}})\tu_{j\eps}^{\mathrm{M}} \dd x \\
		&\qquad - \int_{N_\eps} J_\eps^{\mathrm{M}} \norm{F_\eps^{-\top} \nu_\eps^{\mathrm{M}}} h_{j}(\tu_\eps^{\mathrm{M}})\tu_{j\eps}^{\mathrm{M}} \dd \calH^{n-1}(x) \eqqcolon \sum_{i=1}^7 I_\eps^{(i)}.
    \end{align}
    We estimate the terms $I_\eps^{(i)}$ separately and obtain for any $\theta \in (0,1]$ by Young's inequality
    \begin{align}
         |I_\eps^{(1)}|
         &\lesssim \sumpm \Lpnorm{q_j^\pm}{\infty}{(0,T)\times\Omee^\pm} \Lpnorm{\tu_{j\eps}^\pm}{2}{\Omee^\pm} \Lpnorm{\nabla \tu_{j\eps}^\pm}{2}{\Omee^\pm}
         \lesssim \frac 1 \theta \sumpm \Lpnorm{\tu_{j\eps}^\pm}{2}{\Omee^\pm}^2 + \theta \sumpm \Lpnorm{\nabla\tu_{j\eps}^\pm}{2}{\Omee^\pm}^2, \label{eq:est1}
    \end{align}
    and similarly
    \begin{align}
        |I_\eps^{(2)}|
        &\lesssim \frac 1{\theta\eps} \Lpnorm{\tu_{j\eps}^{\mathrm{M}}}{2}{\OmeMS}^2 + \theta \eps \Lpnorm{\nabla\tu_{j\eps}^{\mathrm{M}}}{2}{\OmeMS}^2,\label{eq:est2}
    \end{align}
    where we have used $\tq_{j\eps}^{\mathrm{M}}\in L^\infty((0,T)\times \OmeMS)$ with $\Lpnorm{\tq_{j\eps}^{\mathrm{M}}}{\infty}{(0,T)\times \OmeMS} \lesssim 1$ which follows from \ref{it:T1:est_psi} and \ref{it:propq}. 
    As we have no uniform bound $\Cknorm{\p_t J_\eps^{\mathrm{M}}}{0}{[0,T]\times \overline{\OmeMS}} \lesssim 1$ at hand, we rewrite $\p_t J_\eps^{\mathrm{M}} = \nabla \cdot(J_\eps^{\mathrm{M}} \tb_\eps^{\mathrm{M}})$ in $I_\eps^{(4)}$, which is a direct consequence of Jacobi's formula for the determinant and the divergence-freeness of the cofactor matrix by Piola's identity (see \cite[p.~117]{MarHug94}). Together with $\tb_\eps^{\mathrm{M}} \rvert_{S_{*,\eps}^\pm} \equiv 0$ we then obtain
	\begin{align}
		I_\eps^{(3)} + I_\eps^{(4)} 
		= - \frac 1 {2\eps} \int_{N_\eps} J_\eps^{\mathrm{M}} \tb_\eps^{\mathrm{M}} (\tu_{j\eps}^{\mathrm{M}})^2 \dd \calH^{n-1}(x).
	\end{align}
	As $\Cknorm{\tb_\eps^{\mathrm{M}}}{0}{[0,T]\times \overline{\OmeMS}}\lesssim \eps$ by \ref{it:T1:est_psi}, with the trace estimate in \autoref{lem:traces} we conclude
	\begin{align}
		|I_\eps^{(3)} + I_\eps^{(4)}|
		&\lesssim \frac 1{\theta\eps} \Lpnorm{\tu_{j\eps}^{\mathrm{M}}}{2}{\OmeMS}^2 + \theta \eps \Lpnorm{\nabla\tu_{j\eps}^{\mathrm{M}}}{2}{\OmeMS}^2.
	\end{align}
    Finally, we estimate the integrals involving the nonlinearities $f$, $g$, and $h$. Using \ref{it:T1:est_psi}, the growth condition \eqref{eq:growth_fgh} on $h_{j}$, $|N_\eps|\simeq 1$ and the trace estimate in \autoref{lem:traces}, we calculate
    \begin{align}
        |I_\eps^{(7)}|
        &\lesssim \int_{N_\eps} (1 + |\tu_\eps^{\mathrm{M}}|) |\tu_{j\eps}^{\mathrm{M}}|  \dd \calH^{n-1}(x)
        \lesssim |N_\eps| + \int_{N_\eps} |\tu_\eps^{\mathrm{M}}|^2 \dd \calH^{n-1}(x)\\
        &\lesssim 1 + \frac 1{\theta\eps} \Lpnorm{\tu_\eps^{\mathrm{M}}}{2}{\OmeMS}^2 + \theta \eps \Lpnorm{\nabla\tu_\eps^{\mathrm{M}}}{2}{\OmeMS}^2. \label{eq:est5}
    \end{align}
    Similarly, using the growth conditions \eqref{eq:growth_fgh} for $f_{j}$ and $g_{j}$ as well as $|\Omee^\pm| \simeq 1$ and $|\OmeMS| \simeq \eps$, we obtain
    \begin{align}
        |I_\eps^{(5)}|
        \lesssim 1 + \sumpm \Lpnorm{\tu_\eps^\pm}{2}{\Omee^\pm}^2, \qquad
        |I_\eps^{(6)}|
        \lesssim 1 + \frac 1\eps \Lpnorm{\tu_\eps^{\mathrm{M}}}{2}{\OmeMS}^2.
    \end{align}
    
    Collecting the previous estimates, choosing $\theta>0$ sufficiently small to absorb the gradient terms and summing over $j$, we finally obtain
    \begin{align}\label{eq:Gronwall_diff}
        &\frac{\dd}{\dd t} \norms{\sqrt{J_\eps}\tu_\eps}{\calL_\eps^m}^2 + \sumpm \Lpnorm{\nabla \tu_{\eps}^\pm}{2}{\Omee^\pm}^2 + \eps \Lpnorm{\nabla \tu_\eps^{\mathrm{M}}}{2}{\OmeMS}^2
        \lesssim 1 + \norms{\tu_\eps}{\calL_\eps^m}^2.
    \end{align}
    Integrating with respect to time, using $J_\eps \simeq 1$ and $\tilde{U}_{\eps,0} \in \calL_\eps^m$ with $\norms{\tilde{U}_{\eps,0}}{\calL_\eps^m}\lesssim 1$ by \ref{it:propU0}, Gronwall's inequality implies $\norms{J_\eps \tu_\eps}{\calL_\eps^m} \simeq \norms{\tu_\eps}{\calL_\eps^m}\lesssim 1$. As the latter estimate is independent of $t\in (0,T)$, we deduce \eqref{eq:apriori1} from \eqref{eq:Gronwall_diff}. Using \eqref{eq:est_J}, we further obtain the second estimate in \eqref{eq:apriori2} as
    \begin{align}
        &\sumpm\Lpnorm{\nabla \tu_\eps^\pm}{2}{\Omee^\pm}
        + \sqrt{\eps} \Lpnorm{\nabla (J_\eps^{\mathrm{M}} \tu_\eps^{\mathrm{M}})}{2}{\OmeMS}\\ 
        &\quad \leq \sumpm\Lpnorm{\nabla \tu_\eps^\pm}{2}{\Omee^\pm} 
        + \sqrt{\eps} \Lpnorm{J_\eps^{\mathrm{M}} \nabla \tu_\eps^{\mathrm{M}}}{2}{\OmeMS}
        + \sqrt{\eps} \Lpnorm{\nabla J_\eps^{\mathrm{M}} \tu_\eps^{\mathrm{M}}}{2}{\OmeMS}
        \lesssim \norms{\tu_\eps}{\calH_\eps^m}.
    \end{align}
    
    By a similar reasoning as before but using an arbitrary test function $\phi_\eps \in \calH_\eps^m$ instead of $\tu_\eps$, we obtain
    \begin{align}
        &|\dotproduct{\p_t(J_\eps \tu_{j\eps})}{\phi_{j\eps}}{\calH_\eps^\prime,\, \calH_\eps}|\\
        &\leq \sumpm\left|\scp{D_{j}^\pm \nabla \tu_{j\eps}^\pm}{\nabla \phi_{j\eps}}{L^2(\Omee^\pm)}\right|
        +\eps \left|\scp{J_\eps^{\mathrm{M}} \tD_{j\eps}^{\mathrm{M}} \nabla \tu_{j\eps}^{\mathrm{M}}}{\nabla \phi_{j\eps}}{L^2(\OmeMS)}\right|
        + \sumpm \left|\scp{\tu_{j\eps}^\pm q_{j}^\pm}{\nabla \phi_{j\eps}}{L^2(\Omee^\pm)}\right|\\
        &\quad + \left|\scp{J_\eps^{\mathrm{M}} \tu_{j\eps}^{\mathrm{M}} \tq_{j\eps}^{\mathrm{M}}}{\nabla \phi_{j\eps}}{L^2(\OmeMS)}\right|
		+ \frac 1\eps \left|\scp{J_\eps^{\mathrm{M}} b_\eps^{\mathrm{M}} \tu_{j\eps}^{\mathrm{M}}}{\nabla \phi_{j\eps}}{L^2(\OmeMS)}\right|
		+ \sumpm \left|\scp{f_{j}(\tu_{\eps}^\pm)}{\phi_{j\eps}}{L^2(\Omee^\pm)} \right|\\
		&\quad+ \frac 1\eps \left|\scp{J_\eps^{\mathrm{M}} g_{j}(\tu_\eps^{\mathrm{M}})}{\phi_{j\eps}}{L^2(\OmeMS)}\right|
		+ \left|\scp{J_\eps^{\mathrm{M}} \norm{F_\eps^{-\top} \nu_\eps^{\mathrm{M}}} h_{j}(\tu_\eps^{\mathrm{M}})}{\phi_{j\eps}}{L^2(N_\eps)} \right|\\
        &\lesssim \sumpm \left(\Lpnorm{\tu_{j\eps}^\pm}{2}{\Omee^\pm} + \Lpnorm{\nabla \tu_{j\eps}^\pm}{2}{\Omee^\pm}\right) \Lpnorm{\nabla \phi_{j\eps}}{2}{\Omee^\pm}
        + \eps \Lpnorm{\nabla \tu_{j\eps}^{\mathrm{M}}}{2}{\OmeMS} \Lpnorm{\nabla \phi_{j\eps}}{2}{\OmeMS}\\
        &\quad + \Lpnorm{\tu_{j\eps}^{\mathrm{M}}}{2}{\OmeMS} \Lpnorm{\nabla \phi_{j\eps}}{2}{\OmeMS}
        + \sumpm \left(1 + \Lpnorm{\tu_{\eps}^\pm}{2}{\Omee^\pm}\right) \Lpnorm{\phi_{j\eps}}{2}{\Omee^\pm}\\
        &\quad + \left(\frac 1 {\sqrt{\eps}} + \frac 1\eps \Lpnorm{\tu_\eps^{\mathrm{M}}}{2}{\OmeMS}\right) \Lpnorm{\phi_{j\eps}}{2}{\OmeMS}
        + \norms{\tu_{\eps}^{\mathrm{M}}}{(\calH_\eps^{\mathrm{M}})^m} \norms{\phi_{j\eps}}{\calH_\eps^{\mathrm{M}}}.
    \end{align}
    Summing over $j$, we deduce
    \begin{align}
        \norms{\p_t(J_\eps \tu_\eps)}{(\calH_\eps^m)^\prime} \lesssim 1 + \norms{\tu_\eps}{\calH_\eps^m},
    \end{align}
    and thereby \eqref{eq:apriori2} by the estimates for $\tu_\eps$ in \eqref{eq:apriori1} after integrating with respect to time.
\end{proof}
As discussed in \cite[Remark~2]{GahNeuPop21} for the case of bulk domains, under the given assumptions and by \eqref{eq:ptu} we have $\p_t \tu_\eps \in L^2(0,T;\calH_\eps^\prime)$ but cannot deduce a uniform bound for $\p_t \tu_\eps$ in $L^2(0,T;\calH_\eps^\prime)$ from this. However, further assuming $\Lpnorm{\p_t J_\eps}{\infty}{(0,T)\times \OmeMS} \lesssim 1$ we could improve this to obtain such a uniform bound.

Using the Galerkin method, estimates similar to that established in \autoref{lem:apriori} and a fixed-point argument, we obtain the existence of weak solutions.
\begin{lemma}[Existence of a weak solution]\label{lem:existence}
    There exists a unique weak solution $\tu_\eps$ of \eqref{eq:RDA_fd_full}.
\end{lemma}
\begin{proof}
    To prove existence of a weak solution $\tu_\eps$ of \eqref{eq:RDA_fd_full}, we rewrite \eqref{eq:RDA_fd_weak} as an equivalent equation for $w_\eps = J_\eps \tu_\eps$ using the spatial regularity of $J_\eps>0$. For all $\phi \in \calH_\eps^m$ and almost every $t\in (0,T)$ it reads
    \begin{align}\label{eq:weak_w}
	\begin{split}
		&\dotproduct{\p_t w_{j\eps}}{\phi_{j\eps}}{\calH_\eps^\prime,\,\calH_\eps} \\
		&\quad+ \sumpm \int_{\Omee^\pm} D_{j}^\pm \nabla w_{j\eps}^\pm \cdot \nabla \phi_{j\eps} \dd x
		+\eps\int_{\OmeMS} \tD_{j\eps}^{\mathrm{M}} \nabla w_{j\eps}^{\mathrm{M}} \cdot \nabla \phi_{j\eps} \dd x
		- \sumpm \int_{\Omee^\pm} w_{j\eps}^\pm q_{j}^\pm \cdot \nabla \phi_{j\eps} \dd x\\
		&\quad - \int_{\OmeMS} w_{j\eps}^{\mathrm{M}} \tq_{j\eps}^{\mathrm{M}} \cdot \nabla \phi_{j\eps} \dd x
		+ \frac 1\eps \int_{\OmeMS} w_{j\eps}^{\mathrm{M}} \tb_\eps^{\mathrm{M}} \cdot \nabla \phi_{j\eps} \dd x
        - \eps\int_{\OmeMS} w_{j\eps}^{\mathrm{M}} \tfrac{\tD_{j\eps}^{\mathrm{M}}}{J_\eps^{\mathrm{M}}}\nabla J_\eps^{\mathrm{M}} \cdot \nabla \phi_{j\eps} \dd x\\
		&= \sumpm \int_{\Omee^\pm} f_{j}(w_{\eps}^\pm)\phi_{j\eps} \dd x
		+ \frac 1\eps \int_{\OmeMS} J_\eps^{\mathrm{M}} g_{j}\left(\tfrac{w_{\eps}^{\mathrm{M}}}{J_\eps^{\mathrm{M}}}\right) \phi_{j\eps} \dd x
		- \int_{N_\eps} J_\eps^{\mathrm{M}} \norm{F_\eps^{-\top} \nu_\eps^{\mathrm{M}}} h_{j}\left(\tfrac{w_{\eps}^{\mathrm{M}}}{J_\eps^{\mathrm{M}}}\right) \phi_{j\eps} \dd \calH^{n-1}(x).
	\end{split}
    \end{align}
    To solve a linearised version of \eqref{eq:weak_w}, we may use a standard Galerkin scheme. When dealing with the initial problem for $\tu_\eps$, the coupling between $\tu_\eps$ and $J_\eps$ in the time derivative would make it necessary to deal with time-dependent basis functions in the scheme, which we overcome by introducing $w_\eps$.
    The existence of a weak solution $w_\eps$ to the problem \eqref{eq:weak_w} where the nonlinearities are independent of the $z$-variable, that is $f,\, g,\, h \in L^2(0,T)^m$, is obtained using a Galerkin scheme, where the compactness of the approximating sequence is due to estimates similar to those in \autoref{lem:apriori}. 
    The existence of a solution $w_\eps$ to the nonlinear problem \eqref{eq:weak_w} then follows using Schäfer's fixed point theorem. For this purpose, let $\beta \in (\frac 12,1)$ and $\calH_\eps^\beta$ defined in analogy to $\calH_\eps$.
    We consider the operator 
    \begin{align}
    	\calF \colon L^2(0,T;(\calH_\eps^\beta)^m) \to W^{1,2,2}(0,T;\calH_\eps^m,(\calH_\eps^m)^\prime), \qquad \calF(\bar{w}_\eps)=w_\eps,
    \end{align}
     where $w_\eps$ denotes the weak solution of \eqref{eq:weak_w} with nonlinearities evaluated in $\bar{w}_\eps$. By the growth conditions \eqref{eq:growth_fgh}, $J_\eps \simeq 1$ and the fact that the trace space of $H^{\beta}(\OmeMS)$ continuously embeds into $L^2(\p\OmeMS)$ if $\beta>\frac 12$, the operator $\calF$ is well-defined. Moreover, due to the continuity of the reaction rates in the last variable and by Pratt's theorem (see \cite{Pra60}), the evaluation of the nonlinearities is continuous with respect to $L^2(0,T;\calH_\eps^\beta)$-convergence, and hence by the compactness of the embedding $W^{1,2,2}(0,T;\calH_\eps,\calH_\eps^\prime) \xhookrightarrow{} L^2(0,T;\calH_\eps^\beta)$ (see \cite[Theorem~5.1, p.~58]{Lio69}) as well as the a priori estimates for the linear problem, we have that $\calF \colon L^2(0,T;(\calH_\eps^\beta)^m) \to L^2(0,T;(\calH_\eps^\beta)^m)$ is continuous and compact. Based on the a priori estimates in \autoref{lem:apriori}, the set
    \begin{align}
        \set{w_\eps \in L^2(0,T;(\calH_\eps^\beta)^m)}{w_\eps = \lambda \calF(w_\eps) \text{ for some } \lambda \in [0,1]}
    \end{align}
    is bounded as it only consists of weak solutions of \eqref{eq:weak_w} where the nonlinearities are multiplied by $\lambda$. Therefore, Sch\"afer's fixed-point theorem ensures the existence of a weak solution $w_\eps$ of \eqref{eq:weak_w}, which yields a weak solution $\tu_\eps = \frac{w_\eps}{J_\eps}$ of \eqref{eq:RDA_fd_full} in the sense of \autoref{def:weak_sol}. 
    
    If $\tu_\eps^{(1)}$ and $\tu_\eps^{(2)}$ are weak solutions of \eqref{eq:RDA_fd_full}, by testing \eqref{eq:RDA_fd_weak} with $\tu_\eps^{(1)}-\tu_\eps^{(2)}$, an argument similar to that in the proof of \autoref{lem:apriori}, based on the Lipschitz continuity of the nonlinearities, yields
    \begin{align}
        \frac{\dd}{\dd t} \norms{\tu_\eps^{(1)}-\tu_\eps^{(2)}}{\calL_\eps^m}^2 \lesssim \norms{\tu_\eps^{(1)}-\tu_\eps^{(2)}}{\calL_\eps^m}^2.
    \end{align}
    The uniqueness of the weak solution is then a consequence of Gronwall's inequality.
\end{proof}

\subsection{Estimates for the shifts of the microsolutions}\label{subsec:shifts}
In \autoref{subsec:convergence}, we aim to establish strong two-scale convergence of the microsolutions $\tu_\eps^{\mathrm{M}}$ to pass to the limit in the nonlinear reaction terms. Using \autoref{thm:strong_2s_cpctness}, this is achieved by controlling shifts with respect to $x^\prime \in \Sigma$ of the functions $\tu_\eps^{\mathrm{M}}$ and its gradients, for which we introduce the following terminology (see \cite{GahNeu21}). 
For $h\in (0,1)$, we write 
\begin{align}
    \Sigma_h \usesym{Sigmah} \coloneqq \set{x^\prime \in \Sigma}{\dist(x^\prime,\p\Sigma)>h}
\end{align}
but cannot work with the corresponding sets
\begin{align}
    \Ome_{\eps,h} \usesym{Omeeh} \coloneqq \Omee \cap (\Sigma_h \times (-H,H)) \qquad \text{and} \qquad 
    \Ome_{\eps,h}^\pm \usesym{Omeehpm} \coloneqq \Omee^\pm \cap \Ome_{\eps,h},
\end{align}
as the channels may intersect with the boundary of $\Sigma_h\times (-H,H)$. Hence, we possibly lose the Lipschitz-regularity of $\Ome_{\eps,h}$ and further introduce the sets
\begin{align}
    \widehat{\Sigma}_{\eps,h} \usesym{Sigmaeh} \coloneqq \interior{\bigcup_{k^\prime \in \calI_{\eps,h}}\eps(\overline{Y}+k^\prime)} \qquad \text{and} \qquad
    \OmeMSh \usesym{OmeMSh} \coloneqq \OmeMS \cap (\widehat{\Sigma}_{\eps,h} \times (-\eps,\eps)),
\end{align}
where $\calI_{\eps,h} = \set{k^\prime \in \IZ^{n-1}}{\eps(Y+k^\prime) \subset \Sigma_h}$. We denote the corresponding bulk regions by
\begin{align}
    \Omeh^+ \usesym{Omehpm} \coloneqq \Omee^+ \cap (\widehat{\Sigma}_{\eps,h} \times (\eps,H)) \qquad \text{and} \qquad
    \Omeh^-\coloneqq \Omee^- \cap (\widehat{\Sigma}_{\eps,h} \times (-H,-\eps)),
\end{align}
the top and bottom (resp.~lateral boundaries) of the channels in $\OmeMSh$ by $\widehat{S}_{\eps,*,h}^\pm\usesym{Ssehpm} $ (resp.~$\widehat{N}_{\eps,h} \usesym{Neh}$), and we finally write
\begin{align}
    \Omeh \usesym{hatOmeeh}= \interior{\Omeh^+ \cup \Omeh^- \cup \OmeMSh \cup \widehat{S}_{\eps,*,h}^+ \cup \widehat{S}_{\eps,*,h}^-}.
\end{align}
Moreover, for $l^\prime \in \IZ^{n-1}$, $U\subset \IR^n$ and $v_\eps \colon (0,T)\times U \to \IR$, we denote $v_\eps^{l^\prime} (t,x) \coloneqq v_\eps(t,x+(l^\prime,0)\eps)$,
extending $v_\eps$ by zero to $\IR^n$, and introduce the shifts
\begin{align}\label{eq:shifts}
    \delta_{l^\prime,\eps} v_\eps(t,x) \coloneqq v_\eps^{l^\prime}(t,x)- v_\eps(t,x).
\end{align}
As $\p_t \tu_\eps$ is only a functional on $\calH_\eps \simeq H^1(\Omee)$, it is a priori not clear how to obtain the time derivative of $\delta_{l^\prime,\eps} \tu_\eps$. For this purpose, we further introduce the spaces
\begin{align}
    \calL_{\eps,h} \usesym{Lepsh} \coloneqq L^2(\Omeh^+) \times L^2(\OmeMSh)\times L^2(\Omeh^-),
\end{align}
equipped with the inner product of $\calL_\eps$, and
\begin{align}
    \calH_{\eps,h,0} \usesym{Hepsh} \coloneqq \set{v_\eps \in H^1(\Omeh)}{v_\eps\rvert_{\p\widehat{\Sigma}_{\eps,h}\times (-H,-\eps)} = v_\eps\rvert_{\p\widehat{\Sigma}_{\eps,h}\times (\eps,H)} = 0},
\end{align}
equipped with the inner product of $\calH_\eps$, using that $\calH_{\eps,h,0} \xhookrightarrow{} \calH_\eps$ after extending functions defined on $\Omeh$ by zero to $\Omee$. Therefore, we have the Gelfand triple $\calH_{\eps,h,0} \xhookrightarrow{} \calL_{\eps,h} \xhookrightarrow{} \calH_{\eps,h,0}^\prime$, and for every $v_\eps \in L^2(0,T;\calH_\eps)$ with $\p_t v_\eps \in L^2(0,T;\calH_\eps^\prime)$ there holds $\p_t \delt v_\eps \in L^2(0,T; \calH_{\eps,h,0}^\prime)$ with
\begin{align}\label{eq:char_ptu_Hh0}
	\dotproduct{\p_t \delt v_\eps}{\phi_\eps}{(\calH_{\eps,h,0})^\prime,\, \calH_{\eps,h,0}} 
	= \dotproduct{\p_t v_\eps}{\delta_{-l^\prime,\eps}\phi_\eps}{\calH_{\eps}^\prime,\, \calH_{\eps}} \qquad \text{for all } \phi_\eps \in \calH_{\eps,h,0}.
\end{align}
Indeed, for every $\psi \in \Cci((0,T))$ and $\phi_\eps \in \calH_{\eps,h,0}$, we have $\delta_{-l^\prime,\eps} \phi_\eps \in \calH_\eps$ and thus
\begin{align}
	\int_0^T \dotproduct{\p_t v_\eps}{\delta_{-l^\prime,\eps} \phi_\eps}{\calH_\eps^\prime,\,\calH_\eps} \psi(t) \dd t
	&= - \int_0^T \scp{v_\eps}{\delta_{-l^\prime,\eps} \phi_\eps}{\calL_\eps} \psi^\prime(t) \dd t
	= - \int_0^T \scp{\delt v_\eps}{\phi_\eps}{\calL_{\eps,h}} \psi^\prime(t) \dd t;
\end{align}
hence \eqref{eq:char_ptu_Hh0} follows due to the $\calH_{\eps,h,0}$-continuity of the map $\phi_\eps \mapsto \dotproduct{\p_t v_\eps}{\delta_{-l,\eps} \phi_\eps}{\calH_\eps^\prime,\, \calH_\eps}$ and the uniqueness of the generalised time derivative.
The shifts of the microsolutions $\tu_\eps$ are then estimated as follows.
\begin{lemma}[Shifts of $J_\eps \tu_\eps$]\label{lem:shifts}
    For every $0<h \ll 1$ and $l^\prime \in \IZ^{n-1}$ with $|l^\prime \eps|\ll h$, there holds 
    \begin{align}
        \Lpnorm{\delt (J_\eps \tu_{j\eps})}{2}{0,T;\calH_{\eps,2h}}
        &\lesssim \norms{\delt (J_\eps(0) \tilde{U}_{j\eps,0})}{\calL_{\eps,h}}
        + \Lpnorm{\delt \tu_{j\eps}^\pm}{2}{0,T;\Omeh^\pm}
        + \eps + |l^\prime \eps| + |l^\prime \eps| \eps,
    \end{align}
    where the constant only depends on $h$, $T$ and the constants from \ref{it:T3:shifts}, \ref{it:shifts:D}, \ref{it:shifts:q} and \autoref{lem:apriori}.
\end{lemma}
\begin{proof} 
	To shorten notation, for measurable sets $E\subset \IR^n$, we write $\scp{\cdot}{\cdot}{E}$ instead of $\scp{\cdot}{\cdot}{L^2(E)}$ and, for $t\in (0,T)$, we write $\Lptnorm{\cdot}{p}{E} = \Lpnorm{\cdot}{p}{(0,t)\times E}$. Moreover, $C_{\theta,h}$ always denotes a generic constant depending only on $\theta, \, h >0$. Testing the weak formulation \eqref{eq:RDA_fd_weak} with $\delta_{-l^\prime,\eps} \phi_\eps$ for functions $\phi_\eps \in \calH_{\eps,h,0}^m$ and taking into account \eqref{eq:char_ptu_Hh0}, it is obtained that the shifts $\delt \tu_\eps$ satisfy 
    \begin{align}
		&\dotproduct{\p_t(\delt(J_\eps \tu_{j\eps}))}{\phi_{j\eps}}{(\calH_{\eps,h,0})^\prime,\,\calH_{\eps,h,0}} \\
		&\quad + \sumpm\scp{\delt(D_{j}^\pm \nabla \tu_{j\eps}^\pm)}{\nabla \phi_{j\eps}}{\Omeh^\pm}
        + \eps \scp{\delt(J_\eps^{\mathrm{M}} \tD_{j\eps}^{\mathrm{M}} \nabla \tu_{j\eps}^{\mathrm{M}})}{\nabla \phi_{j\eps}}{\OmeMSh}\\
		&\quad - \sumpm \scp{\delt(\tu_{j\eps}^\pm q_{j}^\pm)}{\nabla \phi_{j\eps}}{\Omeh^\pm}
        - \scp{\delt(J_\eps^{\mathrm{M}} \tu_{j\eps}^{\mathrm{M}} \tq_{j\eps}^{\mathrm{M}})}{\nabla \phi_{j\eps}}{\OmeMSh}
		+ \frac 1\eps \scp{\delt(J_\eps^{\mathrm{M}} \tu_{j\eps}^{\mathrm{M}} \tb_\eps^{\mathrm{M}})}{\nabla \phi_{j\eps}}{\OmeMSh} \\
		&= \sumpm \scp{\delta f_j(\tu_\eps^\pm)}{\phi_{j\eps}}{\Omeh^\pm}
		+ \frac 1\eps \scp{\delta \left(J_\eps^{\mathrm{M}} g_{j}(\tu_\eps^{\mathrm{M}}) \right)}{\phi_{j\eps}}{\OmeMSh}
		- \scp{\delta \left(J_\eps^{\mathrm{M}} \norm{F_\eps^{-\top} \nu_\eps^{\mathrm{M}}} h_{j}(\tu_\eps^{\mathrm{M}})\right)}{\phi_{j\eps}}{\widehat{N}_{\eps,h}},
    \end{align}
    where we have defined
    \begin{gather}
    	\delta f_j(\tu_\eps^\pm) 
    	\coloneqq f_j((\tu_\eps^\pm)^{l^\prime}) - f_j( \tu_\eps^\pm),\qquad
    	\delta (J_\eps^{\mathrm{M}} g_j(\tu_\eps^{\mathrm{M}})) 
    	\coloneqq J_\eps^{\mathrm{M},l^\prime} g_j((\tu_\eps^{\mathrm{M}})^{l^\prime}) - J_\eps^{\mathrm{M}} g_j(\tu_\eps^{\mathrm{M}}),\\
    	\delta \left(J_\eps^{\mathrm{M}} \norm{F_\eps^{-\top} \nu_\eps^{\mathrm{M}}} h_{j}(\tu_\eps^{\mathrm{M}})\right)
    	\coloneqq J_\eps^{\mathrm{M},l^\prime} \norm{F_\eps^{-\top} \nu_\eps^{\mathrm{M}}}^{l^\prime} h_{j}((\tu_\eps^{\mathrm{M}})^{l^\prime}) - J_\eps^{\mathrm{M}} \norm{F_\eps^{-\top} \nu_\eps^{\mathrm{M}}} h_{j}(\tu_\eps^{\mathrm{M}}).
    \end{gather}
    To derive the desired estimates, we write $w_\eps = J_\eps \tu_\eps$, use $J_\eps \nabla u_\eps = \nabla w_\eps - \frac{\nabla J_\eps}{J_\eps} w_\eps$ and $\delt a_\eps b_\eps = a_\eps^{l^\prime} (\delt b_\eps) + (\delt a_\eps) b_\eps$ which holds for arbitrary real-valued functions $a_\eps$ and $b_\eps$, so that we obtain from the previous equation after rearranging terms
    \begin{align}
		&\dotproduct{\p_t(\delt w_{j\eps})}{\phi_{j\eps}}{(\calH_{\eps,h,0})^\prime,\,\calH_{\eps,h,0}}
        + \sumpm \scp{D_{j}^{\pm,l^\prime} \delt \nabla w_{j\eps}^\pm}{\nabla \phi_{j\eps}}{\Omeh^\pm}
        + \eps \scp{\tD_{j\eps}^{\mathrm{M},l^\prime} \delt \nabla w_{j\eps}^{\mathrm{M}}}{\nabla \phi_{j\eps}}{\OmeMSh}\\
		&\quad= - \sumpm\scp{\delt D_{j}^\pm \, \nabla w_{j\eps}^\pm}{\nabla \phi_{j\eps}}{\Omeh^\pm}
		+ \sumpm \scp{w_{j\eps}^\pm \, \delt q_{j}^{\pm}}{\nabla \phi_{j\eps}}{\Omeh^\pm}
        + \sumpm \scp{\delt w_{j\eps}^\pm \, q_{j}^{\pm,l^\prime} }{\nabla \phi_{j\eps}}{\Omeh^\pm}\\
        &\qquad - \eps \scp{\delt \tD_{j\eps}^{\mathrm{M}} \, \nabla w_{j\eps}^{\mathrm{M}}}{\nabla \phi_{j\eps}}{\OmeMSh}
        + \eps \scp{ w_{j\eps}^{\mathrm{M}}\, \delt \left(\tD_{j\eps}^{\mathrm{M}}\tfrac{\nabla J_\eps^\mathrm{M}}{J_\eps^\mathrm{M}}\right)}{\nabla \phi_{j\eps}}{\OmeMSh}\\
        &\qquad + \eps \scp{\delt w_{j\eps}^{\mathrm{M}} \, \left(\tD_{j\eps}^{\mathrm{M}} \tfrac{\nabla J_\eps^{\mathrm{M}}}{J_\eps^{\mathrm{M}}}\right)^{l^\prime}}{\nabla \phi_{j\eps}}{\OmeMSh}
		+ \scp{w_{j\eps}^{\mathrm{M}} (\delt \tq_{j\eps}^{\mathrm{M}} - \tfrac 1\eps \delt \tb_\eps^{\mathrm{M}})}{\nabla \phi_{j\eps}}{\Omeh^\pm}\\
        &\qquad + \scp{\delt w_{j\eps}^{\mathrm{M}} \, (\tq_{j\eps}^{\mathrm{M},l^\prime} - \tfrac 1\eps \tb_\eps^{\mathrm{M},l^\prime})}{\nabla \phi_{j\eps}}{\OmeMSh}
		+ \sumpm \scp{\delta f_j(\tu_\eps^\pm)}{\phi_{j\eps}}{\Omeh^\pm}\\
		&\qquad+ \frac 1\eps \scp{\delta \left(J_\eps^{\mathrm{M}} g_{j}(\tu_\eps^{\mathrm{M}}) \right)}{\phi_{j\eps}}{\OmeMSh}
        - \scp{\delta \left(J_\eps^{\mathrm{M}} \norm{F_\eps^{-\top} \nu_\eps^{\mathrm{M}}} h_{j}(\tu_\eps^{\mathrm{M}})\right)}{\phi_{j\eps}}{\widehat{N}_{\eps,h}}\\
        &\quad \eqqcolon \sum_{i=1}^3 A_{i,\eps}^\pm + \sum_{i=1}^5 A_{i,\eps}^{\mathrm{M}} + A_{\eps}^{f} +A_{\eps}^{g} + A_{\eps}^{h}.
    \end{align}
    Due to the boundary values of $J_\eps \tu_\eps$, we cannot simply choose it as a test function in the previous formulation and introduce an additional cut-off function $\eta \in \Cci(\widehat{\Sigma}_{h})$ with $0\leq \eta \leq 1$, $\eta \equiv 1$ in $\widehat{\Sigma}_{2h}$ and $|\nabla \eta| \lesssim \frac 1 h$. We thus have $\eta^2 \delt w_\eps \in \calH_{\eps,h,0}$ with  
    \begin{align}
        \dotproduct{\p_t(\delt w_\eps)}{\eta^2 \delt w_\eps}{(\calH_{\eps,h,0})^\prime,\, \calH_{\eps,h,0}}
        = \frac 12 \frac{\dd}{\dd t} \norm{\eta \delt w_\eps}_{\calL_{\eps,h}}^2,
    \end{align}
    and taking $\phi_\eps = \eta^2 \delt w_\eps$ as a test function above, the terms on the left-hand side are estimated by
    \begin{align}
        &\dotproduct{\p_t(\delt w_{j\eps})}{\eta^2 \delt w_{j\eps}}{(\calH_{\eps,h,0})^\prime,\,\calH_{\eps,h,0}}
        + \sumpm \scp{D_{j}^{\pm,l^\prime} \delt \nabla w_{j\eps}^\pm}{\nabla (\eta^2 \delt w_{j\eps})}{\Omeh^\pm}\\
        &\qquad + \eps \scp{\tD_{j\eps}^{\mathrm{M},l^\prime} \delt \nabla w_{j\eps}^{\mathrm{M}}}{\nabla (\eta^2 \delt w_{j\eps})}{\OmeMSh}\\
        &\quad \gtrsim \frac{\dd}{\dd t} \norms{\eta \delt w_\eps}{\calL_{\eps,h}}^2 
        + \sumpm \Lpnorm{\eta \delt \nabla w_{j\eps}^\pm}{2}{\Omeh^\pm}^2
        + \sumpm \scp{D_{j}^{\pm,l^\prime} \delt \nabla w_{j\eps}^\pm}{\eta \nabla \eta \,  \delt w_{j\eps}^\pm}{\Omeh^\pm}\\
        &\qquad + \eps \Lpnorm{\eta \delt \nabla w_{j\eps}^{\mathrm{M}}}{2}{\OmeMSh}^2
        + \eps \scp{\tD_{j\eps}^{\mathrm{M},l^\prime} \delt \nabla w_{j\eps}^{\mathrm{M}}}{\eta \nabla \eta \,  \delt w_{j\eps}^{\mathrm{M}}}{\OmeMSh}.
    \end{align}
    In what follows, we repeatedly use the a priori estimates from \autoref{lem:apriori}, Young's inequality for $\theta \in (0,1)$ and the boundedness of the coefficients and of $F_\eps$, to estimate the latter scalar products and the terms $A_{i,\eps}^{\dagger}$, $A_{\eps}^{f}$, $A_{\eps}^{g}$, $A_{\eps}^{h}$. For the scalar products, we conclude
    \begin{align}
        &\sumpm \int_0^t\left|\scp{D_{j}^{\pm,l^\prime} \delt \nabla w_{j\eps}^\pm}{\eta \nabla \eta \, \delt w_{j\eps}^\pm}{\Omeh^\pm}\right| \dd \tau 
        + \eps  \int_0^t \left| \scp{\tD_{j\eps}^{\mathrm{M},l^\prime} \delt \nabla w_{j\eps}^{\mathrm{M}}}{\eta \nabla \eta \, \delt w_{j\eps}^{\mathrm{M}}}{\OmeMSh}\right|\dd \tau \\
        &\quad \lesssim \theta \Lptnorm{\eta \delt \nabla w_{j\eps}}{2}{\calL_{\eps,h}}^2 
        + \frac {1}{\theta h^2} \Lptnorm{\delt w_{j\eps}^\pm}{2}{\Omeh^\pm}^2 + \frac{\eps}{\theta h^2} \Lptnorm{\delt w_{j\eps}^{\mathrm{M}}}{2}{\OmeMSh}^2\\
        &\quad \lesssim \theta \Lptnorm{\eta \delt \nabla w_{j\eps}}{2}{\calL_{\eps,h}}^2 
        + C_{\theta,h}\Lptnorm{\delt w_{j\eps}^\pm}{2}{\Omeh^\pm}^2 
        + C_{\theta,h}\eps^2.
    \end{align}
    Next, we estimate the different terms $A_{i,\eps}^{\dagger}$ separately. First, we have
    \begin{align}
        \int_0^t |A_{1,\eps}^\pm| \dd \tau 
        & \lesssim \sumpm \theta \Lptnorm{\eta \delt \nabla w_{j\eps}^\pm}{2}{\Omeh^\pm}^2
        + \Lptnorm{\eta \delt w_{j\eps}^\pm}{2}{\Omeh^\pm}^2\\
        &\quad + \Big(\frac 1 \theta + \frac 1{h^2}\Big) \Lptnorm{\delt D_{j}^\pm}{\infty}{\Omeh^\pm}^2 \Lptnorm{\nabla w_{j\eps}^\pm}{2}{\Omeh^\pm}^2\\
        &\lesssim \sumpm \theta \Lptnorm{\eta \delt \nabla w_{j\eps}^\pm}{2}{\Omeh^\pm}^2
        + \Lptnorm{\eta \delt w_{j\eps}^\pm}{2}{\Omeh^\pm}^2
        + C_{\theta,h}\Lptnorm{\delt D_{j}^\pm}{\infty}{\Omeh^\pm}^2.
    \end{align}
    Similarly,
    \begin{align}
        \int_0^t |A_{2,\eps}^\pm| \dd \tau 
        & \lesssim \sumpm \theta \Lptnorm{\eta \delt \nabla w_{j\eps}^\pm}{2}{\Omeh^\pm}^2
        + \Lptnorm{\eta \delt w_{j\eps}^\pm}{2}{\Omeh^\pm}^2 \\
        &\quad + \Big(\frac 1 \theta + \frac 1{h^2}\Big) \Lptnorm{\delt q_j^\pm}{\infty}{\Omeh^\pm}^2 \Lptnorm{w_{j\eps}^\pm}{2}{\Omeh^\pm}^2\\
        &\lesssim \sumpm \theta \Lptnorm{\eta \delt \nabla w_{j\eps}^\pm}{2}{\Omeh^\pm}^2
        + \Lptnorm{\eta \delt w_{j\eps}^\pm}{2}{\Omeh^\pm}^2
        + C_{\theta,h}\Lptnorm{\delt q_j^\pm}{\infty}{\Omeh^\pm}^2 
    \end{align}
    and 
    \begin{align}
        \int_0^t |A_{3,\eps}^\pm| \dd \tau 
        & \lesssim \sumpm \theta \Lptnorm{\eta \delt \nabla w_{j\eps}^\pm}{2}{\Omeh^\pm}^2
        + \frac 1 \theta \Lptnorm{q_j^\pm}{\infty}{\Omee^\pm}^2 \Lptnorm{\eta \delt w_{j\eps}^\pm}{2}{\Omeh^\pm}^2 \\
        &\quad + \frac 1{h} \Lptnorm{q_j^\pm}{\infty}{\Omee^\pm} \Lptnorm{\delt w_{j\eps}^\pm}{2}{\Omeh^\pm}^2\\
        &\lesssim \sumpm \theta \Lptnorm{\eta \delt \nabla w_{j\eps}^\pm}{2}{\Omeh^\pm}^2
        + \frac 1 \theta \Lptnorm{\eta \delt w_{j\eps}^\pm}{2}{\Omeh^\pm}^2 
        + \frac 1{h} \Lptnorm{\delt w_{j\eps}^\pm}{2}{\Omeh^\pm}^2.
    \end{align}
    The proof of the estimates for the terms on $\OmeMSh$ proceeds similarly, but we have to carefully take into account the dependence on $\eps$ and use the assumptions \ref{it:T1:est_psi}--\ref{it:T4:2s_conv}. 
    We compute
    \begin{align}
        \int_0^t |A_{1,\eps}^{\mathrm{M}}| \dd \tau  
        & \lesssim \theta \eps \Lptnorm{\eta \delt \nabla w_{j\eps}^{\mathrm{M}}}{2}{\OmeMSh}^2
        + \frac 1\eps \Lptnorm{\eta \delt w_{j\eps}^{\mathrm{M}}}{2}{\OmeMSh}^2 \\
        &\quad + \Big(\frac \eps\theta + \frac {\eps^3}{h^2}\Big) \Lptnorm{\delt \tD_{j\eps}^{\mathrm{M}}}{\infty}{\OmeMSh}^2 \Lptnorm{\nabla w_{j\eps}^{\mathrm{M}}}{2}{\OmeMSh}^2\\
        &\lesssim \theta \eps \Lptnorm{\eta \delt \nabla w_{j\eps}^{\mathrm{M}}}{2}{\OmeMSh}^2
        +\frac 1\eps \Lptnorm{\eta \delt w_{j\eps}^{\mathrm{M}}}{2}{\OmeMSh}^2\\
        &\quad+ C_{\theta,h}(1+\eps^2)\Lptnorm{\delt \tD_{j\eps}^{\mathrm{M}}}{\infty}{\OmeMSh}^2.
    \end{align}
    Similarly,
    \begin{align}
        \int_0^t |A_{2,\eps}^{\mathrm{M}}| \dd \tau 
        & \lesssim \theta \eps \Lptnorm{\eta \delt \nabla w_{j\eps}^{\mathrm{M}}}{2}{\OmeMSh}^2
        + \frac 1\eps \Lptnorm{\eta \delt w_{j\eps}^{\mathrm{M}}}{2}{\OmeMSh}^2 \\
        &\quad + \Big(\frac \eps \theta + \frac {\eps^3}{h^2}\Big) \Lptnorm{\delt \left(\tD_{j\eps}^{\mathrm{M}} \tfrac{\nabla J_\eps^{\mathrm{M}}}{J_\eps^{\mathrm{M}}}\right)}{\infty}{\OmeMSh}^2 \Lptnorm{w_{j\eps}^{\mathrm{M}}}{2}{\OmeMSh}^2\\
        &\lesssim \theta \eps \Lptnorm{\eta \delt \nabla w_{j\eps}^{\mathrm{M}}}{2}{\OmeMSh}^2 + \frac 1\eps \Lptnorm{\eta \delt w_{j\eps}^{\mathrm{M}}}{2}{\OmeMSh}^2\\
        &\quad + C_{\theta,h}\eps^2(1+\eps^2)\Lptnorm{\delt \left(\tD_{j\eps}^{\mathrm{M}} \tfrac{\nabla J_\eps^{\mathrm{M}}}{J_\eps^{\mathrm{M}}}\right)}{\infty}{\OmeMSh}^2
    \end{align}
    and 
    \begin{align}
        \int_0^t |A_{3,\eps}^{\mathrm{M}}| \dd \tau
        & \lesssim \theta \eps \Lptnorm{\eta \delt \nabla w_{j\eps}^{\mathrm{M}}}{2}{\OmeMSh}^2
        + \frac \eps \theta \Lptnorm{\nabla J_\eps^{\mathrm{M}}}{\infty}{\OmeM}^2 \Lptnorm{\eta \delt w_{j\eps}^{\mathrm{M}}}{2}{\OmeMSh}^2 \\
        &\quad + \frac 1 \eps \Lptnorm{\eta \delt w_{j\eps}^{\mathrm{M}}}{2}{\OmeMSh}^2
        + \frac{\eps^3}{h^2} \Lptnorm{\nabla J_\eps^{\mathrm{M}}}{\infty}{\OmeM}^2 \Lptnorm{\delt w_{j\eps}^{\mathrm{M}}}{2}{\OmeMSh}^2\\
        &\lesssim \theta \eps \Lptnorm{\eta \delt \nabla w_{j\eps}^{\mathrm{M}}}{2}{\OmeMSh}^2
        + \frac 1{\theta\eps} \Lptnorm{\eta \delt w_{j\eps}^{\mathrm{M}}}{2}{\OmeMSh}^2 
        + \frac{\eps^2}{h^2}.
    \end{align}
    The estimates for $A_{4,\eps}^{\mathrm{M}}$ and $A_{5,\eps}^{\mathrm{M}}$ are obtained analogously and we have
    \begin{align}
        \int_0^t |A_{4,\eps}^{\mathrm{M}}| \dd \tau
        &\lesssim \theta \eps \Lptnorm{\eta \delt \nabla w_{j\eps}^{\mathrm{M}}}{2}{\OmeMSh}^2 
        + \frac 1\eps \Lptnorm{\eta \delt w_{j\eps}^{\mathrm{M}}}{2}{\OmeMSh}^2\\
        &\quad + \Big(\frac 1 {\theta\eps} + \frac {\eps}{h^2}\Big) \Big(\Lptnorm{\delt \tq_{j\eps}^{\mathrm{M}}}{\infty}{\OmeMSh}^2 + \frac 1{\eps^2}\Lptnorm{\delt \tb_\eps^{\mathrm{M}}}{\infty}{\OmeMSh}^2\Big) \Lptnorm{w_{j\eps}^{\mathrm{M}}}{2}{\OmeMSh}^2, \\
        &\lesssim \theta \eps \Lptnorm{\eta \delt \nabla w_{j\eps}^{\mathrm{M}}}{2}{\OmeMSh}^2 + \frac 1\eps \Lptnorm{\eta \delt w_{j\eps}^{\mathrm{M}}}{2}{\OmeMSh}^2\\
        &\quad + C_{\theta,h}(1+\eps^2)\Big(\Lptnorm{\delt \tq_{j\eps}^{\mathrm{M}}}{\infty}{\OmeMSh}^2 + \frac 1{\eps^2}\Lptnorm{\delt \tb_\eps^{\mathrm{M}}}{\infty}{\OmeMSh}^2\Big)
    \end{align}
    and 
    \begin{align}
        \int_0^t |A_{5,\eps}^{\mathrm{M}}| \dd \tau 
        &\lesssim \theta \eps \Lptnorm{\eta \delt \nabla w_{j\eps}^{\mathrm{M}}}{2}{\OmeMSh}^2 
        + \frac 1\eps \Lptnorm{\eta \delt w_{j\eps}^{\mathrm{M}}}{2}{\OmeMSh}\\
        &\quad + \frac 1{\theta \eps} \Big(\Lptnorm{\tq_{j\eps}^{\mathrm{M}}}{\infty}{\OmeMS}^2 + \frac 1 {\eps^2} \Lptnorm{\tb_\eps^{\mathrm{M}}}{\infty}{\OmeMS}^2\Big) \Lptnorm{\eta \delt w_{j\eps}^{\mathrm{M}}}{2}{\OmeMSh}^2\\
        &\quad 
        + \frac {\eps}{h^2} \Big(\Lptnorm{\tq_{j\eps}^{\mathrm{M}}}{\infty}{\OmeMS}^2 + \frac 1 {\eps^2} \Lptnorm{\tb_\eps^{\mathrm{M}}}{\infty}{\OmeMS}^2\Big) \Lptnorm{\delt w_{j\eps}^{\mathrm{M}}}{2}{\OmeMSh}^2\\
        &\lesssim \theta \eps \Lptnorm{\eta \delt \nabla w_{j\eps}^{\mathrm{M}}}{2}{\OmeMSh}^2 
        + \frac 1 {\theta \eps} \Lptnorm{\eta \delt w_{j\eps}^{\mathrm{M}}}{2}{\OmeMSh}^2
        +\frac{\eps^2}{h^2}.
    \end{align}
    It remains to give the estimates for the nonlinearities $f$, $g$, and $h$. Due to the uniform Lipschitz continuity of $f$ in the last argument and using $w_\eps = \tu_\eps$ on $\Omee^\pm$, we have
    \begin{align}
        \int_0^t |A_\eps^f| \dd \tau
        \leq \sumpm \int_0^t\int_{\Omeh^\pm} |(f_j\circ u_\eps^\pm)^{l^\prime} - f_j \circ u_\eps^\pm| \, \eta^2 |\delt w_{j\eps}^\pm| \dd x \dd \tau
        \lesssim \Lptnorm{\eta \delt w_{\eps}^\pm}{2}{\Omeh^\pm}^2.
    \end{align}
    Next, we rewrite $A_{\eps}^{g}$ according to
    \begin{align}
        A_{\eps}^{g} 
        &= \frac 1\eps \scp{J_\eps^{\mathrm{M}} \delta g_{j}(\tu_\eps^{\mathrm{M}})}{\eta^2 \delt w_{j\eps}^{\mathrm{M}}}{\OmeMSh} 
        + \frac 1 \eps \scp{\delt J_\eps^{\mathrm{M}} g_{j}(\tu_\eps^{\mathrm{M},l^\prime})}{\eta^2 \delt w_{j\eps}^{\mathrm{M}}}{\OmeMSh} 
        \eqqcolon A_{1,\eps}^{g} + A_{2,\eps}^{g}
    \end{align}
    and estimate both terms separately. Due to the uniform Lipschitz continuity of $g$ in the last argument, we have 
    \begin{align}
        \int_0^t |A_{1,\eps}^{g}| \dd \tau
        & \lesssim \frac 1 \eps \int_0^t\int_{\OmeMSh} |J_\eps^{\mathrm{M}} \tu_\eps^{\mathrm{M},l^\prime} - J_\eps^{\mathrm{M}} \tu_\eps^{\mathrm{M}}| \, \eta^2 |\delt w_{j\eps}^{\mathrm{M}}| \dd x \dd \tau \\
        &\leq \frac 1\eps \int_0^t\int_{\OmeMSh} \eta^2 |\delt w_\eps^{\mathrm{M}}|^2 + |\delt J_\eps^{\mathrm{M}}| |\tu_\eps^{\mathrm{M},l^\prime}| \, \eta^2 |\delt w_{j\eps}^{\mathrm{M}}| \dd x \dd \tau\\
        &\leq \frac 1 \eps \Lptnorm{\eta \delt w_{j\eps}^{\mathrm{M}}}{2}{\OmeMSh}^2 + \frac 1\eps \Lptnorm{\delt J_\eps^{\mathrm{M}}}{\infty}{\OmeMSh} \Lptnorm{\tu_\eps^{\mathrm{M}}}{2}{\OmeMSh} \Lptnorm{\eta \delt w_{j\eps}^{\mathrm{M}}}{2}{\OmeMSh}\\
        &\lesssim \frac 1 \eps \Lptnorm{\eta \delt w_{j\eps}^{\mathrm{M}}}{2}{\OmeMSh}^2 + \Lptnorm{\delt J_\eps^{\mathrm{M}}}{\infty}{\OmeMSh}^2  
    \end{align}
    and, using $|\OmeMSh|, \simeq \eps$, we calculate
    \begin{align}
        \int_0^t |A_{2,\eps}^{g}| \dd \tau 
        &\lesssim \frac 1\eps \int_0^t\int_{\OmeMSh} |\delt J_\eps^{\mathrm{M}}|(1+|\tu_\eps^{\mathrm{M},l^\prime}|) \eta^2 |\delt w_{j\eps}^{\mathrm{M}}| \dd x \dd \tau\\
        &\lesssim \frac 1\eps \Lptnorm{\delt J_\eps^{\mathrm{M}}}{\infty}{\OmeMSh} |\OmeMSh|^{\frac 12} \Lptnorm{\eta \delt w_{j\eps}^{\mathrm{M}}}{2}{\OmeMSh} \\
        &\quad+ \frac 1\eps \Lptnorm{\delt J_\eps^{\mathrm{M}}}{\infty}{\OmeMSh} \Lptnorm{\tu_\eps^{\mathrm{M},l^\prime}}{2}{\OmeMSh} \Lptnorm{\eta \delt w_{j\eps}^{\mathrm{M}}}{2}{\OmeMSh}\\
        &\lesssim \frac 1\eps \Lptnorm{\eta \delt w_{j\eps}^{\mathrm{M}}}{2}{\OmeMSh}^2
        + \Lptnorm{\delt J_\eps^{\mathrm{M}}}{\infty}{\OmeMSh}^2 .
    \end{align}
    The nonlinearity involving $h$ is treated in a similar way. We write
    \begin{align}
        A_{\eps}^{h} 
        &= \scp{J_\eps^{\mathrm{M}} \norm{F_\eps^{-\top} \nu_\eps^{\mathrm{M}}} \delta h_{j}(\tu_\eps^{\mathrm{M}})}{\eta^2 \delt w_{j\eps}^{\mathrm{M}}}{\widehat{N}_{\eps,h}} 
        + \scp{\delt \left(J_\eps^{\mathrm{M}} \norm{F_\eps^{-\top} \nu_\eps^{\mathrm{M}}} \right)h_{j}(\tu_\eps^{\mathrm{M},l^\prime})}{\eta^2 \delt w_{j\eps}^{\mathrm{M}}}{\widehat{N}_{\eps,h}} \\
        &\eqqcolon A_{1,\eps}^{h} + A_{2,\eps}^{h}
    \end{align}
    and use the uniform Lipschitz continuity of $h$ in the last variable together with the trace estimate in \autoref{lem:traces}, which in particular yields $\Lpnorm{\tu_\eps^{\mathrm{M}}}{2}{\widehat{N}_{\eps,h}}\lesssim 1$ by the a priori estimates from \autoref{lem:apriori}. We thus obtain 
    \begin{align}
        \int_0^t |A_{1,\eps}^{h}| \dd \tau
        &\lesssim \int_0^t \int_{\widehat{N}_{\eps,h}} \left|J_\eps^{\mathrm{M}} \norm{F_\eps^{-\top} \nu_\eps^{\mathrm{M}}}\, \tu_\eps^{\mathrm{M},l^\prime} - J_\eps^{\mathrm{M}} \norm{F_\eps^{-\top} \nu_\eps^{\mathrm{M}}}\, \tu_\eps^{\mathrm{M}}\right| \, \eta^2 |\delt w_{j\eps}^{\mathrm{M}}| \dd \calH^{n-1}(x) \dd \tau \\
        &\lesssim \int_0^t \int_{\widehat{N}_{\eps,h}} \eta^2 |\delt w_\eps^{\mathrm{M}}|^2 + |\delt J_\eps^{\mathrm{M}}| \, |\tu_\eps^{\mathrm{M},l^\prime}| \eta^2 |\delt w_{j\eps}^{\mathrm{M}}| \dd \calH^{n-1}(x) \dd \tau\\
        &\lesssim \Lptnorm{\eta \delt w_{j\eps}^{\mathrm{M}}}{2}{\widehat{N}_{\eps,h}}^2 
        + \Lptnorm{\delt J_\eps^{\mathrm{M}}}{\infty}{\widehat{N}_{\eps,h}} \Lptnorm{\tu_\eps^{\mathrm{M}}}{2}{\widehat{N}_{\eps,h}} \Lptnorm{\eta \delt w_{j\eps}^{\mathrm{M}}}{2}{\widehat{N}_{\eps,h}}\\
        &\lesssim \theta \eps \Lptnorm{\eta \delt \nabla  w_{j\eps}^{\mathrm{M}}}{2}{\OmeMSh}^2 
        + \frac{\theta \eps}{h^2} \Lptnorm{\delt w_{j\eps}^{\mathrm{M}}}{2}{\OmeMSh}^2 
        + \frac 1 {\theta \eps} \Lptnorm{\eta \delt w_{j\eps}^{\mathrm{M}}}{2}{\OmeMSh}^2\\
        &\quad + \Lptnorm{\delt J_\eps^{\mathrm{M}}}{\infty}{\widehat{N}_{\eps,h}}^2\\
        &\lesssim \theta \eps \Lptnorm{\eta \delt \nabla  w_{j\eps}^{\mathrm{M}}}{2}{\OmeMSh}^2 
        +\frac 1 {\theta\eps} \Lptnorm{\eta \delt w_{j\eps}^{\mathrm{M}}}{2}{\OmeMSh}^2 
        + \frac{\eps^2}{h^2}
        + \Lptnorm{\delt J_\eps^{\mathrm{M}}}{\infty}{\widehat{N}_{\eps,h}}^2.
    \end{align}
    Moreover, using $|\widehat{N}_{\eps,h}|\simeq 1$, the term $A_{2,\eps}^{h}$ can be estimated by
    \begin{align}
        \int_0^t |A_{2,\eps}^{h}| \dd \tau
        &\lesssim \int_0^t \int_{\widehat{N}_{\eps,h}} \left|\delt (J_\eps^{\mathrm{M}} \norm{F_\eps^{-\top} \nu_\eps^{\mathrm{M}}})\right|(1+|\tu_\eps^{\mathrm{M},l^\prime}|) \, \eta^2 |\delt w_{j\eps}^{\mathrm{M}}| \dd \calH^{n-1}(x) \dd \tau\\
        &\lesssim \Lptnorm{\delt (J_\eps^{\mathrm{M}} \norm{F_\eps^{-\top} \nu_\eps^{\mathrm{M}}})}{\infty}{\widehat{N}_{\eps,h}} |\widehat{N}_{\eps,h}|^{\frac 12} \Lptnorm{\eta \delt w_{j\eps}^{\mathrm{M}}}{2}{\widehat{N}_{\eps,h}} \\
        &\quad + \Lptnorm{\delt (J_\eps^{\mathrm{M}} \norm{F_\eps^{-\top} \nu_\eps^{\mathrm{M}}})}{\infty}{\widehat{N}_{\eps,h}} \Lptnorm{\tu_\eps^{\mathrm{M},l^\prime}}{2}{\widehat{N}_{\eps,h}} \Lptnorm{\eta \delt w_{j\eps}^{\mathrm{M}}}{2}{\widehat{N}_{\eps,h}}\\
        &\lesssim \theta \eps \Lptnorm{\eta \delt \nabla  w_{j\eps}^{\mathrm{M}}}{2}{\OmeMSh}^2 
		+ \frac{\theta \eps}{h^2} \Lptnorm{\delt w_{j\eps}^{\mathrm{M}}}{2}{\OmeMSh}^2 
		+ \frac 1 {\theta \eps} \Lptnorm{\eta \delt w_{j\eps}^{\mathrm{M}}}{2}{\OmeMSh}^2\\
		&\quad + \Lptnorm{\delt (J_\eps^{\mathrm{M}} \norm{F_\eps^{-\top} \nu_\eps^{\mathrm{M}}})}{\infty}{\widehat{N}_{\eps,h}}^2\\
        &\lesssim \theta \eps \Lptnorm{\eta \delt \nabla w_{j\eps}^{\mathrm{M}}}{2}{\OmeMSh}^2
        + \frac 1 {\theta \eps} \Lptnorm{\eta \delt w_{j\eps}^{\mathrm{M}}}{2}{\OmeMSh}^2
        + \frac{\eps^2}{h^2}\\
        &\quad + \Lptnorm{\norm{\delt F_\eps^{-\top}\nu_\eps^{\mathrm{M}}}}{\infty}{\widehat{N}_{\eps,h}}^2
        + \Lptnorm{\delt J_\eps^{\mathrm{M}}}{\infty}{\widehat{N}_{\eps,h}}^2,
    \end{align}
    where we have used 
    \begin{align}
        \Lptnorm{\delt (J_\eps^{\mathrm{M}} \norm{F_\eps^{-\top} \nu_\eps^{\mathrm{M}}})}{\infty}{\widehat{N}_{\eps,h}} \lesssim  \Lptnorm{\norm{\delt F_\eps^{-\top}\nu_\eps^{\mathrm{M}}}}{\infty}{\widehat{N}_{\eps,h}} + \Lptnorm{\delt J_\eps^{\mathrm{M}} \norm{F_\eps^{-\top}\nu_\eps^{\mathrm{M}}}}{\infty}{\widehat{N}_{\eps,h}}
    \end{align}
    in the last step.
    Altogether, writing $W_{\eps,0}=J_\eps(0) \tilde{U}_{\eps,0}$, for $\theta$ sufficiently small and almost every $t\in (0,T)$ we obtain
    \begin{align}
        &\norms{\eta \delt w_{j\eps}(t)}{\calL_{\eps,h}}^2 
        + \sumpm \Lptnorm{\eta \delt \nabla w_{j\eps}^\pm}{2}{\Omeh^\pm}^2 
        + \eps \Lptnorm{\eta \delt \nabla w_{j\eps}^{\mathrm{M}}}{2}{\OmeMSh}^2 \\
        & \lesssim \norms{\delt {W}_{j\eps,0}}{\calL_{\eps,h}}^2
        + \Lptnorm{\eta \delt w_{j\eps}}{2}{\calL_{\eps,h}}^2 
        + \eps^2\\
        &\quad + \left(\sumpm \Lptnorm{\delt w_{j\eps}^\pm}{2}{\Omeh^\pm}^2 
        + \Lptnorm{\delt D_{j}^\pm}{\infty}{\Omeh^\pm}^2
        + \Lptnorm{\delt q_j^\pm}{\infty}{\Omeh^\pm}^2\right)\\
        &\quad + (1+\eps^2)\left(\Lptnorm{\delt \tD_{j\eps}^{\mathrm{M}}}{\infty}{\Omeh}^2
        + \eps^2\Lptnorm{\delt \left(\tD_{j\eps}^{\mathrm{M}} \tfrac{\nabla J_\eps^{\mathrm{M}}}{J_\eps^{\mathrm{M}}}\right)}{\infty}{\OmeMSh}^2\right)\\
        &\quad +(1+\eps^2)\left(\Lptnorm{\delt \tq_{j\eps}^{\mathrm{M}}}{\infty}{\OmeMSh}^2 + \frac 1{\eps^2}\Lpnorm{\delt \tb_\eps^{\mathrm{M}}}{\infty}{0,T)\times \OmeMSh}^2\right)\\
        &\quad + \Lptnorm{\delt J_\eps^{\mathrm{M}}}{\infty}{\OmeMSh}^2
        + \Lptnorm{\delt J_\eps^{\mathrm{M}}}{\infty}{\widehat{N}_{\eps,h}}^2
        + \Lptnorm{\norm{\delt F_\eps^{-\top}\nu_\eps^{\mathrm{M}}}}{\infty}{\widehat{N}_{\eps,h}}^2.
    \end{align}
    For invertible matrices $A, \, B \in \IR^{n\times n}$, we have $A^{-1} - B^{-1} = - A^{-1} (A-B) B^{-1}$ and hence, by \ref{it:T1:est_psi} and \ref{it:T3:shifts}, also $\norm{\delt F_\eps^{-\top}}\lesssim |l^\prime \eps|$. Therefore, by \ref{it:shifts:D}--\ref{it:shifts:U0}, we conclude
    \begin{align}
        &\Lptnorm{\delt D_{j}^\pm}{\infty}{\Omeh^\pm} 
        + \Lptnorm{\delt \tD_{j\eps}^{\mathrm{M}}}{\infty}{\OmeMSh} 
        + \Lptnorm{\delt q_j^\pm}{\infty}{\Omeh^\pm} 
        + \Lptnorm{\delt \tq_{j\eps}^{\mathrm{M}}}{\infty}{\OmeMSh}\\
        &\quad + \Lptnorm{\delt J_\eps^{\mathrm{M}}}{\infty}{\OmeMSh}
        + \Lptnorm{\delt J_\eps^{\mathrm{M}}}{\infty}{\widehat{N}_{\eps,h}}
        + \Lptnorm{\norm{\delt F_\eps^{-\top}\nu_\eps^{\mathrm{M}}}}{\infty}{\widehat{N}_{\eps,h}}\\
        &\quad + \frac 1{\eps}\Lptnorm{\delt \tb_\eps^{\mathrm{M}}}{\infty}{\OmeMSh} 
        + \eps\Lptnorm{\delt \left(\tD_{j\eps}^{\mathrm{M}} \tfrac{\nabla J_\eps^{\mathrm{M}}}{J_\eps^{\mathrm{M}}}\right)}{\infty}{\OmeMSh}\lesssim |l^\prime\eps|,
    \end{align}
    where the estimates hold uniformly in $t\in (0,T)$.
	An application of Gronwall's inequality concludes the proof.
\end{proof}

A further ingredient for deriving strong compactness of the sequence of microsolutions $(\tu_\eps^\dagger)_\eps$, where $\dagger \in \{\pm,\, \mathrm{M}\}$, is a suitable control on the generalised time derivatives $\p_t\tu_\eps^\dagger$, see \autoref{thm:strong_2s_cpctness}. Following \cite[Section~2.2]{GahNeu21}, we next show how they are related to the time derivative $\p_t\tu_\eps \in L^2(0,T;\calH_\eps^\prime)$, as it is not obvious how to restrict a functional on $\calH_\eps$ to the subdomains $\Omee^\dagger$. For this purpose, we consider the spaces 
\begin{align}
    \calH_{\eps,0}^\pm \usesym{Heps0pm} = H^1_{S_{*,\eps}^\pm,0}(\Omee^\pm) \quad \text{and} \quad \calH_{\eps,0}^{\mathrm{M}} \usesym{Heps0M} = H^1_{S_{*,\eps}^+ \cup S_{*,\eps}^-,0}(\OmeMS),
\end{align}
each equipped with the scalar product of $\calH_\eps$ by considering $\calH_{\eps,0}^\dagger$ as a subspace of $\calH_\eps$ extending functions defined only on $\Omee^\pm$ or $\OmeMS$ to $\Omee$ by zero.
Then, $\calH_{\eps,0}^\dagger \xhookrightarrow{} L^2(\Omee^\dagger) \xhookrightarrow{} (\calH_{\eps,0}^\dagger)^\prime$ is a Gelfand triple and, for any function $v_\eps \in L^2(0,T;\calH_\eps)$ with $\p_t v_\eps \in L^2(0,T;\calH_\eps^\prime)$, we have $\p_t v_\eps^\dagger \in L^2(0,T;\calH_{\eps,0}^\dagger)$ with 
\begin{align}
	\begin{array}{rclcl}
		\dotproduct{\p_t v_\eps^\pm}{\phi_\eps^\pm}{(\calH_{\eps,0}^\pm)^\prime,\, \calH_{\eps,0}^\pm}\!
		&=&\! \dotproduct{\p_t v_\eps}{\phi_\eps^\pm}{\calH_{\eps}^\prime,\, \calH_{\eps}} && \text{for every } \phi_\eps^\pm \in \calH_{\eps,0}^\pm,\\
		\dotproduct{\p_t v_\eps^{\mathrm{M}}}{\phi_\eps^{\mathrm{M}}}{(\calH_{\eps,0}^{\mathrm{M}})^\prime,\, \calH_{\eps,0}^{\mathrm{M}}} \!
		&=&\! \eps \dotproduct{\p_t v_\eps}{\phi_\eps^{\mathrm{M}}}{\calH_{\eps}^\prime,\, \calH_{\eps}} && \text{for every } \phi_\eps^{\mathrm{M}} \in \calH_{\eps,0}^{\mathrm{M}}
	\end{array}
\end{align}
and, hence,
\begin{align}
	\norms{\p_t v_\eps^\pm}{(\calH_{\eps,0}^\pm)^\prime} \leq \norms{\p_t v _\eps}{\calH_{\eps}^\prime} \qquad \text{and} \qquad
	\norms{\p_t v_\eps^{\mathrm{M}}}{(\calH_{\eps,0}^{\mathrm{M}})^\prime} \leq \eps \norms{\p_t v _\eps}{\calH_{\eps}^\prime}.
	\label{eq:pt_u_pmM}
\end{align}
This is indeed a direct consequence of the identities
\begin{align}
    \int_0^T \dotproduct{\p_t v_\eps}{\phi_\eps^\pm}{\calH_\eps^\prime,\,\calH_\eps} \psi(t) \dd t 
    &= - \int_0^T \scp{v_\eps}{\phi_\eps^\pm}{\calL_\eps} \psi^\prime(t) \dd t 
    = - \int_0^T \scp{v_\eps^\pm}{\phi_\eps^\pm}{L^2(\Omee^\pm)} \psi^\prime(t) \dd t,\\
    \int_0^T \dotproduct{\p_t v_\eps}{\phi_\eps^{\mathrm{M}}}{\calH_\eps^\prime,\,\calH_\eps} \psi(t) \dd t 
    &= - \int_0^T \scp{v_\eps}{\phi_\eps^{\mathrm{M}}}{\calL_\eps} \psi^\prime(t) \dd t 
    = - \frac 1\eps \int_0^T \scp{v_\eps^{\mathrm{M}}}{\phi_\eps^{\mathrm{M}}}{L^2(\OmeMS)} \psi^\prime(t) \dd t,
\end{align}
where $\phi_\eps^\dagger \in \calH_{\eps,0}^\dagger$ and $\psi \in \Cci((0,T))$, such as the uniqueness of the generalised time derivative and the $\calH_{\eps,0}^\dagger$-continuity of the maps $\phi_\eps^\dagger \mapsto \dotproduct{\p_t v_\eps}{\phi_\eps^\dagger}{\calH_\eps^\prime,\, \calH_\eps}$.

\section{Limit passage and identification of the macroscopic model}\label{sec:macro_prob}
In this section, we use the estimates obtained in \autoref{sec:existence} to pass to the limit $\eps \to 0$ in the microscopic model \eqref{eq:RDA_fd_weak}. In a first step, we construct a suitable limit function $\tu_0$. Then, we discuss the effective model satisfied by the limit function, both on the reference and the time-dependent domain.

\subsection{Convergence results for the reference problem}\label{subsec:convergence}
In what follows, we gather the convergence properties of the sequences $(\tu_\eps^\pm)_\eps$ and $(\tu_\eps^{\mathrm{M}})_\eps$. Similar results were obtained in \cite{NeuJae07} and \cite{GahNeu21}, both in case of a static microstructure; hence, we pay particular attention when dealing with the microsolutions $\tu_\eps^{\mathrm{M}}$.

In order to keep track of the microscopic structure of the layer in the limit $\eps \to 0$, we use the notion of two-scale convergence introduced in \cite{Ngu89, All92}. The dimension-reduction due to the shrinking of $\OmeMS$ to the $(n-1)$-dimensional hyperplane $\Sigma$ is taken into account by including prefactors in $\eps$ in a suitable way, see \cite{GahNeu21, NeuJae07, MarMar00}. We denote by $C^0_\#(\overline{Z_*}) \usesym{C0Zsper}$ the space of continuous $(Y,0)$-periodic functions on $\IR^{n-1}\times [-1,1]$ restricted to $\overline{Z_*}$. The space $C^0_\#(N)$ is defined in an analogous way.
\begin{definition}[Two-scale convergence in thin layers]\label{def:2s_conv}
	Let $p, \, p^\prime \in [1,\infty)$ such that $\frac 1p + \frac 1{p^\prime} = 1$. 
	\begin{enumerate}[label=(\roman*)]
		\item A sequence of functions $v_\eps\in L^p((0,T)\times \OmeMS)$ \underline{\em converges in the (weak) two-scale sense (on $\OmeMS$)} to $v_0 \in L^p((0,T)\times \Sigma \times Z_*)$ if for all $\phi \in L^{p^\prime}((0,T)\times \Sigma; C^0_\#(\overline{Z_*})))$ there holds
		\begin{align}
			\lim_{\eps \to 0} \frac 1 \eps \int_0^T \int_{\OmeMS} v_\eps(t,x) \phi\left(t,x^\prime,\tfrac{x}{\eps}\right) \dd x \dd t
			= \int_0^T \int_\Sigma \int_{Z_*} v_0(t,x^\prime,z) \phi(t,x^\prime,z) \dd z \dd x^\prime \dd t.
		\end{align}
		The sequence $(v_\eps)_\eps$ \ul{\em converges in the strong two-scale sense (on $\OmeMS$)} to $v_0$ if, additionally,
		\begin{align}
			\lim_{\eps\to 0} \eps^{-\frac1p} \Lpnorm{v_\eps}{p}{(0,T)\times \OmeMS} = \Lpnorm{v_0}{p}{(0,T)\times \Sigma \times Z_*}.
		\end{align}
		\item A sequence of functions $v_\eps\in L^p((0,T)\times N_\eps)$ \ul{\em converges in the (weak) two-scale sense on $N_\eps$} to $v_0 \in L^p((0,T)\times \Sigma \times N)$ if for all $\phi \in L^{p^\prime}(0,T;C^0(\overline{\Sigma}; C^0_\#(\overline{N})))$ there holds 
		\begin{align}
			\lim_{\eps \to 0} \int_0^T \int_{N_\eps} v_\eps(t,x) \phi\left(t,x^\prime,\tfrac{x}{\eps}\right) \dd \calH^{n-1}(x) \dd t
			= \int_0^T \int_\Sigma \int_{N} v_0(t,x^\prime,z) \phi(t,x^\prime,z) \dd \calH^{n-1}(z) \dd x^\prime \dd t.
		\end{align}
		The sequence $(v_\eps)_\eps$ \ul{\em converges in the strong two-scale sense on $N_\eps$} to $v_0$ if, additionally,
		\begin{align}
			\lim_{\eps\to 0} \Lpnorm{v_\eps}{p}{(0,T)\times N_\eps} = \Lpnorm{v_0}{p}{(0,T)\times \Sigma \times N}.
		\end{align}
	\end{enumerate}
\end{definition}

A notion of two-scale convergence on time-dependent bulk domains was introduced in \cite{Wie22}, which coincides with the standard one on periodic bulk domains under suitable assumptions on the transformation $\psi_\eps$. The next remark briefly discusses how a definition of two-scale convergence on the thin (time-dependent) layer $\OmeMS(t)$, which is compatible with that of \autoref{def:2s_conv}, can be given in the same spirit.
\begin{rem}[Two-scale convergence on thin (time-dependent) layers]\label{rem:gen_2s_conv}
	Let $G_\eps(t)\in \{\OmeMS,\, \OmeMS(t)\}$ and $G(t,x^\prime) \in \{Z_*,\, Z_*(t,x^\prime)\}$, where $Z_*(t,x^\prime) \coloneqq \psi_0^{-1}(t,x^\prime, Z_*)$. 
	Given a sequence of functions ${v_\eps \in L^p(0,T;G_\eps(t))}$ and $v_0\in L^p((0,T)\times \Sigma \times Z)$, we say that $(v_\eps)_\eps$ two-scale converges to $v_0$ if for all $\phi \in L^{p^\prime}((0,T)\times \Sigma; C^0_{\#}(\overline{Z}))$ there holds
	\begin{align}\label{eq:gen_2s_conv}
		\lim_{\eps \to 0} \frac 1 \eps \int_0^T \int_{\OmeM} \bar v_\eps(t,x) \phi\left(t,x^\prime,\tfrac{x}{\eps}\right) \dd x \dd t 
		= \int_0^T \int_\Sigma \int_Z v_0(t,x^\prime, z) \phi(t,x^\prime,z) \dd z \dd x^\prime \dd t,
	\end{align} 
	where $\bar v_\eps$ denotes the extension of $v_\eps$ to $\OmeM$ by zero. By a suitable choice of test functions, we may then consider $v_0$ as a function in $L^p((0,T)\times \Sigma;G(t,x^\prime))$. Consequently, this notion of two-scale convergence for functions naturally induces the notion of two-scale convergence given in \autoref{def:2s_conv}.
	
	On the other hand, it can be shown that a sequence of functions $v_\eps \in L^p(0,T;\OmeMS(t))$ converges to $v_0\in L^p((0,T)\times \Sigma;Z_*(t,x^\prime))$ in the two-scale sense of \eqref{eq:gen_2s_conv} if and only if the sequence of transformed functions $\tilde v_\eps \coloneqq v_\eps \circ_x \psi_\eps$ converges to $\tilde v_0 \coloneqq v_0 \circ_z \psi_0$ in the two-scale sense as given by \autoref{def:2s_conv}. This is achieved by adapting the arguments given in \cite[Section~3]{Wie22} for the case of (porous) bulk domains as the proofs are mainly based on unfolding techniques, and the unfolding operator introduced in \autoref{def:unfolding} enjoys similar properties. At this point, the two-scale convergence $\frac 1\eps (\psi_\eps - \mathrm{id}_{\OmeMS}) \to \psi_0 - \mathrm{id}_{Z_*}$ is of crucial importance which is our reason to include it in the assumptions on the transformation $\psi_\eps$.
\end{rem}

The main compactness results we use in what follows are given in \autoref{thm:convergence} below, which is an extension of \autoref{lem:2s_cpct} (weak two-scale compactness) to functions defined on the whole microscopic domain $\Omee$ and \autoref{thm:strong_2s_cpctness} (strong two-scale compactness). We write $H^1_{\pm,0}(Z_*) \usesym{H1pm0} = H^1_{S_*^+ \cup S_*^-,0}(Z_*)$ for the sake of clearer notation and first introduce the space $\calH$ as
\begin{align}
	\calH \usesym{H}\coloneqq \set{v=(v^+,v^{\mathrm{M}},v^-) \in H^1(\Ome^+)\times L^2(\Sigma;H^1(Z_*)) \times H^1(\Ome^-)}{v^{\mathrm{M}} = v^\pm(\cdot_{x^\prime},0) \text{ on } \Sigma\times S_*^\pm},
\end{align}
equipped with the inner product
\begin{align}\label{eq:scp_H}
	\scp{v}{w}{\calH} \coloneqq \sumpm \scp{v^\pm}{w^\pm}{L^2(\Ome^\pm)} + \scp{v^{\mathrm{M}}}{w^{\mathrm{M}}}{L^2(\Sigma \times Z_*)} + \scp{\nabla_z v^{\mathrm{M}}}{\nabla_z w^{\mathrm{M}}}{L^2(\Sigma \times Z_*)}.
\end{align}
We note that the transmission condition on $\Sigma \times S_*^\pm$ implies that $v^{\mathrm{M}}\rvert_{S_*^\pm}(\cdot_t,\cdot_{x^\prime},z)$ is independent of $z\in S_*^\pm$.
Moreover, we have the Gelfand triple $\calH \xhookrightarrow{} \calL \xhookrightarrow{} \calH^\prime$, where $\calL \usesym{L} \coloneqq L^2(\Ome^+) \times L^2(\Sigma \times Z_*) \times L^2(\Ome^-)$ is eqipped with the usual scalar product.
In \autoref{thm:convergence} below, it turns out that the space $\calH$ is the appropriate choice to capture properties of the limit of the sequence of microsolutions $(\tu_\eps)_\eps$.
We note that for functions $v \in \calH$ we have $v^{\mathrm{M}} = v^\pm(\cdot_{x^\prime},0)$ on $\Sigma\times S_*^\pm$, which is the limit analogue of the transmission conditions $v_\eps^\pm = v_\eps^{\mathrm{M}}$ on $S_{*,\eps}^\pm$ for functions $v_\eps \in \calH_\eps$. 
As for functions $\phi \in \calH$ the evaluation $\phi(x,\tfrac{x^\prime}\eps)$ is not well-defined for $x \in \OmeMS \cup N_\eps$, we further consider the subspace of smooth functions
\begin{align}
	\calH^\infty \coloneqq \left[C^\infty(\overline{\Ome^+}) \times \Cci(\Sigma; C^\infty(\overline{Z_*})) \times C^\infty (\overline{\Ome^-})\right] \cap \calH.
\end{align}
Then, the space $\calH^\infty$ is dense in $\calH$ with respect to the norm induced by \eqref{eq:scp_H}, see \cite[Proposition~4.3]{GahNeu21}.

\begin{theorem}[Compactness]\label{thm:convergence}
	Let $v_\eps \in L^2(0,T;\calH_\eps)$ with $\p_t v_\eps \in L^2(0,T;\calH_\eps^\prime)$ be functions with
	\begin{align}\label{eq:apriori}
		\Lpnorm{\p_t v_\eps}{2}{0,T;\calH_\eps^\prime} + \Lpnorm{v_\eps}{2}{0,T;\calH_\eps} \lesssim 1.
	\end{align}
	Then, there exists $v_0=(v_0^+,v_0^{\mathrm{M}},v_0^-) \in L^2(0,T;\calH)$ with $\p_t v_0 \in L^2(0,T;\calH^\prime)$ such that the following properties hold true.
	\begin{enumerate}[label=(\roman*)]
		\item\label{it:conv_bulk} \underline{\em Convergence in the bulk.} The limits in the bulks satisfy 
        \begin{align}
            v_0^\pm \in L^2((0,T);H^1(\Ome^\pm)) \quad \text{with} \quad 
            \p_t v_0^\pm \in L^2(0,T;(H^1_{\Sigma,0}(\Ome^\pm))^\prime),
        \end{align} and, for a subsequence, we have
		\begin{align}
			\begin{array}{rclcl}
				\chi_{\Omee^\pm} v_{\eps}^\pm &\to& v_{0}^\pm && \text{strongly in } L^2((0,T)\times \Ome^\pm),
				\\
				\chi_{\Omee^\pm} \nabla v_{\eps}^\pm &\to& \nabla v_{0}^\pm && \text{weakly in } L^2((0,T)\times \Ome^\pm),
				\\
				v_{\eps}^\pm(\cdot_t,\cdot_{x^\prime},\pm \eps) &\to& v_{0}^\pm(\cdot_t,\cdot_{x^\prime},0) && \text{strongly in } L^2((0,T)\times \Sigma).
			\end{array}
		\end{align}
		\item\label{it:conv_layer} \underline{\em Convergence in the layer.}
        The limit in the layer satisfies 
        \begin{align}
            v_0^{\mathrm{M}} \in L^2(0,T;L^2(\Sigma;H^1(Z_*))) \quad \text{with} \quad 
            \p_t v_0^{\mathrm{M}} \in L^2(0,T;L^2(\Sigma;(H^1_{\pm,0}(Z_*))^\prime))
        \end{align}
        and, for a subsequence, we have
		\begin{align}
			\begin{array}{rclcl}
				v_{\eps}^{\mathrm{M}} &\to& v_{0}^{\mathrm{M}} && \text{in the two-scale sense},\\
				\eps \nabla v_{\eps}^{\mathrm{M}} &\to& \nabla_z v_{0}^{\mathrm{M}} &&\text{in the two-scale sense}.
			\end{array}
		\end{align}
		\item\label{it:conv_pt} \underline{\em Convergence of the time derivative.} For any $\phi \in \Cci([0,T);\calH^\infty)$, writing
		\begin{align}\label{eq:spec_test_fct}
			\phi_\eps(t,x) \coloneqq 
			\left\{\begin{array}{rclcl}
				\phi^\pm(t,x^\prime,x_n\mp\eps) && \text{for } x \in \Omee^\pm,\\
				\phi^{\mathrm{M}}\left(t,x^\prime,\tfrac{x}\eps\right)&& \text{for } x \in \OmeMS,\\
			\end{array}\right.
		\end{align}
		we have 
		\begin{subequations}
			\begin{align}\label{eq:conv_ptv}
				\lim_{\eps \to 0} \int_0^T \dotproduct{\p_t v_\eps}{\phi_{\eps}}{\calH_\eps^\prime,\, \calH_\eps} \dd t 
				&= \int_0^T \dotproduct{\p_t v_0}{\phi}{\calH^\prime,\, \calH} \dd t.
			\end{align}
			In particular, $v_0 \in C^0([0,T];\calL)$ and 
			\begin{align}\label{eq:conv_v0}
				\scp{v_\eps(0)}{\phi_{\eps}(0)}{\calL_\eps} 
				&\xrightarrow{\eps\to} \scp{v_0(0)}{\phi(0)}{\calL}.
			\end{align}
		\end{subequations}
	\end{enumerate}
\end{theorem}
\begin{proof}
	We split the proof in several steps and first show \ref{it:conv_bulk} and \ref{it:conv_layer}. From this, we conclude the coupling condition between $v_0^\pm$ and $v_0^{\mathrm{M}}$ incorporated in the definition of the space $\calH$ and, thus, $v_0 \in L^2(0,T;\calH)$. Finally, the proof of \ref{it:conv_pt} is based on the convergence properties obtained in \ref{it:conv_bulk} and \ref{it:conv_layer}.
	
	To show \ref{it:conv_bulk}, the arguments given in \cite[Section~5.1]{NeuJae07} can be followed in general;  however, there it is assumed that $\p_t v_\eps^\pm \in L^2((0,T)\times \Omee^\pm)$. In our case, by \eqref{eq:pt_u_pmM}, we only have $\p_t v_\eps^\pm \in L^2(0,T;(\calH_{\eps,0}^\pm)^\prime)$, but the used change of variables and application of the Aubin--Lions lemma can easily be adapted to this setting (also see \cite[Proposition~4.1]{GahNeu21}).
	
	Next, \ref{it:conv_layer} is a direct consequence of \eqref{eq:pt_u_pmM} and the a priori estimates for $v_\eps^{\mathrm{M}}$ obtained from \eqref{eq:apriori} together with the compactness result in \autoref{lem:2s_cpct}~\ref{it:2s_cpct_H1}, where the regularity of $\p_tv_0^{\mathrm{M}}$ is due to \autoref{lem:2s_cpct}~\ref{it:2s_cpct_pt}. 
	
	To show $v_0 \in L^2(0,T;\calH)$, it remains to establish the coupling condition $v_0^{\mathrm{M}}(t,x^\prime,z) = v_0^\pm(t,x^\prime,0)$ for almost every $(t,x^\prime,z)\in (0,T)\times \Sigma \times S_*^\pm$.
	We argue similarly as in \cite[Theorem~4.2]{GahNeu21} (also see \cite[Section~5.3]{NeuJae07}). By the strong convergence of the traces from \ref{it:conv_bulk}, we also have $v_{\eps}^\pm(\cdot_t,\cdot_{x^\prime},\pm\eps) \to v_{0}^\pm(\cdot_t,\cdot_{x^\prime},0)$ in the two-scale sense with respect to test functions in $L^2((0,T)\times \Sigma; C^\infty_\per(\overline{\Sigma_*}))$, where $\Sigma_*\coloneqq \Sigma \cap Z_*$, see \cite[p.~1487]{All92}.
	Then, for all $\phi \in C^\infty((0,T)\times \Sigma; C^\infty_\#(\overline{Z_*}))$ with compact support with respect to $z\in Z_* \cup S_*^+ \cup S_+^-$, by \ref{it:conv_layer}, using the transmission condition $v_\eps^\pm = v_\eps^{\mathrm{M}}$ on $(0,T)\times S_{*,\eps}^\pm$ and the previous observation, we have
	\begin{align}
		&\int_0^T \int_\Sigma \int_{Z_*} \nabla_z v_{0}^{\mathrm{M}}(t,x^\prime,z) \phi(t,x^\prime,z) \dd z \dd x^\prime \dd t\\
		&\quad =\lim_{\eps \to 0} \frac 1 \eps \int_0^T \int_{\OmeMS} \eps \nabla v_{\eps}^{\mathrm{M}}(t,x) \phi\left(t,x^\prime, \tfrac{x}{\eps}\right) \dd x \dd t\\
		&\quad =\lim_{\eps \to 0} \frac 1 \eps \left[- \int_0^T \int_{\OmeMS} \left(\eps v_{\eps}^{\mathrm{M}}(t,x) \nabla_{x^\prime}\phi\left(t,x^\prime, \tfrac{x}{\eps}\right) 
		+ v_{\eps}^{\mathrm{M}}(t,x) \nabla_{z}\phi\left(t,x^\prime, \tfrac{x}{\eps}\right)\right) \dd x \dd t \right.\\
		&\qquad \phantom{\lim_{\eps \to 0} \frac 1 \eps} + \sumpm\left. \int_0^T \int_{S_{*,\eps}^\pm} \eps v_{\eps}^{\mathrm{M}}(t,x) \phi\left(t,x^\prime, \tfrac{x}{\eps}\right) \nu_\eps^{\mathrm{M}} \dd \calH^{n-1}(x) \dd t \right]\\
		&\quad = -\int_0^T \int_\Sigma\int_{Z_*} v_{0}^{\mathrm{M}}(t,x^\prime, z) \nabla_z \phi(t,x^\prime,z) \dd z \dd x^\prime \dd t \\
		&\qquad + \sumpm \int_0^T \int_\Sigma \int_{S_*^\pm} v_{0}^\pm(t,x^\prime,0) \phi(t,x^\prime,z) n_\pm \dd \calH^{n-1}(z) \dd x^\prime \dd t,
	\end{align}
	where $n_\pm=(0^\prime, \pm1)\in \IR^n$. By the divergence theorem together with the choice of the support of $\phi$, we thus have
    \begin{align}
        \sumpm \int_0^T \int_\Sigma \int_{S_*^\pm} \left(v_{0}^{\mathrm{M}}(t,x^\prime,z)  - v_{0}^\pm(t,x^\prime,0)\right)\phi(t,x^\prime,z) n_\pm \dd \calH^{n-1}(z) \dd x^\prime \dd t 
        = 0,
    \end{align}
    which after applying the fundamental lemma of the calculus of variations yields the desired result.
	
	Finally, to obtain \ref{it:conv_pt}, we argue similarly as in \cite[Lemma~B.5]{GahNeu25} (also see \cite[Proposition~4]{GahNeuPop21} for the case of bulk regions) but have to consider functions defined on the whole domain $\Omee$ instead. 
	Denoting by $\delta_t^h$ difference quotients in time, that is
	\begin{align}
		\delta_t^h v_0^\pm &\coloneqq \frac{v_0^\pm(\cdot_t+h,\cdot_x)-v_0^\pm}{h} \qquad \text{in } (0,T-h)\times \Ome^\pm, \\
		\delta_t^h v_0^{\mathrm{M}} &\coloneqq \frac{v_0^{\mathrm{M}}(\cdot_t+h,\cdot_{x^\prime},\cdot_z)-v_0^{\mathrm{M}}}{h} \qquad \text{in } (0,T-h)\times \Sigma\times Z_*, 
	\end{align}
	by the reflexivity of $L^2(0,T;(\calH^m)^\prime)$ it suffices to bound $\Lpnorm{\delta_t^h v_0}{2}{0,T-h; \calH^\prime}$ uniformly in $h$ to conclude $\p_t v_0 \in L^2(0,T;\calH^\prime)$ (see \cite[Proposition~2.2.26]{GasPap05}, or \cite[Chapter~5.8, Theorem~3]{Eva10} for the case of real-valued functions). 
	For any $\phi \in \Cci((0,T); \calH^\infty)$, by the convergence properties of $(v_\eps)_\eps$ obtained in \ref{it:conv_bulk} and \ref{it:conv_layer}, we deduce
	\begin{align}
		&\dotproduct{\delta_t^h v_0}{\phi}{L^2(0,T-h;\calH^\prime),\, L^2(0,T-h;\calH)}\\
		&\quad= \sum_\pm \int_0^{T-h} \int_{\Ome^\pm} \delta_t^h v_0^\pm \phi^\pm \dd x \dd t
		+ \int_0^{T-h}\int_\Sigma \int_{Z_*} \delta_t^h v_0^{\mathrm{M}} \phi^{\mathrm{M}}\dd z \dd x^\prime \dd t\\
		&\quad= \lim_{\eps \to 0} \int_0^{T-h} \int_{\Omee^\pm} \delta_t^h v_\eps^\pm \phi_{\eps}^\pm \dd x \dd t
		+ \lim_{\eps \to 0} \frac 1\eps \int_0^{T-h} \int_{\OmeMS} \delta_t^h v_\eps^{\mathrm{M}} \phi_{\eps}^{\mathrm{M}} \dd x \dd t\\
		&\quad= \lim_{\eps \to 0} \int_0^{T-h} \dotproduct{\delta_t^h v_\eps}{\phi_{\eps}}{\calH_\eps^\prime,\, \calH_\eps} \dd t
		\leq \limsup_{\eps \to 0} \Lpnorm{\delta_t^h v_\eps}{2}{0,T-h;\calH_\eps^\prime} \Lpnorm{\phi_{\eps}}{2}{0,T-h;\calH_\eps}.
	\end{align}
	As $\Lpnorm{\p_t v_\eps }{2}{0,T;(\calH_\eps^m)^\prime}\lesssim 1$, we also have $\Lpnorm{\delta_t^h v_\eps}{2}{0,T-h;(\calH_\eps^m)^\prime}\lesssim 1$ uniformly in $h$ and it suffices to bound $\limsup_{\eps \to 0}\Lpnorm{\phi_{\eps}}{2}{0,T-h;\calH_\eps}$. By an explicit calculation using the oscillation lemma for thin domains (see \cite[Lemma~4.3]{NeuJae07}) we have $\Lpnorm{\phi_{\eps}}{2}{0,T-h;\calH_\eps} \to \Lpnorm{\phi}{2}{0,T-h;\calH}$ as $\eps \to 0$, and thus obtain 
	\begin{align}
		\left|\dotproduct{\delta_t^h v_0}{\phi}{L^2(0,T;\calH^\prime),\, L^2(0,T;\calH)}\right|
		\lesssim \Lpnorm{\phi}{2}{0,T-h;\calH},
	\end{align}
	which by the density of $\calH^\infty$ in the space $\calH$ extends to functions $\phi \in L^2(0,T;\calH)$ and therefore yields $\p_t v_0 \in L^2(0,T;\calH^\prime)$. The convergence \eqref{eq:conv_ptv} for test functions $\phi \in \Cci((0,T); \calH^\infty)$ then follows similarly as in \cite[Proposition~4]{GahNeuPop21} and \cite[Lemma~B.5]{GahNeu25} by the definition of the generalised time derivative.
	Using a cut-off function in time near $0$, the previous convergence also holds true for functions $\phi \in \Cci([0,T);\calH^\infty)$.
	We thus further have $v_0 \in C^0([0,T];L^2(\Ome^+)\times L^2(\Sigma\times Z_*)\times L^2(\Ome^-))$ by the regularity of $\p_t v_0$; hence, the evaluation of $v_0$ in $t=0$ is well-defined. Finally, \eqref{eq:conv_v0} is obtained by an integration by parts in time, using \eqref{eq:conv_ptv} and the convergence properties of $(v_\eps)_\eps$, since 
    \begin{align}
        -\scp{v_\eps(0)}{\phi_\eps(0)}{\calL_\eps} 
        &= \int_0^T \dotproduct{\p_t v_\eps}{\phi_\eps}{\calH_\eps^\prime, \, \calH_\eps} \dd t + \int_0^T \scp{v_\eps}{\p_t \phi_\eps}{\calL_\eps} \dd t\\
        &\xrightarrow{\eps \to 0} \dotproduct{\p_t v_0}{\phi}{\calH^\prime, \, \calH} \dd t + \int_0^T \scp{v_\eps}{\p_t \phi_\eps}{\calL} \dd t
        = -\scp{v_0(0)}{\phi(0)}{\calL},
    \end{align}
    which concludes the proof.
\end{proof}

It should be noted that the regularity and uniform bounds on the time derivatives $\p_t v_\eps$ are needed to obtain the time regularity of the limit function $v_0$, the convergence of the time derivatives in \ref{it:conv_pt} and the strong convergences in \ref{it:conv_bulk} --- and consequently for the derivation of the coupling condition between $v_0^\pm$ and $v_0^{\mathrm{M}}$ on $\Sigma \times S_*^\pm$ --- but not for the derivation of the weak (two-scale) convergence of the function and its gradient in \ref{it:conv_bulk} and \ref{it:conv_layer}. 

Moreover, the time regularity of the limit functions $v_0^\dagger$, $\dagger \in \{\pm,\, \mathrm{M}\}$, may be obtained directly by an argument similar to that used for the microsolutions $\tu_\eps^\dagger$ at the end of \autoref{subsec:shifts}.

\begin{rem}[Regularity of $J_0$]\label{rem:reg_J0}
	With \autoref{thm:convergence} at hand, we can gather significantly more information about the two-scale limit of $(J_\eps)_\eps$. First of all, as the functions $J_\eps$ satisfy a priori estimates similar to \eqref{eq:apriori} by \ref{it:T2:est_J}, we have $J_0\coloneqq (1,J_0^{\mathrm{M}},1) \in L^2(0,T;\calH)$ with $\p_t J_0 \in L^2(0,T;\calH^\prime)$ as well as $J_0^{\mathrm{M}} \in C^0([0,T];L^2(\Sigma\times Z_*))$. Moreover, we obtain $\p_t J_0^{\mathrm{M}} \in L^2(0,T;L^2(\Sigma;(H^1_{\pm,0}(Z_*))^\prime))$ and, due to the strong two-scale convergence of $(J_\eps^{\mathrm{M}})_\eps$ to $J_0^{\mathrm{M}}$ and the uniform boundedness of $(\calT_\eps J_\eps^{\mathrm{M}}) \subset L^\infty((0,T)\times \Sigma;W^{1,p}(Z_*))$, we further deduce $J_0^{\mathrm{M}} \in L^\infty ((0,T)\times \Sigma; W^{1,p}(Z_*))$ for every $p\in [1,\infty)$ and $J_0^{\mathrm{M}} \simeq 1$.
\end{rem}
We also obtain the following convergence of the Jacobian at time $t=0$, and hence of the initial conditions.
\begin{rem}[Convergence of $(J_\eps(0))_{\eps}$]
    From \autoref{thm:convergence} we also have $J_\eps^{\mathrm{M}}(0) \to J_0^{\mathrm{M}}(0)$ weakly in the two-scale sense. The strong two-scale convergence of $(J_\eps)_\eps$ combined with the continuity of both $J_\eps$ and $J_0$ in $t=0$ therefore allows to conclude $J_\eps^{\mathrm{M}}(0) \to J_0^{\mathrm{M}}(0)$ strongly in the two-scale sense in every $L^p$--space, $p\in [1,\infty)$. By an application of \autoref{lem:2s_prod} together with \ref{it:limit:U0}, we therefore also obtain $\tilde{U}_{j\eps,0}^{\mathrm{M}} \to \tilde{U}_{j0,0}^{\mathrm{M}}$ in the strong two-scale sense.  
\end{rem}

It is convenient to formulate the result of the application of \autoref{thm:convergence} to the sequence of microsolutions $(\tu_\eps)_\eps$ of \eqref{eq:RDA_fd_full} as a corollary.
\begin{cor}[Convergence of the microsolutions]\label{cor:conv_u}
    Let $(\tu_\eps)_\eps$ be the sequence of microsolutions of \eqref{eq:RDA_fd_full}.
	There exists a function $\tu_0 \in L^2(0,T;\calH^m)$ with $\p_t(J_0 \tu_0) \in L^2(0,T;(\calH^m)^\prime)$, and therefore $J_0 \tu_0 \in C^0([0,T];\calL^m)$, such that, for a subsequence, the following convergence properties hold true.
	\begin{enumerate}[label=(\roman*)]
		\item\label{it:conv_bulk_u} \underline{\em Convergence in the bulk.} The limits in the bulks satisfy 
		\begin{align}\label{eq:conv_upm}
			\begin{array}{rclcl}
				\chi_{\Omee^\pm} \tu_{j\eps}^\pm &\to& \tu_{j0}^\pm && \text{strongly in } L^2((0,T)\times \Ome^\pm),
				\\
				\chi_{\Omee^\pm} \nabla \tu_{j\eps}^\pm &\to& \nabla \tu_{j0}^\pm && \text{weakly in } L^2((0,T)\times \Ome^\pm).
			\end{array}
		\end{align}
		\item\label{it:conv_layer_u} \underline{\em Convergence in the layer.} The limit in the layer satisfies 
		\begin{align}\label{eq:conv_uM}
			\begin{array}{rclcl}
				\tu_{j\eps}^{\mathrm{M}} &\to& \tu_{j0}^{\mathrm{M}} && \text{in the two-scale sense},\\
				\eps \nabla \tu_{j\eps}^{\mathrm{M}} &\to& \nabla_z \tu_{j0}^{\mathrm{M}} &&\text{in the two-scale sense}.
			\end{array}
		\end{align}
		\item\label{it:conv_ptu} \underline{\em Convergence of the time derivative.} For any $\phi \in \Cci([0,T);(\calH^\infty)^m)$ and $\phi_\eps$ defined as in \eqref{eq:spec_test_fct} we have 
			\begin{align}
				\lim_{\eps \to 0} \int_0^T \dotproduct{\p_t (J_\eps \tu_{j\eps})}{\phi_{j\eps}}{\calH_\eps^\prime,\, \calH_\eps} \dd t 
				&= \int_0^T \dotproduct{\p_t (J_0 \tu_{j0})}{\phi_j}{\calH^\prime,\, \calH} \dd t\label{eq:2s_conv_ptu},\\
				\scp{(J_\eps \tu_{j\eps})(0)}{\phi_{j\eps}(0)}{\calL_\eps} 
				&\to \scp{(J_0\tu_{j0})(0)}{\phi_j(0)}{L^2(\Ome^+)\times L^2(\Sigma \times Z_*) \times L^2(\Ome^-)}\label{eq:conv_IC}.
			\end{align}
	\end{enumerate}
\end{cor}
\begin{proof}
    By the a priori estimates in \eqref{eq:apriori2}, choosing $v_\eps = J_\eps \tu_{j\eps}$ in \autoref{thm:convergence}, we directly obtain the existence of a limit function $v_0 \in L^2(0,T; \calH^m)$ with $\p_t v_0 \in L^2(0,T;(\calH^m)^\prime)$ such that, for a subsequence, the convergence results in \ref{it:conv_bulk_u} and \ref{eq:2s_conv_ptu} are valid and $J_\eps \tu_{j\eps}^{\mathrm{M}} \to v_{j0}^{\mathrm{M}}$ holds in the two-scale sense. 
    Choosing $v_\eps = \tu_{j\eps}$ in \autoref{thm:convergence} and using the a priori estimates \eqref{eq:apriori1}, we further deduce the existence of a limit function $\tu_0 \in L^2(0,T;\calH^m)$ such that the convergences in \ref{it:conv_layer_u} hold, and are satisfied in \ref{it:conv_bulk_u} in the weak sense. It hence remains to identify $J_0 \tu_0 = v_0$, which then also yields the claimed time regularity of $\tu_0^{\mathrm{M}}$.
    This identification is due to the fact that we have $J_\eps^{\mathrm{M}} \to J_0^{\mathrm{M}}$ strongly in the two-scale sense in any $L^p$, $p\in [1,\infty)$, and thus also for the inverse of the Jacobians. Indeed, from \autoref{lem:2s_prod} we may now conclude $J_\eps^{\mathrm{M}} \tu_{j\eps}^{\mathrm{M}} \to J_0^{\mathrm{M}} \tu_{j0}^{\mathrm{M}}$ in the two-scale sense, and therefore $J_0^{\mathrm{M}} \tu_0^{\mathrm{M}} = v_0^{\mathrm{M}}$.
\end{proof}
We postpone a precise formulation of the time regularity of $\tu_0^\dagger$, $\dagger \in \{\pm,\, \mathrm{M}\}$ to \autoref{cor:macro_prob_weak} below as it is not needed for the derivation of the macroscopic problem in \autoref{thm:macroscopic}.

In order to pass to the limit in the nonlinearities, we need to establish strong two-scale convergence of the solutions $\tu_\eps^{\mathrm{M}}$ in the channel domain $\OmeMS$. Making use of the unfolding operator $\calT_\eps$ (see \autoref{def:unfolding}) strong two-scale convergence of the sequence $(\tu_\eps^{\mathrm{M}})_\eps$, can be established by showing strong $L^2$-convergence of the unfolded sequence $(\calT_\eps \tu_\eps^{\mathrm{M}})$, see \autoref{lem:char_2s_conv}. The main compactness result we use to show this is given in \autoref{thm:strong_2s_cpctness}.
\begin{lemma}[Strong convergence in the layer]\label{lem:strong_2s_conv}
    For every $p\in [1,2)$ and $\beta \in (\tfrac 12,1)$, the following converges hold true, possibly after passing to a subsequence,
    \begin{align}\label{eq:s_conv_uM}
        \begin{array}{rclcl}
        	\calT_\eps \tu_{j\eps}^{\mathrm{M}} &\to& \tu_{j0}^{\mathrm{M}} && \text{ strongly in } L^p(\Sigma;L^2((0,T)\times Z_*)), \\
        	\calT_\eps (\tu_{j\eps}^{\mathrm{M}}\rvert_{N_\eps}) &\to& \tu_{j0}^{\mathrm{M}}\rvert_N && \text{ strongly in } L^p(\Sigma;L^2((0,T)\times N)),\\
        	\calT_\eps(J_\eps \tu_{j\eps}^{\mathrm{M}}) &\to& J_0^{\mathrm{M}} \tu_{j0}^{\mathrm{M}} && \text{ strongly in } L^p(\Sigma;L^2(0,T;H^\beta(Z_*))).
        \end{array}
    \end{align}
\end{lemma}
\begin{proof}
    We establish the conditions \ref{it:strong_2s_cpct_apriori} and \ref{it:strong_2s_cpct_shifts} of \autoref{thm:strong_2s_cpctness} for the sequence $(J_\eps^{\mathrm{M}} \tu_\eps^{\mathrm{M}})$. The estimate \ref{it:strong_2s_cpct_apriori} is given by the a priori estimates in \autoref{lem:apriori} and uses \eqref{eq:pt_u_pmM}. 
    The shifts of $J_\eps^{\mathrm{M}} \tu_\eps^{\mathrm{M}}$ are controlled by \autoref{lem:shifts}, which implies \ref{it:strong_2s_cpct_shifts}, where we use that $\Lpnorm{\delt \tu_{j\eps}^\pm}{2}{0,T;\Omeh^\pm} \to 0$ as $|\eps l^\prime| \to 0$ due to \eqref{eq:conv_upm}, and \ref{it:shifts:U0}. 
    \autoref{thm:strong_2s_cpctness} thus allows us to find a function $v_0 \in L^2(0,T;L^2(\Sigma;H^1(Z_*)^m))$ such that, for a subsequence, $\calT_\eps (J_\eps^{\mathrm{M}} \tu_{j\eps}^{\mathrm{M}}) \to v_{j0}$ strongly in $L^p(\Sigma;L^2(0,T;H^\beta(Z_*)^m))$, and using \autoref{cor:conv_u}~\ref{it:conv_layer_u}, \autoref{lem:2s_prod} and the strong two-scale convergence $J_\eps^{\mathrm{M}} \to J_0^{\mathrm{M}}$, we conclude $v_0=J_0^{\mathrm{M}} \tu_0^{\mathrm{M}}$. Therefore, the last convergence in \eqref{eq:s_conv_uM} is obtained.
    The continuity of the embeddings $H^\beta(Z_*) \xhookrightarrow{} L^2(Z_*)$ and $H^\beta(Z_*)\xhookrightarrow{} L^2(N)$ further gives $\calT_\eps (J_\eps^{\mathrm{M}} \tu_{j\eps}^{\mathrm{M}}) \to J_0^{\mathrm{M}} \tu_{j0}^{\mathrm{M}}$ strongly in $L^p(\Sigma;L^2((0,T)\times Z_*))$ and $\calT_\eps (J_\eps^{\mathrm{M}} \tu_{j\eps}^{\mathrm{M}}\rvert_{N_\eps}) \to J_0^{\mathrm{M}} \tu_{j0}^{\mathrm{M}}\rvert_N$ strongly in $L^p(\Sigma;L^2((0,T)\times N))$. 
    A direct calculation using $J_\eps^{\mathrm{M}}, J_0^{\mathrm{M}} \sim 1$, the strong two-scale convergence of $(J_\eps^{\mathrm{M}})_\eps$ in every $L^q$, $q\in (1,\infty)$ and Hölder's inequality thus yields $\calT_\eps \tu_{j\eps}^{\mathrm{M}} \to \tu_{j0}^{\mathrm{M}}$ in $L^p(\Sigma;L^2((0,T)\times Z_*))$.
    The second convergence in \eqref{eq:s_conv_uM} is a consequence of Pratt's theorem (see \cite{Pra60}) and uses the strong convergence of $(J_\eps^{\mathrm{M}})_\eps$ in $L^p(\Sigma;L^2(0,T;H^\beta(Z_*)))$, which holds due to \autoref{thm:strong_2s_cpctness} and the assumptions \ref{it:T1:est_psi} and \ref{it:T2:est_J} together with the continuity of the trace embedding $H^\beta(Z_*) \xhookrightarrow{} L^2(\p Z_*)$.
\end{proof}

\subsection{Identification of the macroscopic reference model}\label{subsec:macro_prob_ref}
We next derive the limit problem which the function $\tu_0 = (\tu_0^+,\tu_0^{\mathrm{M}},\tu_0^-)$ satisfies and discuss some of its properties. 
We write $F_0 \coloneqq \DD_z \psi_0$ and define, using the same notation as in \ref{it:limit:D}--\ref{it:limit:U0},
\begin{align}
	\tD_{j0}^{\mathrm{M}} = F_0^{-1} \bar D_{j0}^{\mathrm{M}} F_0^{-\top}, \qquad 
	\tq_{j0}^{\mathrm{M}} = F_0^{-1} \bar q_{j0}^{\mathrm{M}}, \qquad 
	\tilde{U}_{0,0} = U_{0,0} \circ \psi_0(0,\cdot), \qquad
	\tb_0^{\mathrm{M}} = F_0^{-1} \p_t \psi_0.
\end{align}
First, we pass to the limit in the microscopic formulation \eqref{eq:RDA_fd_weak}, in which we denote the outer unit normal of $Z_*$ by $\nu^{\mathrm{M}}$.

\begin{theorem}[Macroscopic transformed problem]\label{thm:macroscopic}
    The limit function $\tu_0 \in L^2(0,T;\calH^m)$ satisfies $\p_t(J_0 \tu_0) \in L^2(0,T;(\calH^m)^\prime)$ and for every $\phi \in \calH^m$ and almost every $t\in (0,T)$ we have 
    \begin{align}
    	\begin{split}\label{eq:macroscopic_fd_weak}
    		&\dotproduct{\p_t(J_0 \tu_{j0})}{\phi_j}{\calH^\prime,\,\calH} 
    		 + \sumpm \int_{\Ome^\pm} D_{j}^\pm \nabla \tu_{j0}^\pm \cdot \nabla \phi_j^\pm \dd x
    		+ \int_{\Sigma} \int_{Z_*} J_0^{\mathrm{M}} \tD_{j0}^{\mathrm{M}} \nabla_z \tu_{j0}^{\mathrm{M}} \cdot \nabla_z \phi_j^{\mathrm{M}} \dd z \dd x^\prime\\
    		&\quad -\sumpm \int_{\Ome^\pm} \tu_{j0}^\pm q_{j}^\pm\cdot \nabla \phi_j^\pm \dd x 
    		- \int_{\Sigma} \int_{Z_*} J_0^{\mathrm{M}} \tu_{j0}^{\mathrm{M}} \tq_{j0}^{\mathrm{M}} \cdot \nabla_z \phi_j^{\mathrm{M}} \dd z \dd x^\prime
    		+ \int_{\Sigma} \int_{Z_*} J_0^{\mathrm{M}} \tu_{j0}^{\mathrm{M}} \tb_0^{\mathrm{M}}\cdot \nabla_z \phi_j^{\mathrm{M}}  \dd z\dd x^\prime\\
    		&= \sumpm \int_{\Ome^\pm} f_{j}(\tu_0^\pm)\phi_j^\pm \dd x
    		+ \int_{\Sigma} \int_{Z_*} J_0^{\mathrm{M}} g_{j}(\tu_0^{\mathrm{M}}) \phi_j^{\mathrm{M}} \dd z \dd x^\prime - \int_\Sigma \int_{N} J_0^{\mathrm{M}} \norm{F_0^{-\top} \nu^{\mathrm{M}}} {h}_{j}(\tu_0^{\mathrm{M}}) \phi_j^{\mathrm{M}} \dd \calH^{n-1}(z) \dd x^\prime
    	\end{split}
  \end{align}
   together with the initial conditions $\tu_0^\pm = \tilde{U}_0^\pm$ a.e.~in $\Ome^\pm$ and $(J_0^{\mathrm{M}} \tu_0^{\mathrm{M}})(0)=J_0^{\mathrm{M}}(0) \tilde{U}_{0,0}^{\mathrm{M}}$ a.e.~in $\Sigma \times Z_*$. 
\end{theorem}
\begin{proof}
    The regularity of $\tu_0$ is proven in \autoref{cor:conv_u}. Given $\phi \in \Cci((0,T);(\calH^\infty)^m)$, we consider test functions	$\phi_\eps \in \Cci((0,T);(\calH_\eps)^m)$ of the form \eqref{eq:spec_test_fct} in the weak formulation \eqref{eq:RDA_fd_weak} to obtain
    \begin{align}
    	&\int_0^T \dotproduct{\p_t(J_\eps \tu_{j\eps})}{\phi_{j\eps}}{\calH_\eps^\prime,\, \calH_\eps} \dd t + \sumpm \int_0^T\int_{\Omee^\pm} \left(D_{j}^\pm \nabla \tu_{j\eps}^\pm 
    	- \tu_{j\eps}^\pm q_{j}^\pm\right)\cdot (\nabla \phi_j^\pm)_\eps \dd x  \dd t \\
    	&\quad + \int_0^T \int_{\OmeMS} \left(\eps J_\eps^{\mathrm{M}} \tD_{j\eps}^{\mathrm{M}} \nabla \tu_{j\eps}^{\mathrm{M}} 
    	- J_\eps^{\mathrm{M}} \tu_{j\eps}^{\mathrm{M}} \tq_{j\eps}^{\mathrm{M}}
    	+ \frac 1\eps J_\eps^{\mathrm{M}} \tu_{j\eps}^{\mathrm{M}} \tb_\eps^{\mathrm{M}}\right)\cdot \left((\nabla_{x^\prime} \phi_j^{\mathrm{M}})_\eps + \frac 1\eps (\nabla_z \phi_j^{\mathrm{M}})_\eps \right) \dd x  \dd t \\
    	&= \sumpm \int_0^T \int_{\Omee^\pm} f_{j}(\tu_{\eps}^\pm)\phi_{j\eps}^\pm \dd x  \dd t
    	+ \frac 1\eps \int_0^T \int_{\OmeMS} J_\eps^{\mathrm{M}} g_{j}(\tu_\eps^{\mathrm{M}}) \phi_{j\eps}^{\mathrm{M}} \dd x  \dd t\\
    	&\qquad- \int_0^T \int_{N_\eps} J_\eps^{\mathrm{M}} \norm{F_\eps^{-\top} \nu_\eps^{\mathrm{M}}} h_{j}(\tu_\eps^{\mathrm{M}}) \phi_{j\eps}^{\mathrm{M}} \dd \calH^{n-1}(x) \dd t,
    \end{align} 
    where $(\nabla \phi^\pm)_\eps(t,x) = \nabla \phi^\pm(t,x^\prime,x_n \mp \eps)$, $(\nabla_{x^\prime} \phi^{\mathrm{M}})_\eps(t,x) = \nabla_{x^\prime}\phi^{\mathrm{M}}(t,x^\prime,\tfrac{x_n}{\eps})$ and similarly for $(\nabla_z \phi^{\mathrm{M}})_\eps$.
    The limit of the integrals over $\Omee^\pm$ is obtained using the convergence results from \autoref{cor:conv_u}~\ref{it:conv_bulk_u} together with $(\nabla \phi_j^\pm)_\eps \to \nabla \phi_{j}^\pm$ in $L^2((0,T)\times \Ome^\pm)^n$.
    By \autoref{lem:2s_prod} as well as the assumptions \ref{it:T4:2s_conv} and \ref{it:limit:D}--\ref{it:limit:U0}, we have $\tD_{j\eps}^{\mathrm{M}} \to \tD_{j0}^{\mathrm{M}}$, $\tq_{j\eps}^{\mathrm{M}} \to \tq_{j0}^{\mathrm{M}}$, $\tb_{\eps}^{\mathrm{M}} \to \tb_0^{\mathrm{M}}$, and $\tilde{U}_{j\eps,0}^{\mathrm{M}} \to \tilde{U}_{j0,0}^{\mathrm{M}}$ in the strong two-scale sense. Therefore, due to the convergence results in \autoref{cor:conv_u}~\ref{it:conv_layer_u}, products of the data and $J_\eps^{\mathrm{M}}$ with the functions $\tu_{j\eps}^{\mathrm{M}}$ and $\eps \nabla \tu_{j\eps}^{\mathrm{M}}$ converge in the two-scale sense, and we have that $\phi_j^{\mathrm{M}}$ and $\nabla_z \phi_j^{\mathrm{M}}$ are admissible test functions to pass to the limit $\eps \to 0$ in the integrals over $\OmeMS$ in the previous equation.
    The convergence of the generalised time derivative is obtained by \autoref{cor:conv_u}~\ref{it:conv_ptu}.
    Moreover, due to the Lipschitz continuity of the reaction rates in the last argument, the strong (two-scale) convergence of $(\tu_\eps^\dagger)_\eps$ and \autoref{lem:s_2s_conv_NL}, for $p\in [1,2)$ we have
    \begin{align}
    	\begin{array}{rclcl}
    		\chi_{\Omee^\pm} f_j(\tu_\eps^\pm) &\to& f_j(\tu_0^\pm) &&  \text{strongly in }L^2((0,T)\times \Ome^\pm),\\
    		g_j(\tu_\eps^{\mathrm{M}}) &\to& g_j(\tu_0^{\mathrm{M}}) &&  \text{strongly in the two-scale sense in $L^p$,} \\
    		\norm{F_\eps^{-\top} \nu_\eps^{\mathrm{M}}} h_j(\tu_\eps^{\mathrm{M}}) &\to& \norm{F_0^{-\top} \nu^{\mathrm{M}}} h_j(\tu_0^{\mathrm{M}}) &&  \text{strongly in the two-scale sense on $N_\eps$ in $L^p$.}
    	\end{array}
    \end{align}
    In total, along a subsequence we obtain
	\begin{align}
		\begin{split}\label{eq:macro_IC}
			&\int_0^T \dotproduct{\p_t(J_0 \tu_{j0})}{\phi_j}{\calH^\prime,\, \calH} \dd t + \sumpm \int_0^T\int_{\Ome^\pm} \left(D_{j}^\pm \nabla \tu_{j0}^\pm 
			- \tu_{j0}^\pm q_{j}^\pm\right)\cdot \nabla \phi_j^\pm \dd x \dd t \\ 
			&\quad + \int_0^T \int_{\Sigma} \int_{Z_*} \left(J_0^{\mathrm{M}} \tD_{j0}^{\mathrm{M}} \nabla_z \tu_{j0}^{\mathrm{M}} 
			- J_0^{\mathrm{M}} \tu_{j0}^{\mathrm{M}} \tq_{j0}^{\mathrm{M}}
			+ J_0^{\mathrm{M}} \tu_{j0}^{\mathrm{M}} \tb_0^{\mathrm{M}}\right)\cdot \nabla_z \phi_j^{\mathrm{M}}  \dd z\dd x^\prime  \dd t \\
			&= \sumpm \int_0^T \int_{\Ome^\pm} f_{j}(\tu_0^\pm)\phi_{j}^\pm \dd x \dd t 
			+ \int_0^T \int_{\Sigma} \int_{Z_*} J_0^{\mathrm{M}} g_{j}(\tu_0^{\mathrm{M}}) \phi_{j}^{\mathrm{M}} \dd z \dd x^\prime  \dd t \\
			&\qquad- \int_0^T \int_\Sigma \int_{N} J_0^{\mathrm{M}} \norm{F_0^{-\top} \nu^{\mathrm{M}}} h_{j}(\tu_0^{\mathrm{M}}) \phi_{j}^{\mathrm{M}} \dd \calH^{n-1}(z) \dd x^\prime \dd t.
		\end{split}
	\end{align}
    By density, \eqref{eq:macro_IC} holds true for all $\phi \in \Cci((0,T);\calH^m)$. In particular, choosing test functions with separated $t$--dependency, we deduce \eqref{eq:macroscopic_fd_weak}. The initial conditions are a direct consequence of \eqref{eq:conv_IC} together with \ref{it:limit:U0}.
\end{proof}

\begin{rem}[Limit passage in \eqref{eq:RDA_fd_weak}]
	A more standard way to pass to the limit $\eps \to 0$ in \eqref{eq:RDA_fd_weak} and to conclude the regularity of $\p_t(J_0 \tu_0)$ is by integrating by parts in time in the weak microscopic formulation and taking the limit in the resulting terms separately. However, when dealing with test functions $\phi \in \Cci([0,T);\calH^\infty)^m$, and without having a result like \eqref{eq:conv_IC} at hand, it is a priori not clear why $\scp{J_\eps(0) \tilde{U}_{j\eps,0}}{\phi_{j\eps}(0)}{\calL_\eps}  \to \scp{J_0(0)\tilde{U}_{j0,0}}{\phi_j(0)}{\calL}$ holds in the limit $\eps \to 0$.
\end{rem}

\begin{cor}[Convergence of the entire sequence]
	The function $\tu_0$ is the unique weak solution of the macroscopic problem \eqref{eq:macroscopic_fd_weak}. In particular, the convergence result from \autoref{cor:conv_u} holds true for the entire sequence $(\tu_\eps)_\eps$.
\end{cor}
\begin{proof}
	The uniqueness of weak solutions is obtained similarly as for the microscopic problem in \autoref{lem:existence} by an application of Gronwall's inequality, using that the limit data $\tD_0$, $\tq_0$ and $\tb_0$ satisfy similar estimates as the microscopic ones.
	In particular, the entire sequence $(\tu_\eps)_\eps$ converges to the solution $\tu_0$ of the macroscopic problem \eqref{eq:macroscopic_fd_weak}. 
\end{proof}

Next, we aim to provide more insight into the individual problems solved by the functions $\tu_0^\pm$ and $\tu_0^{\mathrm{M}}$, which, in a final step, yields macroscopic problems on the bulk domains $\Ome^\pm$ coupled by effective transmission conditions given in terms of local cell problems stated on the standard cell $Z_*$.

\begin{cor}[Macroscopic transformed problem and local cell problems]\label{cor:macro_prob_weak}
	For the limit function $\tu_0=(\tu_0^+,\tu_0^{\mathrm{M}},\tu_0^-)$, there holds 
    \begin{align}
        \tu_0^\pm \in L^2(0,T;H^1(\Ome^\pm)^m) \quad \text{with} \quad 
        \p_t \tu_0^\pm \in L^2(0,T;(H^1_{\Sigma,0}(\Ome^\pm)^m)^\prime)
    \end{align}
    as well as $\tu_0^\pm(0) = \tilde{U}_0^\pm$ a.e.~in $\Ome^\pm$ and $\tu_0^\pm = \tu_0^{\mathrm{M}}$ on $(0,T)\times \Sigma \times S_*^\pm$. 
    Further, for every $\phi^\pm \in H^1_{\Sigma,0}(\Ome^\pm)^m$ and almost every $t\in (0,T)$, we have
	\begin{align}
		&\dotproduct{\p_t \tu_{j0}^\pm}{\phi_j^\pm}{(H^1_{\Sigma,0}(\Ome^\pm))^\prime,\, H^1_{\Sigma,0}(\Ome^\pm)}
		+ \int_{\Ome^\pm} \left(D_{j}^\pm \nabla \tu_{j0}^\pm  - \tu_{j0}^\pm q_{j}^\pm\right)\cdot \nabla \phi_j^\pm \dd x
		= \int_{\Ome^\pm} f_{j}(\tu_0^\pm)\phi_{j}^\pm \dd x.
	\end{align}
	Here, the function 
    \begin{align}
        \tu_0^{\mathrm{M}} \in L^2(0,T;L^2(\Sigma;H^1(Z_*)^m)) \quad \text{with} \quad
        \p_t (J_0^{\mathrm{M}}\tu_0^{\mathrm{M}}) \in L^2(0,T;L^2(\Sigma;(H^1_{\pm,0}(Z_*)^m)^\prime))
    \end{align}
    is such that $(J_0^{\mathrm{M}} \tu_0^{\mathrm{M}})(0) = J_0^{\mathrm{M}}(0) \tilde{U}_{0,0}^{\mathrm{M}}$ a.e.~in $\Sigma \times Z_*$; moreover, the local cell problem is given by: For every $\phi^{\mathrm{M}} \in L^2(\Sigma;H^1_{\pm,0}(Z_*)^m)$ and almost every $t\in (0,T)$ 
	\begin{align}
		&\dotproduct{\p_t (J_0^{\mathrm{M}}\tu_{j0}^{\mathrm{M}})}{\phi_j^{\mathrm{M}}}{L^2(\Sigma;(H^1_{\pm,0}(Z_*))^\prime),\, L^2(\Sigma;H^1_{\pm,0}(Z_*))}
		+ \int_{\Sigma} \int_{Z_*} J_0^{\mathrm{M}} \tD_{j0}^{\mathrm{M}} \nabla_z \tu_{j0}^{\mathrm{M}} \cdot \nabla_z \phi_j^{\mathrm{M}} \dd z \dd x^\prime\\
		&\quad - \int_{\Sigma} \int_{Z_*} J_0^{\mathrm{M}} \tu_{j0}^{\mathrm{M}} \tq_{j0}^{\mathrm{M}} \cdot \nabla_z \phi_j^{\mathrm{M}} \dd z \dd x^\prime
		+ \int_{\Sigma} \int_{Z_*} J_0^{\mathrm{M}} \tu_{j0}^{\mathrm{M}} \tb_0^{\mathrm{M}}\cdot \nabla_z \phi_j^{\mathrm{M}}  \dd z\dd x^\prime\\
		&= \int_{\Sigma} \int_{Z_*} J_0^{\mathrm{M}} g_{j}(\tu_0^{\mathrm{M}}) \phi_j^{\mathrm{M}} \dd z \dd x^\prime 
		- \int_\Sigma \int_{N} J_0^{\mathrm{M}} \norm{F_0^{-\top} \nu^{\mathrm{M}}} {h}_{j}(\tu_0^{\mathrm{M}}) \phi_j^{\mathrm{M}} \dd \calH^{n-1}(z) \dd x^\prime.
	\end{align}
\end{cor}

\begin{proof}
    Using the Gelfand triples $H^1_{\Sigma,0}(\Ome^\pm) \xhookrightarrow{} L^2(\Ome^\pm) \xhookrightarrow{} (H^1_{\Sigma,0}(\Ome^\pm))^\prime$ and $L^2(\Sigma;H^1_{\pm,0}(Z_*)) \xhookrightarrow{} L^2(\Sigma \times Z_*) \xhookrightarrow{} L^2(\Sigma; (H^1_{\pm,0}(Z_*))^\prime)$, for all $\Phi^\pm, \, \Phi^{\mathrm{M}}  \in \calH^m$ of the form $\Phi^+ = (\phi^+,0,0)$, $\Phi^{\mathrm{M}} = (0,\phi^{\mathrm{M}},0)$ and $\Phi^- = (0,0, \phi^-)$, where $\phi^\pm \in H^1_{\Sigma,0}(\Ome^\pm)^m$ and $\phi^{\mathrm{M}} \in L^2(\Sigma; H^1_{\pm,0}(Z_*)^m)$, we have
	\begin{align}
		\dotproduct{\p_t \tu_{j0}^\pm}{\phi_j^\pm}{(H^1_{\Sigma,0}(\Ome^\pm))^\prime,\, H^1_{\Sigma,0}(\Ome^\pm)}
		&= \dotproduct{\p_t (J_0\tu_{j0})}{\Phi_j^\pm}{\calH^\prime,\,\calH},\\
		\dotproduct{\p_t (J_0^{\mathrm{M}}\tu_{j0}^{\mathrm{M}})}{\phi_j^{\mathrm{M}}}{L^2(\Sigma;(H^1_{\pm,0}(Z_*))^\prime),\, L^2(\Sigma;H^1_{\pm,0}(Z_*))}
		&= \dotproduct{\p_t (J_0\tu_{j0})}{\phi_j^{\mathrm{M}}}{\calH^\prime,\,\calH}.
	\end{align}
    This immediately yields the claimed time regularity of $\tu_0^\dagger$; choosing $\Phi^\pm,\, \Phi^{\mathrm{M}} \in \calH^m$ as test functions in \eqref{eq:macroscopic_fd_weak} the proof is complete.
\end{proof}

In contrast to the microsolutions $\tu_\eps$, the limit function $\tu_0^{\mathrm{M}}$ is not continuous in time, that is $\tu_0^{\mathrm{M}} \notin C^0([0,T];L^2(\Sigma \times Z_*))$ in general, as the following remark shows. In particular, the initial condition $\tu_0^{\mathrm{M}} = \tilde{U}_{0,0}^{\mathrm{M}}$ a.e.~in $\Sigma \times Z_*$ is ill-posed.

\begin{rem}[Regularity of $\p_t \tu_0^{\mathrm{M}}$ and initial condition of $\tu_0^{\mathrm{M}}$]\label{rem:reg_ptu0M}
	Due to the low regularity of $\p_t J_0^{\mathrm{M}}$, see \autoref{rem:reg_J0}, we only have 
	\begin{align}\label{eq:ptu0}
		\p_t \tu_0^{\mathrm{M}} = (J_0^{\mathrm{M}})^{-1} \p_t(J_0^{\mathrm{M}} \tu_0^{\mathrm{M}}) - (J_0^{\mathrm{M}})^{-1} \tu_0^{\mathrm{M}} \p_t J_0^{\mathrm{M}} \in L^1((0,T)\times \Sigma;(W^{1,q}_{\pm,0}(Z_*)^m)^\prime)
	\end{align}
	for every $q>n$ (see \cite[p.~122]{GahNeuPop21} for a similar argument on bulk domains). Indeed, for every $\phi\in H^1_{\pm,0}(Z_*)$ and almost every $(t,x^\prime) \in (0,T)\times \Sigma$, there holds $(J_0^{\mathrm{M}})^{-1}(t,x^\prime,\cdot) \phi \in H^1_{\pm,0}(Z_*)$ using that $J_0^{\mathrm{M}} \in L^\infty((0,T)\times \Sigma; W^{1,p}(Z_*))$ for every $p \geq 1$ together with the embeddings $W^{1,p}(Z_*) \xhookrightarrow{} L^\infty(Z_*)$ for $p>n$ and $H^1_{\pm,0}(Z_*)\xhookrightarrow{} L^{2^*}(Z_*)$ as well as the generalised Hölder inequality. Therefore, the first term in \eqref{eq:ptu0} belongs to $L^2((0,T)\times \Sigma;(H^1_{\pm,0}(Z_*)^m)^\prime)$. By a similar argument, for every $\phi \in W^{1,q}_{\pm,0}(Z_*)$ with $q>n$, we have $(J_0^{\mathrm{M}})^{-1}(t,x^\prime,\cdot) \tu_0^{\mathrm{M}}(t,x^\prime,\cdot) \phi \in H^1_{\pm,0}(Z_*)^m$; hence, the second term in \eqref{eq:ptu0} belongs to $L^1((0,T)\times \Sigma;(W^{1,q}_{\pm,0}(Z_*)^m)^\prime)$. We stress that the additional requirement $\p_t J_0^{\mathrm{M}} \in L^\infty((0,T)\times \Sigma\times Z_*)$ would allow to obtain $\p_t \tu_0^{\mathrm{M}} \in L^2((0,T)\times \Sigma;(W^{1,q}_{\pm,0}(Z_*)^m)^\prime)$ and, therefore, also $\tu_0^{\mathrm{M}} \in C^{0}([0,T] \times \overline{\Sigma}; L^2(Z_*)^m)$.

    Nevertheless, we emphasise that by rephrasing the argument given in \cite[p.~123]{GahNeuPop21}, there exists a set $N\subset (0,T)$ of measure zero such that 
	\begin{align}
		\lim_{t\to 0,\, t \notin N}\Lpnorm{\tu_0^{\mathrm{M}}(t) - \tilde{U}_{0,0}^{\mathrm{M}}}{1}{\Sigma \times Z} = 0,
	\end{align}
	which gives a weak formulation of the initial condition for $\tu_0^{\mathrm{M}}$.
\end{rem}

\begin{rem}[Strong formulation of the macroscopic transformed problem]
    The strong problem associated with the weak formulation in \autoref{cor:macro_prob_weak} reads as follows. The limit functions $\tu_0^\pm$ solve
    \begin{align}
		\left\{\!\!\begin{array}{r@{\ }c@{\ }l@{\ }c@{\ \, }l}
			\p_t \tu_{j0}^\pm - \nabla \cdot(D_j^\pm \nabla \tu_{j0}^\pm - \tu_{j0}^\pm q_j^\pm) & = & f_j(\tu_0^\pm) && \text{in } (0,T)\times \Ome^\pm,\\
			(D_j^\pm \nabla \tu_{j0}^\pm - \tu_{j0}^\pm q_j^\pm)\cdot \nu^\pm & = & 0 && \text{on } (0,T)\times \p\Ome^\pm \setminus \Sigma,\\
			\tu_{j0}^\pm & = & \tu_{j0}^{\mathrm{M}} && \text{on } (0,T)\times \Sigma \times S_*^\pm,
		\end{array}\right.
	\end{align}
    with initial condition $\tu_0^\pm(0) = \tilde{U}_0^\pm$ in $\Ome^\pm$, whereas $\tu_0^{\mathrm{M}}$ with initial condition $(J_0^{\mathrm{M}}\tu_0^{\mathrm{M}})(0) = J_0^{\mathrm{M}}(0)\tilde{U}_{0,0}^{\mathrm{M}}$ in $\Sigma \times Z_*$ is a solution of the local cell problems
    \begin{align}
		\left\{\!\!\begin{array}{r@{\ }c@{\ }l@{\ }c@{\ \, }l}
			\p_t (J_0^{\mathrm{M}} \tu_{j0}^{\mathrm{M}}) - \nabla_z \cdot(J_0^{\mathrm{M}} \tD_{j0}^{\mathrm{M}} \nabla_z \tu_{j0}^{\mathrm{M}} - J_0^{\mathrm{M}} \tu_{j0}^{\mathrm{M}} \tq_{j0}^{\mathrm{M}} + J_0^{\mathrm{M}} \tu_{j0}^{\mathrm{M}} \tb_0^{\mathrm{M}}) &=& J_0^{\mathrm{M}} g_j(\tu_0^{\mathrm{M}}) && \text{in } (0,T)\times \Sigma \times Z_*,\\
			-(\tD_{j0}^{\mathrm{M}} \nabla_z \tu_{j0}^{\mathrm{M}} - \tu_{j0}^{\mathrm{M}} \tq_{j0}^{\mathrm{M}} + \tu_{j0}^{\mathrm{M}} \tb_0^{\mathrm{M}}) \cdot \nu^{\mathrm{M}}  & = & \norm{F_0^{-\top} \nu^{\mathrm{M}} } h_{j}(\tu_0^{\mathrm{M}}) && \text{on } (0,T)\times \Sigma \times N.
		\end{array}\right.
	\end{align}
\end{rem}

We next show that under additional regularity assumptions on the time derivatives $\p_t \tu_0^\dagger$ we have an effective jump condition for the total fluxes on $S_*^\pm$ in a weak sense. For a Lipschitz domain $U\subset \IR^n$ with outer unit normal $\nu$, we denote by $H^1_{\div}(U) \usesym{H1div}$ the space of all functions in $L^2(U)^n$ whose distributional divergence belongs to $L^2(U)$. Then, there exists a linear and bounded operator $\gamma_\nu \colon H^1_{\div}(U) \to H^{-\frac12}(\p U)$, called the normal trace operator, such that for all $u\in C^\infty(\overline{U})^n$ we have $\gamma_\nu(u) = u\rvert_{\p U} \cdot \nu$ and 
\begin{align}
    \dotproduct{\gamma_\nu(u)}{\phi}{H^{-\frac12}(\p U),\, H^{\frac12}(\p U)} = \int_U (u\cdot \nabla \phi + \phi \nabla \cdot u )\dd x \quad \text{for all } u \in H^1_{\div}(U), \, \phi \in H^1(U),
\end{align}
see \cite[Theorem~5.3]{SimSoh92}.
In what follows, $\gamma_{\nu^\pm}$ resp.~$\gamma_{\nu^*}$ denote the normal trace operators associated with $\Ome^\pm$ resp.~$Z_*$. 
\begin{lemma}[Effective transmission condition for the total flux]\label{lem:eff_transmission}
	Assume that $\p_t \tu_0^\pm \in L^2((0,T)\times \Ome^\pm)^m$ and $\p_t(J_0^{\mathrm{M}}\tu_0^{\mathrm{M}}) \in L^2((0,T)\times \Sigma \times Z_*)^m$. Then, for almost every $t\in (0,T)$, the total fluxes 
	\begin{align}
		F_j^\pm = D_j^\pm \nabla \tu_{j0}^\pm - \tu_{j0}^\pm q_j^\pm \qquad \text{and} \qquad 
		F_j^{\mathrm{M}}  = J_0^{\mathrm{M}} D_{j0}^{\mathrm{M}} \nabla \tu_{j0}^{\mathrm{M}} - J_0^{\mathrm{M}} \tu_{j0}^{\mathrm{M}} \tq_{j0}^{\mathrm{M}} + J_0^{\mathrm{M}} \tu_{j0}^{\mathrm{M}} \tb_0^{\mathrm{M}}
	\end{align} 
	satisfy $F_j^\pm \in H^1_{\div}(\Ome^\pm)$, $F_j^{\mathrm{M}} \in L^2(\Sigma;H^1_{\div}(Z_*))$, and for every $\phi \in \calH^m$ we have
	\begin{align}
		&\dotproduct{\gamma_{\nu^\pm}(F_j^\pm)}{\phi_j^\pm}{H^{-\frac 12} (\p\Ome^\pm),\, H^{\frac 12}(\p\Ome^\pm)} \\
		&\quad= - \int_\Sigma \dotproduct{\gamma_{\nu^*}(F_j^{\mathrm{M}})}{\phi_j^{\mathrm{M}}}{H^{-\frac 12} (\p Z_*),\, H^{\frac 12}(\p Z_*)} \dd x^\prime 
		- \int_\Sigma \int_{N} J_0^{\mathrm{M}} \norm{F_0^{-\top} \nu^{\mathrm{M}}} {h}_{j}(\tu_0^{\mathrm{M}}) \phi_j^{\mathrm{M}} \dd \calH^{n-1}(z) \dd x^\prime.
	\end{align}
\end{lemma}
\begin{proof}
Under the given regularity of the time derivatives $\p_t \tu_0^\dagger$, $\dagger \in \{\pm,\, \mathrm{M}\}$, choosing test functions $\phi^\pm \in \Cci(\Ome^\pm)^m$ resp.~$\phi^{\mathrm{M}} \in L^2(\Sigma;\Cci(Z_*)^m)$ in the weak macroscopic formulations obtained in \autoref{cor:macro_prob_weak}, for almost every $t\in (0,T)$ we obtain 
	\begin{align}
		\begin{array}{rcl}
			F_j^\pm \in  H^1_{\div}(\Ome^\pm) &\text{with}& \nabla \cdot F_j^\pm = \p_t \tu_{j0}^\pm -f_j (\tu_0^\pm),\\
			F_j^{\mathrm{M}} \in L^2(\Sigma; H^1_{\div}(Z_*)) &\text{with}& \nabla \cdot F_j^{\mathrm{M}} = \p_t (J_0^{\mathrm{M}}\tu_{j0}^{\mathrm{M}}) -g_j (\tu_0^{\mathrm{M}}).
		\end{array}
	\end{align} 
	Using $(\phi^+,\phi^{\mathrm{M}},0) \in \calH^m$ (resp.~$(0,\phi^{\mathrm{M}},\phi^-) \in \calH^m$) as a test function in \eqref{eq:macroscopic_fd_weak}, by the generalised divergence theorem we deduce
	\begin{align}
		&\dotproduct{\gamma_{\nu^\pm}(F_j^\pm)}{\phi_j^\pm}{H^{-\frac 12} (\p\Ome^\pm),\, H^{\frac 12}(\p\Ome^\pm)} \\
		& =\int_{\Ome^\pm} F_j^\pm \cdot \nabla \phi_j^\pm \dd x + \int_{\Ome^\pm} (\p_t \tu_{j0}^\pm -f_j (\tu_0^\pm)) \phi_j^\pm \dd x\\
		& = - \int_{\Sigma} \int_{Z_*} (\p_t(J_0^{\mathrm{M}} \tu_{j0}^{\mathrm{M}}) -J_0^{\mathrm{M}} g_{j}(\tu_0^{\mathrm{M}})) \phi_j^{\mathrm{M}} \dd z \dd x^\prime - \int_\Sigma \int_{Z_*} F_j^{\mathrm{M}} \cdot \nabla \phi_j^{\mathrm{M}} \dd z \dd x^\prime\\
		&\quad - \int_\Sigma \int_{N} J_0^{\mathrm{M}} \norm{F_0^{-\top} \nu^{\mathrm{M}}} {h}_{j}(\tu_0^{\mathrm{M}}) \phi_j^{\mathrm{M}} \dd \calH^{n-1}(z) \dd x^\prime\\
		& = - \int_\Sigma \dotproduct{\gamma_{\nu^*}(F_j^{\mathrm{M}})}{\phi_j^{\mathrm{M}}}{H^{-\frac 12} (\p Z_*),\, H^{\frac 12}(\p Z_*)} \dd x^\prime 
		- \int_\Sigma \int_{N} J_0^{\mathrm{M}} \norm{F_0^{-\top} \nu^{\mathrm{M}}} {h}_{j}(\tu_0^{\mathrm{M}}) \phi_j^{\mathrm{M}} \dd \calH^{n-1}(z) \dd x^\prime,
	\end{align} 
	and the proof is complete. 
\end{proof}

Higher regularity of the time derivatives $\p_t \tu_0^\dagger$, as it was assumed right from the beginning in \cite{NeuJae07}, may be obtained by increasing the regularity of the data, see e.g.~\cite[Chapter~V, Theorem~6.1]{Lio61}. Next, for sufficiently smooth solutions $\tu_0$ we aim to give further insight into the transmission condition obtained in \autoref{lem:eff_transmission} by a similar choice of test functions as in \cite{NeuJae07}.
\begin{rem}[Effective jump conditions for the total fluxes]\label{rmk:jump}
	Given $\psi \in \Cci((-H,H))$ with $\psi(x_n)=1$ for $|x_n| < \frac H2$, we define 
	\begin{align}
		\dotproduct{\gamma_{\nu^\pm}\rvert_\Sigma(F_j^\pm)}{\eta_j}{H^{-\frac 12} (\Sigma),\, H^{\frac 12}(\Sigma)} 
		\coloneqq \dotproduct{\gamma_{\nu^\pm}(F_j^\pm)}{\psi \eta_j}{H^{-\frac 12} (\p\Ome^\pm),\, H^{\frac 12}(\p\Ome^\pm)}
	\end{align} 
	for every $\eta \in \Cci(\Sigma)^m$. Thus, taking the $z$--independent test function $\phi=\psi \eta$ in \autoref{lem:eff_transmission}, we deduce
	\begin{align}
		&\dotproduct{\gamma_{\nu^\pm}\rvert_\Sigma(F_j^\pm)}{\eta_j}{H^{-\frac 12} (\Sigma),\, H^{\frac 12}(\Sigma)} \\
		&\quad= - \int_\Sigma \dotproduct{\gamma_{\nu^*}(F_j^{\mathrm{M}})}{1}{H^{-\frac 12} (\p Z_*),\, H^{\frac 12}(\p Z_*)} \, \eta_j \dd x^\prime
		- \int_\Sigma \int_{N} J_0^{\mathrm{M}} \norm{F_0^{-\top} \nu^{\mathrm{M}}} {h}_{j}(\tu_0^{\mathrm{M}}) \dd \calH^{n-1}(z) \eta_{j} \dd x^\prime.
	\end{align}
	In particular, for a sufficiently regular solution $\tu_0$ of \eqref{eq:macroscopic_fd_weak}, the macroscopic total flux $F_j^\pm$ at $\Sigma$ equals the averaged microscopic flux $F_j^{\mathrm{M}}$ over the top (resp.~bottom) of the standard cell, that is
	\begin{align}
		(D_j^\pm \nabla \tu_{j0}^\pm - \tu_{j0}^\pm q_j^\pm) \cdot n^\pm 
		= \int_{S_*^\pm} (\tD_{j0}^{\mathrm{M}} \nabla \tu_{j0}^{\mathrm{M}} - \tu_{j0}^{\mathrm{M}} \tq_{j0}^{\mathrm{M}}) \cdot n^\pm \dd \calH^{n-1}(z)\qquad \text{on } (0,T)\times \Sigma,
	\end{align}
	where $n^\pm = (0^\prime,\pm 1)\in \IR^n$, and $J_0^{\mathrm{M}}\equiv 1$ on $(0,T)\times \Sigma \times S_*^\pm$ and the boundary condition on $N$ for $\tu_0^{\mathrm{M}}$ have been used. This interface condition for the total fluxes thus complements the condition $\tu_0^\pm= \tu_0^{\mathrm{M}}$ on $(0,T)\times \Sigma \times S_*^\pm$ for the solutions, where we stress that it implies that $\tu_0^{\mathrm{M}}\rvert_{S_*^\pm}(\cdot_t,\cdot_{x^\prime},z)$ is independent of $z\in S_*^\pm$. Similarly, the jump of the total fluxes across $(0,T)\times \Sigma$ is given by
	\begin{align}
		(D_j^+ \nabla \tu_{j0}^+ - \tu_{j0}^+ q_j^+ - (D_j^- \nabla \tu_{j0}^- - \tu_{j0}^- q_j^-))\cdot n^+ 
		= \sumpm \int_{S_*^\pm} (\tD_{j0}^{\mathrm{M}} \nabla \tu_{j0}^{\mathrm{M}} - \tu_{j0}^{\mathrm{M}} \tq_{j0}^{\mathrm{M}}) \cdot n^+ \dd \calH^{n-1}(z).
	\end{align}
\end{rem}

\subsection{Formulation of the macroscopic model with cell problems evolving in time}\label{subsec:macro_prob_evolving}
In a last step, we formulate the macroscopic problem \eqref{eq:macroscopic_fd_weak} on the macroscopic domains $\Ome^\pm$ but with local cell problems posed on cells evolving in time and depending additionally on $\Sigma$. More precisely, for $(t,x^\prime) \in (0,T)\times\Sigma$, we define $Z_*(t,x^\prime)\usesym{Zstx} \coloneqq \psi_0(t,x^\prime, Z_*)$ and $N(t,x^\prime)\usesym{Nstx}\coloneqq \psi_0(t,x^\prime,N)$, and we write
\begin{align}
	\QTM \usesym{QTM} &\coloneqq \set{(t,x^\prime,z)}{t \in (0,T),\, x^\prime \in \Sigma, \, z \in Z_*(t,x^\prime)},\\
	\calN_{T} \usesym{NT} &\coloneqq \set{(t,x,z)}{t \in (0,T),\, x^\prime \in \Sigma, \, z \in N(t,x^\prime)}.
\end{align}
Then, denoting $\psi_0^{-1}(t,x^\prime,\cdot) \coloneqq \psi_0(t,x^\prime,\cdot)^{-1}$, we write $u_0=(\tu_0^+,\tu_0^{\mathrm{M}} \circ_z \psi_0^{-1},\tu_0^-)$ and define the macroscopic quantities
\begin{align}
	D_{j0}^{\mathrm{M}} \coloneqq \bar D_{j0}^{\mathrm{M}} \circ_z \psi_0^{-1}, \qquad
	q_{j0}^{\mathrm{M}} \coloneqq \bar q_{j0}^{\mathrm{M}} \circ_z \psi_0^{-1}, \qquad
	b_0^{\mathrm{M}} \coloneqq \p_t\psi_0 \circ_z \psi_0^{-1}, \qquad
	U_{j0,0}^{\mathrm{M}} \coloneqq \tilde{U}_{j0,0}^{\mathrm{M}} \circ_z \psi_0^{-1}(0,\cdot)
\end{align}
on the evolving macroscopic layer $\QTM$. We further recall that, by the assumptions on the microscopic transformation $\psi_\eps$, we have $\psi_0(t,\cdot)\rvert_{\Sigma \times S_*^\pm} = \mathrm{id}_{\Sigma\times Z_*}$; hence, the top and bottom of the evolving channel $Z_*(t,x^\prime)$ in $Z$ is again given by $S_*^\pm$. The reverse transformation $\psi_0^{-1}$ applied to the weak macroscopic problem in \autoref{cor:macro_prob_weak} yields the following weak equation on the macroscopic domains $\Ome^\pm$, coupled with an equation on $\QTM$.

\begin{cor}[Evolving macroscopic problem]\label{cor:ev_macro_prob}
	The function $u_0$ satisfies
	\begin{align}
		 u_0 \in L^2(0,T;[H^1(\Ome^+)^m\times L^2(\Sigma;H^1(Z_*(t,x^\prime))^m)\times H^1(\Ome^-)^m])
	\end{align}
    and $u_0^\pm = u_0^{\mathrm{M}}$ on $(0,T)\times \Sigma \times S_*^\pm$, and for all $\phi^\pm \in C^1([0,T]\times \overline{\Ome^\pm})^m$ and $\phi^{\mathrm{M}} \in C^1([0,T]\times \overline{\Sigma} \times \overline{Z})^m$ with $\phi^{\mathrm{M}} = \phi^\pm$ on $(0,T)\times \Sigma \times S_*^\pm$ and $\phi^\pm(T,\cdot)=\phi^{\mathrm{M}}(T,\cdot)=0$ we have
	\begin{align}
	\begin{split}\label{eq:macroscopic_weak}
		&- \sumpm \int_0^T \int_{\Ome^\pm} u_{j0}^\pm \p_t \phi_j^\pm \dd x \dd t
		- \int_0^T \int_\Sigma \int_{Z_*(t,x^\prime)} u_{j0}^{\mathrm{M}} \p_t \phi_j^{\mathrm{M}} \dd z \dd x^\prime \dd t\\
		&\quad + \sumpm \int_0^T \int_{\Ome^\pm} D_j^\pm \nabla u_{j0}^\pm \cdot \nabla \phi_j^\pm \dd x \dd t 
		+ \int_0^T \int_\Sigma \int_{Z_*(t,x^\prime)} D_{j0}^{\mathrm{M}} \nabla_z u_{j0}^{\mathrm{M}} \cdot \nabla_z \phi_j^{\mathrm{M}} \dd z \dd x^\prime \dd t \\
		&\quad - \sumpm \int_0^T \int_{\Ome^\pm} q_j^\pm u_{j0}^\pm \cdot \nabla \phi_j^\pm \dd x \dd t 
		- \int_0^T \int_\Sigma \int_{Z_*(t,x^\prime)} u_{j0}^{\mathrm{M}} q_{j0}^{\mathrm{M}} \cdot \nabla_z \phi_j^{\mathrm{M}} \dd z \dd x^\prime \dd t \\
		&= \sumpm \int_0^T \int_{\Ome^\pm} f_j(u_0^\pm) \phi_j^\pm \dd x \dd t
		  + \int_0^T \int_\Sigma \int_{Z_*(t,x^\prime)} g_j(u_0^{\mathrm{M}}) \phi_j^{\mathrm{M}} \dd z \dd x^\prime \dd t\\
		  &\quad- \int_0^T \int_\Sigma \int_{N(t,x^\prime)} h_j(u_0^{\mathrm{M}}) \phi_j^{\mathrm{M}} \dd \calH^{n-1}(z) \dd x^\prime \dd t\\
		&\quad + \sumpm \int_{\Ome^\pm} U_{j0}^\pm \phi_j^\pm(0) \dd x
		+ \int_\Sigma \int_{Z_*(0,x^\prime)} U_{j0,0}^{\mathrm{M}} \phi_j^{\mathrm{M}}(0) \dd z \dd x^\prime. 
		\end{split}
	\end{align}
\end{cor}
\begin{proof}
	As $\psi_0(t,\cdot)\rvert_{\Sigma \times S_*^\pm} = \mathrm{id}_{\Sigma\times Z_*}$, the functions $\tilde \phi = (\phi^+,\phi^{\mathrm{M}} \circ_z \psi_0, \phi^-) \in L^2(0,T;\calH^m)$, with $\phi^\pm,\, \phi^{\mathrm{M}}$ as above, are admissible in the weak formulation \eqref{eq:macroscopic_fd_weak}. Therefore, by applying the inverse transformation $\psi_0^{-1}$ and using transformation rules similar to that employed in \autoref{sec:existence} for the derivation of the weak microscopic problem on the reference domain, the result is obtained.
\end{proof}
We emphasise that information on the microscopic evolution of the channels is captured in the two-scale limit transformation $\psi_0$, and hence in the evolving microcells $Z_*(t,x^\prime)$ for $(t,x^\prime) \in (0,T)\times \Sigma$. The strong macroscopic problems on $\Ome^\pm$ are therefore coupled by local cell problems on $\QTM$ taking into account this evolution.

\begin{rem}[Strong formulation of the macroscopic problem]
	The weak formulation \eqref{eq:macroscopic_weak} is associated with the problem
	\begin{subequations}
		\begin{align}\label{eq:macro_time}
			\left\{\!\!\begin{array}{r@{\ }c@{\ }l@{\ }c@{\ \, }l}
				\p_t u_{j0}^\pm - \nabla \cdot(D_j^\pm \nabla u_{j0}^\pm - q_j^\pm u_{j0}^\pm) & = & f_j(u_0^\pm) && \text{in } (0,T)\times \Ome^\pm,\\
				(D_j^\pm \nabla u_{j0}^\pm - q_j^\pm u_{j0}^\pm)\cdot \nu^\pm & = & 0 && \text{on } (0,T)\times \p\Ome^\pm \setminus \Sigma,\\
				u_0^\pm(0) & = & {U}_0^\pm && \text{in } \Ome^\pm,
			\end{array}\right.
		\end{align}
		where the function $u_0^{\mathrm{M}}$ solves for $(t,x^\prime) \in (0,T)\times \Sigma$ the local cell problems on $\QTM$ given by
		\begin{align}\label{eq:cell_time}
			\left\{\!\!\begin{array}{r@{\ }c@{\ }l@{\ }c@{\ \, }l}
				\p_t u_{j0}^{\mathrm{M}} - \nabla_z \cdot(D_{j0}^{\mathrm{M}} \nabla_z u_{j0}^{\mathrm{M}} - u_{j0}^{\mathrm{M}} q_{j0}^{\mathrm{M}}) &=& g_j(u_0^{\mathrm{M}}) && \text{in } \QTM,\\
				-(D_{j0}^{\mathrm{M}} \nabla_z u_{j0}^{\mathrm{M}} - u_{j0}^{\mathrm{M}} q_{j0}^{\mathrm{M}} + u_{j0}^{\mathrm{M}} b_0^{\mathrm{M}}) \cdot \nu^{\mathrm{M}}  & = & h_{j}(u_0^{\mathrm{M}}) && \text{on } \calN_{T},\\
				u_0^{\mathrm{M}}(0) & = & U_{0,0}^{\mathrm{M}} && \text{in } \set{(x^\prime,z)}{x^\prime \in \Sigma, \, z\in Z_*(0,x^\prime)}.
			\end{array}\right.
		\end{align} 
		The systems \eqref{eq:macro_time} and \eqref{eq:cell_time} are coupled by the transmission conditions
		\begin{align}\label{eq:TC_time}
			\left\{\!\!\begin{array}{r@{\ }c@{\ }l@{\ }c@{\ \, }l}
				u_0^\pm & = & u_0^{\mathrm{M}} && \text{on } (0,T)\times \Sigma\times S_*^\pm,  \\
				\left(D_{j}^\pm \nabla u_{j0}^\pm - u_{j0}^\pm q_{j}^\pm \right) \cdot n^\pm & = & \int_{S_*^\pm} \left(D_{j0}^{\mathrm{M}} \nabla u_{j0}^{\mathrm{M}} - u_{j0}^{\mathrm{M}} q_{j0}^{\mathrm{M}}\right) \cdot n^\pm \dd \calH^{n-1}(z)&& \text{on } (0,T)\times \Sigma,
			\end{array}\right.
		\end{align}
		which give rise to jump conditions for the flux across the hypersurface $\Sigma$.
	\end{subequations}
\end{rem}

\section{Conclusion}
We derived a macroscopic model for a coupled system of reaction--diffusion--advection equations posed on two bulk domains connected via channels within a thin layer. The microscopic (channel) structure evolves in time and its evolution is known a priori. We considered nonlinear (boundary) reaction rates as well as a critical, low diffusivity of order~$\eps$ inside the layer, causing the limit function to depend on both the macroscopic and the microscopic variable. The resulting macroscopic equations on the bulk regions are coupled through effective transmission conditions on the interface $\Sigma$, expressed in terms of local cell problems on a space--time-dependent reference cell $Z_*(t,x^\prime)$. This reference cell depends on the two-scale limit of the microscopic evolution, with the macroscopic variable $x^\prime \in \Sigma$ acting as a parameter, and thus encodes information on the evolving microstructure. 
The macroscopic problem was obtained by first proving an abstract (two-scale) convergence results for sequences taking values in an appropriately scaled function space and then identifying the system of equations satisfied by the limit functions, with all convergence properties derived solely from a priori estimates for the solutions of the microscopic problem.

In modelling applications, the evolution of the microstructure is often not known a priori but determined by different physical, chemical or mechanical processes taking place on the lateral boundary of the channels. This leads to a strongly coupled nonlinear system including free boundaries, which provides an interesting potential extension of our analysis. Besides having to deal with the nonlinear coupling, a suitable construction of the microscopic transformation depending on the processes on the channel walls is an essential challenge in this task; nevertheless, the results obtained in this work provide a firm basis for such an undertaking.

\section*{Acknowledgements}
The authors thank Markus Gahn (University of Augsburg) for valuable discussions and comments.

\printbibliography

\appendix

\section{Two-scale convergence for (non-evolving) thin channels}\label{app:two-scale}
In this appendix, we summarise the most important results on two-scale convergence in (non-evolving) thin channels (see \autoref{def:2s_conv}) using the notation introduced in \autoref{subsec:shifts} and \autoref{subsec:convergence}. We mostly skip the proofs but give references to the literature.

\subsection{Two-scale convergence and periodic unfolding}
The following lemma collects compactness properties of (weak) two-scale convergence based on a priori bounds, which for simplicity we only state for the case $p=2$. Partial results for $p \in (1,\infty)$ are given in \cite[Lemma~3.2]{GahNeu21}.

\begin{lemma}[Two-scale compactness]\label{lem:2s_cpct} 
	\text{} \vspace{-0.5em}
	\begin{enumerate}[label=(\roman*)]
		\item\label{it:2s_cpct_L2} If $v_\eps \in L^2((0,T)\times \OmeMS)$ satisfies $\frac{1}{\sqrt{\eps}} \Lpnorm{v_\eps}{2}{(0,T)\times \OmeMS} \lesssim 1$, there exists $v_0 \in L^2((0,T)\times \Sigma \times Z_*)$ such that, for a subsequence, $v_\eps \to v_0$ in the two-scale sense.
		\item\label{it:2s_cpct_L2N} If $v_\eps \in L^2((0,T)\times N_\eps)$ satisfies $\Lpnorm{v_\eps}{2}{(0,T)\times N_\eps} \lesssim 1$, there exists $v_0 \in L^2((0,T)\times \Sigma \times N)$ such that, for a subsequence, $v_\eps \to v_0$ in the two-scale sense on $N_\eps$.
		\item\label{it:2s_cpct_H1} If $v_\eps \in L^2(0,T;H^1(\OmeMS))$ satisfies
		\begin{align}
			\frac{1}{\sqrt{\eps}} \Lpnorm{v_\eps}{2}{(0,T)\times \OmeMS} + \sqrt{\eps} \Lpnorm{\nabla v_\eps}{2}{(0,T)\times \OmeMS} \lesssim 1,
		\end{align}
		there exists $v_0 \in L^2(0,T;L^2(\Sigma;H^1(Z_*)))$ such that, for a subsequence, $v_\eps \to v_0$ and $\eps \nabla v_\eps \to \nabla_z v_0$ in the two-scale sense.
		
		\item\label{it:2s_cpct_pt} If $v_\eps \in L^2(0,T;H^1(\OmeMS))$ with $\p_tv_\eps \in L^2(0,T;(\calH_{\eps,0}^{\mathrm{M}})^\prime)$ satisfies
		\begin{align}
		    \frac 1\eps \Lpnorm{\p_t v_\eps}{2}{0,T;(\calH_{\eps,0}^{\mathrm{M}})^\prime} + \frac{1}{\sqrt{\eps}} \Lpnorm{v_\eps}{2}{(0,T)\times \OmeMS} + \sqrt{\eps} \Lpnorm{\nabla v_\eps}{2}{(0,T)\times \OmeMS} \lesssim 1,
		\end{align}
		there exists $v_0 \in L^2(0,T;L^2(\Sigma;H^1(Z_*)))$ with $\p_t v_0 \in L^2(0,T;L^2(\Sigma;(H^1_{\pm,0}(Z_*))^\prime)$ such that, for a subsequence, $v_\eps \to v_0$ and $\eps \nabla v_\eps \to \nabla_z v_0$ in the two-scale sense. 
	\end{enumerate}
\end{lemma}

\begin{proof}
	The compactness results \ref{it:2s_cpct_L2} and \ref{it:2s_cpct_L2N} are stated in \cite[Theorem~4.4]{BhaGahNeu22} without proof as it is similar to \cite[Proposition~4.2]{NeuJae07}. Note that the latter one is based on an oscillation lemma on the lateral boundary, see \cite[Lemma~4.3]{BhaGahNeu22}. The proof of \ref{it:2s_cpct_H1} can be found in \cite[Lemma~3.2]{GahNeu21}. 
	The proof of \ref{it:2s_cpct_pt} is similar as in \cite[Lemma~B.5]{GahNeu25} (also see \cite[Proposition~4]{GahNeuPop21} for the case of bulk regions and \autoref{thm:convergence}) by adjusting the space of test functions accordingly.
\end{proof}

Next, we introduce the unfolding operator $\calT_\eps$ for thin channels which allows to establish (strong) two-scale convergence of a sequence $(v_\eps)_\eps$ by showing (strong) $L^p$--convergence of the unfolded sequence $(\calT_\eps v_\eps)_\eps$, see \autoref{lem:char_2s_conv} below.
\begin{definition}[Unfolding operator, {\cite[Definition~3.3]{GahNeu21}}]\label{def:unfolding}
	Let $(G_\eps,G)\in \{(\OmeMS,Z_*),(N_\eps,N)\}$ and $p\in [1,\infty]$. The \underline{\em unfolding operator} $\calT_\eps \colon L^p((0,T)\times G_\eps) \to L^p((0,T)\times \Sigma \times G)$ is given by
	\begin{align}
		\calT_\eps v_\eps(t,x^\prime,z) \coloneqq v_\eps\left(t,\eps\left(\floor*{\tfrac{x^\prime}{\eps}},0\right)+\eps z\right) \qquad \text{for } (t,x^\prime,z) \in (0,T)\times \Sigma \times G,
	\end{align}
	where $\floor*{\cdot}$ denotes the integer part.
\end{definition}
We emphasise that time only acts as a parameter in the preceding definition of the unfolding operator, hence the following results are also valid in the stationary case. As mentioned in \cite[p.~1588]{GahNeu21}, for functions in $L^p(0,T;W^{1,p}(\OmeMS))$ it makes sense not to distinguish the notation of the unfolding operator on $\OmeMS$ and on $N_\eps$ as it commutes with the trace operator.
We have the following important properties of $\calT_\eps$, see \cite[Lemma~4.6]{NeuJae07} and \cite[Lemma~3.4]{GahNeu21} as well as \cite[Chapter~1]{CioDamGri18} for similar results in the case of bulk domains.
\begin{lemma}[Properties of $\calT_\eps$]\label{lem:propT}Let $p\in [1,\infty)$.
	\begin{enumerate}[label=(\roman*)]
		\item\label{it:lem_prop_T1} For any $v_\eps \in L^p((0,T)\times \OmeMS)$, we have $\Lpnorm{\calT_\eps v_\eps}{p}{(0,T)\times\Sigma \times Z_*} = \eps^{-\frac 1p} \Lpnorm{v_\eps}{p}{(0,T)\times \OmeMS}$.
		\item For any $v_\eps \in L^p((0,T)\times N_\eps)$, we have $\Lpnorm{\calT_\eps v_\eps}{p}{(0,T)\times\Sigma \times N} = \Lpnorm{v_\eps}{p}{(0,T)\times N_\eps}$.
		\item\label{it:lem_prop_T2} For any $v_\eps \in L^p(0,T;W^{1,p}(\OmeMS))$, we have $\nabla_z (\calT_\eps v_\eps) = \eps \calT_\eps(\nabla_{x} v_\eps)$.
	\end{enumerate}
\end{lemma}

\begin{lemma}[Characterisation of two-scale convergence]\label{lem:char_2s_conv}
	Let $(G_\eps,G)\in \{(\OmeMS,Z_*),(N_\eps,N)\}$ and $p\in (1,\infty)$. 
	Then, for functions $v_\eps \in L^p((0,T)\times G_\eps)$ and $v_0 \in L^p((0,T)\times \Sigma \times G)$ the following assertions are equivalent:
	\begin{enumerate}[label=(\roman*)]
		\item\label{it:prop_conv_TTe1} The convergence $\calT_\eps v_\eps \to v_0$ holds weakly in $L^p((0,T)\times \Sigma \times G)$.
		\item\label{it:prop_conv_TTe2} The convergence $v_\eps \to v_0$ holds in the two-scale sense in $L^p$.
	\end{enumerate}
	The equivalence above also holds true if weak (two-scale) convergence is replaced by strong (two-scale) convergence.
\end{lemma}
\begin{proof}
	A proof of the result for the weak convergences on bulk domains is given in \cite[Proposition~1.19]{CioDamGri18} and can easily be adapted to thin domains and the lateral boundaries of the channels. We also refer to \cite[Proposition~4.7]{NeuJae07}) for the case $p=2$ and thin channels. For the result regarding the strong convergences, it is used that weak convergence together with convergence of the norms is equivalent to strong convergence in uniformly convex Banach spaces, see \cite[Proposition~3.32]{Bre2011}.
\end{proof}

Based on the characterisation of two-scale convergence by convergence of the unfolded sequence, the next lemma gives a useful criterion for establishing two-scale convergence of products.
\begin{lemma}[Two-scale convergence of products]\label{lem:2s_prod}
	Let $p,\, q \in [1,\infty)$ and $v_\eps \in L^p((0,T)\times \OmeMS)$ and $w_\eps \in L^\infty((0,T)\times \OmeMS)$. Assuming that
	\begin{enumerate}[label=(\roman*)]
		\item there exists $v_0 \in L^p((0,T)\times \Sigma \times Z_*)$ with $v_\eps \to v_0$ in the (strong) two-scale sense in $L^p$,
		\item $\Lpnorm{w_\eps}{\infty}{\OmeMS} \lesssim 1$ and there exists $w_0 \in L^{\infty}((0,T)\times \Sigma \times Z_*)$ with $w_\eps \to w_0$ in the strong two-scale sense in $L^{q}$,
	\end{enumerate}
	it also holds $v_\eps w_\eps \to v_0 w_0$ in the (strong) two-scale sense in $L^p$.
\end{lemma}
\begin{proof}
	We argue similarly as in \cite[Lemma~1.16]{Wie24}, where a similar statement was proven for porous bulk domains, and begin with the case of strong two-scale convergence of $(v_\eps)_\eps$. A result for products of multiple sequences is given in \cite[Proposition 8]{ReiPet22}. According to \autoref{lem:char_2s_conv}, it suffices to show $\calT_\eps(v_\eps w_\eps) \to v_0 w_0$ in $L^p((0,T)\times \Sigma \times Z_*)$. As     
	we have $\calT_\eps v_\eps \to v_0$ in $L^p((0,T)\times \Sigma \times Z_*)$ and $\calT_\eps w_\eps \to w_0$ in $L^{q}((0,T)\times \Sigma \times Z_*)$, it follows $\calT_\eps (v_\eps w_\eps) \to v_0 w_0$ in $L^r((0,T)\times \Sigma \times Z_*)$ with $\frac 1r = \frac 1p + \frac 1q$ and, hence, pointwise almost everywhere for a subsequence. From the boundedness of $(\calT_\eps w_\eps)_\eps$ in $L^\infty((0,T)\times \Sigma \times Z_*)$ and a variant of the dominated convergence theorem (see Pratt's theorem \cite{Pra60} in combination with Scheffé's lemma \cite[Section~3]{Rie29}), we thus obtain $\calT_\eps(v_\eps w_\eps) \to v_0 w_0$ in $L^p((0,T)\times \Sigma \times Z_*)$ for a subsequence. As this is valid for any subsequence, the convergence holds true for the whole sequence.
	
	To prove the statement regarding the weak two-scale convergence, we only need to show $\calT_\eps(v_\eps w_\eps) \to v_0 w_0$ weakly in $L^p((0,T)\times \Sigma \times Z_*)$. For every $\phi\in L^{p^\prime}((0,T)\times \Sigma \times Z_*)$, by the dominated convergence theorem and the boundedness of $(\calT_\eps w_\eps)_\eps$ in $L^\infty((0,T)\times \Sigma \times Z_*)$, we have $(\calT_\eps w_\eps) \phi \to w_0 \phi$ in $L^{p^\prime}((0,T)\times \Sigma \times Z_*)$ along a subsequence. By the continuity of the dual pairing, we thus deduce
	\begin{align}
		\int_0^T \int_{\Sigma} \int_{Z_*} \calT_\eps(v_\eps w_\eps) \phi \dd z \dd x^\prime \dd t 
		= \int_0^T \int_{\Sigma} \int_{Z_*} \calT_\eps v_\eps \, ((\calT_\eps w_\eps) \phi) \dd z \dd x^\prime \dd t
		\to \int_0^T \int_{\Sigma} \int_{Z_*} v_0 w_0 \phi \dd z \dd x^\prime \dd t
	\end{align}
	along a subsequence and conclude as before.
\end{proof}

\subsection{Strong two-scale convergence}
In this section, we present an argument to obtain strong two-scale convergence based on a priori estimates, see \autoref{thm:strong_2s_cpctness} below. The following lemma shows why this is sufficient in order to pass to the limit in suitably well-behaved nonlinearities.
\begin{lemma}[Strong two-scale convergence of nonlinearities]\label{lem:s_2s_conv_NL}
	Let $p\in [1,2]$ and suppose that $g$ and $h$ satisfy \ref{it:propfg}. Then, for every sequence of functions $v_\eps \in L^2((0,T)\times \OmeMS)$ which converges to $v_0 \in L^2((0,T)\times \Sigma \times Z_*)$ strongly in the two-scale sense in $L^p$, we have
	\begin{align}
		g(v_\eps) \to g(v_0) \qquad \text{strongly in the two-scale sense in $L^p$.}
	\end{align}
	Similarly, if $(v_\eps)_\eps$ converges to $v_0 \in L^2((0,T)\times \Sigma \times N)$ strongly in the two-scale sense on $N_\eps$ in $L^p$, we have
	\begin{align}
		h(v_\eps) \to h(v_0) \qquad \text{strongly in the two-scale sense on $N_\eps$ in $L^p$.}
	\end{align}
\end{lemma}

\begin{proof}
	The result is an immediate consequence of \autoref{lem:char_2s_conv} and the theory of Nemytskii operators (see e.g.~\cite[Proposition~26.6]{ZeiIIB_90}), noting that the application of $\calT_\eps$ commutes with the nonlinear functions $g$ and $h$. 
\end{proof}
Aiming to establish strong two-scale convergence of a sequence of functions $v_\eps \in L^2(0,T;H^1(\OmeMS))$, by \autoref{lem:char_2s_conv} it is sufficient to prove strong convergence of the unfolded sequence $(\calT_\eps v_\eps)_\eps$. Due to the low regularity of $\calT_\eps v_\eps$ with respect to the variable $x^\prime\in \Sigma$, we cannot apply a standard Aubin--Lions argument and use a Kolmogorov-type argument for Banach space-valued functions instead, see \cite{GahNeu16}. For this, the main ingredients are $\eps$-uniform estimates on the time derivatives $\p_t(\calT_\eps v_\eps)$ and control of shifts with respect to the $x^\prime$--variable, where we follow the reasoning of \cite{GahNeu21}.
Estimates for the time derivatives are obtained by considering the $L^2$-adjoint of $\calT_\eps$, see \cite{GahNeu21}. 

\begin{definition}[Averaging operator]
	The $L^2$-adjoint $\calU_\eps$ of $\eps \calT_\eps$ is called \underline{\em averaging operator}, that is $\calU_\eps \colon L^2((0,T)\times \Sigma \times Z_*) \to L^2((0,T)\times \OmeMS)$ is a linear and bounded operator such that for all $v_\eps \in L^2((0,T)\times \OmeMS)$ and $\phi \in L^2((0,T)\times \Sigma \times Z_*)$ we have
	\begin{align}
		\scp{\calU_\eps \phi}{v_\eps}{L^2((0,T)\times \OmeMS)} 
		= \eps \scp{\phi}{\calT_\eps v_\eps}{L^2((0,T)\times \Sigma \times Z_*)}.
	\end{align}
\end{definition}
By a computation similar to that in \cite[Chapter~1.3]{CioDamGri18}, one obtains (also see \cite[p.~1590]{GahNeu21})
\begin{align}
	\calU_\eps \phi (t,x) 
	= \int_Y \phi\left(t,\eps\left(y+ \floor*{\tfrac{x^\prime}{\eps}}\right),\left(\left\{\tfrac{x^\prime}{\eps}\right\},\tfrac{x_n}{\eps}\right)\right) \dd y \qquad \text{ for } (t,x)\in (0,T)\times \OmeMS,
\end{align}
but we will not make use of this explicit representation.
From the definition of $\calU_\eps$, we directly obtain the following quantitative norm estimate.
\begin{cor}[Norm of $\calU_\eps$, {\cite[Corollary~3.8]{GahNeu21}}]\label{cor:est_U}
	For all $\phi \in L^2((0,T)\times \Sigma\times Z_*)$, the following estimate holds:
	\begin{align}
		\Lpnorm{\calU_\eps \phi}{2}{(0,T)\times \OmeMS} \leq \sqrt{\eps} \Lpnorm{\phi}{2}{(0,T)\times \Sigma \times Z_*}.
	\end{align}
\end{cor}

The estimates for $\p_t (\calT_\eps \tu_\eps^{\mathrm{M}})$ are then based on the following functional analytic result, which for given functions $u$ and an operator $A$ acting on $u$ allows to compute the time derivative of $Au$ in terms of $\p_t u$ and the adjoint operator $A^*$ of $A$.
\begin{lemma}[Time derivative of $Au$, {\cite[Lemma~3.7]{GahNeu21}}]\label{lem:pt_A}
	Let $V,\, W$ be reflexive, separable Banach spaces and $Y,\, X$ Hilbert spaces such that we have the Gelfand triples $V\xhookrightarrow{} Y \xhookrightarrow{} V^\prime$ and $W \xhookrightarrow{} X \xhookrightarrow{} W^\prime$.
	Moreover, assume that for a linear and bounded operator $A\in \calL(Y,X)$ for its adjoint $A^*$ we have $A^*\rvert_W \in \calL(W;V)$.
	Then, for all $u\in L^2(0,T;Y)$ with $\p_t u \in L^2(0,T;V^\prime)$ there holds $\p_t(Au)\in L^2(0,T;W^\prime)$ with
	\begin{align}
		\dotproduct{\p_t(Au)}{w}{W^\prime,W}=\dotproduct{\p_t u}{A^*w}{V^\prime,\, V} \qquad \text{ for all } w\in W.
	\end{align}
	Consequently, we have the estimate $\Lpnorm{\p_t(Au)}{2}{0,T;W^\prime} \leq \norms{A^*\rvert_W}{\calL(W;V)} \Lpnorm{\p_t u}{2}{0,T;V^\prime}$. 
\end{lemma}
\begin{proof}
    The proof is based on the uniqueness of the generalised time derivative; we omit its details.
\end{proof}

As mentioned in \cite[p.1591]{GahNeu21}, concerning the spatial regularity of $\calU_\eps$ we crucially need that the lateral channel boundary $N$ of the channel $Z_*$ does not touch the lateral boundary of the standard cell $Z$ as otherwise one has to consider functions with vanishing trace on $N \cap \p Z$.
Since we aim to obtain an estimate for $\p_t (\calT_\eps v_\eps)$, where $\p_t v_\eps \in L^2(0,T;(\calH_{\eps,0}^{\mathrm{M}})^\prime)$, we restrict the domain of $\calU_\eps$ to functions vanishing on the top and bottom $S_*^\pm$ of the channel in $Z$ in what follows, although this assumption is not needed for the spatial regularity result stated below. 
\begin{lemma}[Spatial regularity of $\calU_\eps$]\label{lem:reg_U}
	The averaging operator $\calU_\eps$ is a linear and bounded operator 
	\begin{align}
		\calU_\eps \colon L^2((0,T)\times \Sigma; H^1_{\pm,0}(Z_*)) \to L^2(0,T;\calH_{\eps,0}^{\mathrm{M}}).
	\end{align}
	Moreover, for all $\phi \in L^2((0,T)\times \Sigma; H^1_{\pm,0}(Z_*))$ we have $\nabla_{x}(\calU_\eps \phi) = \frac{1}{\eps} \calU_\eps(\nabla_{z} \phi)$ and the following estimate holds:
	\begin{align}
		\Lpnorm{\calU_\eps \phi}{2}{0,T;\calH_{\eps,0}^{\mathrm{M}}}
		\leq \Lpnorm{\phi}{2}{(0,T)\times \Sigma;H^1_{\pm,0}(Z_*)}.
	\end{align}
\end{lemma}
\begin{proof}
	Following the arguments given in \cite[Proposition~6]{GahNeuKna18}, for $\phi \in L^2((0,T)\times \Sigma; H^1_{\pm,0}(Z_*))$ and $v_\eps \in \Cci((0,T)\times \OmeMS)$, we calculate
	\begin{align}
		\int_0^T \int_{\OmeMS} (\calU_\eps\phi)(t,x) \p_{x_i} v_\eps(t,x) \dd x \dd t 
		&= \eps \int_0^T \int_\Sigma \int_{Z_*} \phi(t,x^\prime,z) \calT_\eps(\p_{x_i} v_\eps)(t,x^\prime,z) \dd z \dd x \dd t\\
		&= - \int_0^T \int_\Sigma \int_{Z_*} \p_{z_i} \phi(t,x^\prime,z) (\calT_\eps v_\eps)(t,x^\prime, z) \dd z \dd x \dd t\\
		& \quad + \int_0^T \int_\Sigma \int_{\p Z_*} \phi(t,x^\prime,z) (\calT_\eps v_\eps)(t,x^\prime, z) \nu_i(z) \dd \calH^{n-1}(z) \dd x \dd t,
	\end{align}
	where we have used \autoref{lem:propT}~\ref{it:lem_prop_T2} and integrated by parts in the last step. By the choice of $v_\eps$, we have $\calT_\eps v_\eps \rvert_{\p Z_*} = 0$; thus, the boundary integral vanishes and the weak differentiability of $\calU_\eps \phi$ is proven. Consequently, for every $\phi\in L^2((0,T)\times\Sigma;H^1_{\pm,0}(Z_*))$ we calculate 
	\begin{align}
		\Lpnorm{\calU_\eps \phi}{2}{0,T;\calH_{\eps,0}^{\mathrm{M}}}^2
		&= \frac 1 {\eps} \Lpnorm{\calU_\eps \phi}{2}{(0,T)\times\OmeMS}^2 + \frac 1\eps \Lpnorm{\calU_\eps(\nabla_z \phi)}{2}{(0,T)\times \OmeMS}^2
		\leq \Lpnorm{\phi}{2}{(0,T)\times\Sigma;H^1_{\pm,0}(Z_*)},
	\end{align}
	where we have used \autoref{cor:est_U} for the last estimate.
\end{proof}

To apply \autoref{lem:pt_A} to the unfolded sequence $(\calT_\eps v_\eps)_\eps$, we consider $\calT_\eps$ and $\calU_\eps$ as stationary operators, which is possible due to their definition in which time only acts as an additional parameter. 
\begin{cor}[Estimate for $\p_t(\calT_\eps v_\eps)$, {\cite[Proposition~3.10]{GahNeu21}}]\label{cor:pt_T}
	Let $v_\eps \in L^2((0,T)\times \OmeMS)$ with $\p_t v_\eps \in L^2(0,T;(\calH_{\eps,0}^{\mathrm{M}})^\prime)$. 
	Then, $\p_t(\calT_\eps v_\eps) \in L^2(0,T;L^2(\Sigma,H^1_{\pm,0}(Z_*))^\prime)$ and, for all $\phi \in L^2(\Sigma; H^1_{\pm,0}(Z_*))$ and almost every $t\in (0,T)$, there holds 
	\begin{align}\label{eq:char_pt_T}
		\dotproduct{\p_t(\calT_\eps v_\eps)}{\phi}{L^2(\Sigma;(H^1_{\pm,0}(Z_*))^\prime),\, L^2(\Sigma; H^1_{\pm,0}(Z_*))}
		= \frac 1 \eps \dotproduct{\p_t v_\eps}{\calU_\eps \phi}{(\calH_{\eps,0}^{\mathrm{M}})^\prime,\, \calH_{\eps,0}^{\mathrm{M}}}.
	\end{align}
	In particular, we have
	\begin{align}\label{eq:norm_pt_T}
		\Lpnorm{\p_t(\calT_\eps v_\eps)}{2}{0,T;L^2(\Sigma;(H^1_{\pm,0}(Z_*))^\prime)} \leq \frac 1 \eps \Lpnorm{\p_t v_\eps}{2}{0,T;(\calH_{\eps,0}^{\mathrm{M}})^\prime}.
	\end{align}
\end{cor}
\begin{proof}
	In view of \autoref{lem:pt_A}, we set
	\begin{align}
		V= \calH_{\eps,0}^{\mathrm{M}}, \qquad Y = L^2(\OmeMS), \qquad W = L^2(\Sigma; H^1_{\pm,0}(Z_*)), \qquad X = L^2(\Sigma \times Z_*)
	\end{align}
	and $A=\calT_\eps \in \calL(Y;X)$ with $A^* = \frac 1 \eps \calU_\eps$. Then, by \autoref{lem:reg_U}, we have that the restriction $A^*\rvert_W=\tfrac 1\eps \calU_\eps\rvert_W \in \calL(W;V)$ is well-defined. Consequently, $\p_t(\calT_\eps v_\eps) \in L^2(0,T;L^2(\Sigma;(H^1_{\pm,0}(Z_*))^\prime))$ and \eqref{eq:char_pt_T} holds. The bound \eqref{eq:norm_pt_T} is then a direct consequence of the estimates in \autoref{lem:pt_A} and \autoref{lem:reg_U}.
\end{proof}

The second tool for the application of the Kolmogorov-type compactness result from \cite{GahNeu16} is a good control of shifts with respect to the $x^\prime$-variable.
First, we show that shifts of $\calT_\eps v_\eps$ with respect to $x^\prime \in \Sigma$ are estimated by shifts of $v_\eps$ into direction of $(N^\prime,0)\eps$ with $N^\prime \in \IZ^{n-1}$.

\begin{lemma}[Shifts of the unfolded function, {\cite[Lemma~3.11]{GahNeu21}}]\label{lem:shifts_Tv}
	For $h\in (0,1)$ sufficiently small, there exist $\eps_0, \xi_0 >0$ such that for all $\eps\ll \eps_0$, $\xi^\prime \in \IR^{n-1}$ with $|\xi^\prime|\ll \xi_0$ and $v_\eps \in L^2((0,T)\times \OmeMS)$ there holds
	\begin{align}
		\Lpnorm{\calT_\eps v_\eps(\cdot, \cdot + \xi^\prime,\cdot)- \calT_\eps v_\eps}{2}{(0,T)\times \Sigma_{2h}\times Z_*}^2 \leq \frac 1 \eps \sum_{m^\prime \in \{0,1\}^{n-1}} \Lpnorm{\delta_{l^\prime,\eps} v_\eps}{2}{(0,T)\times \OmeMSh}^2,
	\end{align}
	where $l^\prime=l^\prime(\eps,\xi^\prime,m^\prime)= m^\prime  + \floor*{\frac{\xi^\prime}{\eps}}\in \IZ^{n-1}$.
\end{lemma}
\begin{proof}
	The proof follows the ideas in \cite[p.~709]{NeuJae07} by introducing a special decomposition of each microcell $\eps(Y+k^\prime)$, which allows to estimate $\floor*{\frac{x^\prime + \xi^\prime}{\eps}}$ for $x^\prime \in \eps(Y+k^\prime)$, and uses $\Sigma_{2h} \subset \widehat{\Sigma}_{\eps,h}$ for $\eps>0$ sufficiently small.
\end{proof}

The main result for strong two-scale compactness is the following. Its proof is similar to that of \cite[Theorem~3.12]{GahNeu21}, but we give some details for the main steps for the convenience of the reader. A similar result is established in \cite[Lemma~4.2]{AmaAntPanPia10} but based on $L^\infty$-estimates for the sequence $(v_\eps)_\eps$, which however we cannot guarantee in our setting.
\begin{theorem}[Strong two-scale compactness]\label{thm:strong_2s_cpctness}
	Let $p\in [1,2)$, $\beta \in \left(\frac 12 ,1\right)$ and $v_\eps \in L^2(0,T;H^1(\OmeMS))$ with $\p_t v_\eps \in L^2(0,T;(\calH_{\eps,0}^{\mathrm{M}})^\prime)$ such that
	\begin{enumerate}[label=\roman*)]
		\item\label{it:strong_2s_cpct_apriori} we have the estimate
		\begin{align}
			\frac 1\eps \Lpnorm{\p_t v_\eps}{2}{0,T;(\calH_{\eps,0}^{\mathrm{M}})^\prime} + \frac{1}{\sqrt{\eps}} \Lpnorm{v_\eps}{2}{(0,T)\times \OmeMS} + \sqrt{\eps} \Lpnorm{\nabla v_\eps}{2}{(0,T)\times \OmeMS} \lesssim 1,
		\end{align}
		\item\label{it:strong_2s_cpct_shifts} for all $h\in (0,1)$ sufficiently small and $l^\prime \in \IZ^{n-1}$ with $|l^\prime \eps|\ll h$ the following convergence holds:
		\begin{align}
			\frac 1{\sqrt{\eps}} \Lpnorm{\delta_{l^\prime,\eps} v_\eps}{2}{(0,T)\times \OmeMSh} +\sqrt{\eps} \Lpnorm{\delta_{l^\prime,\eps} \nabla v_\eps}{2}{(0,T)\times \OmeMSh} \xrightarrow{\eps l^\prime \to 0} 0.
		\end{align}
	\end{enumerate}
	Then, there exists $v_0 \in L^2(0,T;L^2(\Sigma;H^1(Z_*)))$ such that, at least for a subsequence, the following convergences hold:
	\begin{align}
		\begin{array}{rclcl}
			v_{\eps} &\to& v_0 && \text{in the two-scale sense},\\
			\eps \nabla v_\eps &\to& \nabla_z v_0 && \text{in the two-scale sense},\\
			\calT_\eps v_\eps &\to& v_0 && \text{strongly in } L^p(\Sigma;L^2(0,T;H^\beta(Z_*))).
		\end{array}
	\end{align}
\end{theorem}
\begin{proof}
	The two-scale convergence of $(v_\eps)_\eps$ and $(\eps \nabla v_\eps)_\eps$ follows directly from \autoref{lem:2s_cpct}~\ref{it:2s_cpct_H1}.	
	In order to establish the strong convergence of the sequence
	\begin{align}
		(\calT_\eps v_\eps)_\eps \subset L^2(0,T;L^2(\Sigma; H^1(Z_*))) \xhookrightarrow{} L^p(\Sigma; L^2(0,T;H^\beta(Z_*))),
	\end{align}
	we use the Kolmogorov-type compactness result given in \cite[Theorem~2.2]{GahNeu16} and have to check the following two conditions:
	\begin{enumerate}[label=(K\arabic*)]
		\item\label{it:K1} For every rectangle $C\subset \Sigma$, the sequence $(\int_C \calT_\eps v_\eps \dd x^\prime)_\eps$ is relatively compact in $L^2(0,T;H^\beta(Z_*))$.
		\item\label{it:K2} For every $\xi^\prime \in \IR^{n-1}$ with $0\leq \xi_i^\prime \leq 1$, $i=1,\ldots,n$, we have
		\begin{align}\label{eq:char_pt}
			\sup_{\eps>0} \Lpnorm{\calT_\eps v_\eps(\cdot,\cdot+\xi^\prime,\cdot) - \calT_\eps v_\eps}{p}{\Sigma_{\xi^\prime}; L^2(0,T;H^\beta(Z_*))} \xrightarrow{\xi^\prime \to 0} 0,
		\end{align}
	\end{enumerate}
	where $\Sigma_{\xi^\prime} = \Sigma \cap (\Sigma - \xi^\prime)$. 
	By the compactness of the embedding $H^1(Z_*)\xhookrightarrow{} H^\beta(Z_*)$ and the continuity of $H^\beta(Z_*)\xhookrightarrow{} H^1_{\pm,0}(Z_*)^\prime$, due to the Aubin-Lions lemma it suffices to show the boundedness of $\left(\int_C \calT_\eps v_\eps \dd x^\prime\right)_\eps$ in $W^{1,2,2}(0,T;H^1(Z_*),H^1_{\pm,0}(Z_*)^\prime)$ in order to prove \ref{it:K1}. This follows the same arguments as in \cite[Theorem~3.12]{GahNeu21}, using \ref{it:strong_2s_cpct_apriori}, \autoref{lem:propT} and the norm estimate on $\p_t(\calT_\eps v_\eps)$ from \autoref{cor:pt_T}.
	
	For the proof of \ref{it:K2}, by the triangle inequality, using $\Sigma_{\xi^\prime}\setminus (\Sigma_{\xi^\prime})_h \subset \Sigma \setminus \Sigma_h$, it suffices to show that for $h>0$ and $\xi\leq \xi_0(h)$ sufficiently small we have
	\begin{subequations}
		\begin{align}
			\sup_{\eps>0} \Lpnorm{\calT_\eps v_\eps(\cdot,\cdot+\xi^\prime,\cdot) - \calT_\eps v_\eps}{p}{\Sigma_{h}; L^2(0,T;H^\beta(Z_*))} &\xrightarrow{\xi^\prime \to 0}0,\label{eq:K2a}\\
			\sup_{\eps>0} \Lpnorm{\calT_\eps v_\eps}{p}{\Sigma\setminus \Sigma_h; L^2(0,T;H^\beta(Z_*))} &\xrightarrow{h \to 0}0\label{eq:K2b}.
		\end{align}
	\end{subequations}
	The convergence \eqref{eq:K2b} is immediate by H\"older's inequality using \ref{it:strong_2s_cpct_apriori} and $|\Sigma\setminus \Sigma_h|\lesssim |h|$. To establish \eqref{eq:K2a}, one may use \autoref{lem:shifts_Tv} together with \ref{it:strong_2s_cpct_shifts} to obtain for given $\rho>0$ and $|\xi^\prime|\leq \xi_0(h)$ and $\eps \leq \eps_0(h)$ sufficiently small 
	\begin{align}
		\sup_{\eps\in (0,\eps_0(h)]} \Lpnorm{\calT_\eps v_\eps(\cdot,\cdot+\xi^\prime,\cdot) - \calT_\eps v_\eps}{p}{\Sigma_{h}; L^2(0,T;H^\beta(Z_*))} \leq \frac \rho 2.
	\end{align}
	Finally, as $\frac 1\eps \in \IN$, there are only finitely many values of $\eps>\eps_0(h)$, and applying the standard Kolomogorov-argument to the finite family $(\calT_\eps v_\eps)_{\eps <\eps_0(h)}$, for $|\xi^\prime|$ possibly being even smaller, we have 
	\begin{align}
		\sup_{\eps>\eps_0(h)} \Lpnorm{\calT_\eps v_\eps(\cdot,\cdot+\xi^\prime,\cdot) - \calT_\eps v_\eps}{p}{\Sigma_{h}; L^2(0,T;H^\beta(Z_*))} \leq \frac \rho 2,
	\end{align}
    which concludes the proof of \eqref{eq:K2a}.
\end{proof}

{\onehalfspacing \printnomenclature\label{symbols}}

\end{document}